\documentclass[12pt,a4paper]{article}
\usepackage{amsmath,amscd,amssymb,amsthm,mathrsfs,amsfonts,layout,indentfirst,graphicx,caption,mathabx,  tikz, stmaryrd,appendix}
\usepackage{palatino}  
\usepackage{slashed} 
\usepackage{mathrsfs} 
\usepackage[nottoc]{tocbibind}   

\usepackage{lipsum}
\let\OLDthebibliography\thebibliography
\renewcommand\thebibliography[1]{
	\OLDthebibliography{#1}
	\setlength{\parskip}{0pt}
	\setlength{\itemsep}{2pt} 
}

\allowdisplaybreaks  
\usepackage{latexsym}
\usepackage{chngcntr}
\usepackage[colorlinks,linkcolor=blue,anchorcolor=blue,linktocpage]{hyperref}
\hypersetup{citecolor=[rgb]{0,0.5,0}}

\usepackage{fullpage}
\counterwithin{figure}{section}

\theoremstyle{definition}
\newtheorem{df}{Definition}[section]

\newtheorem{rem}[df]{Remark}

\theoremstyle{plain}
\newtheorem{thm}[df]{Theorem}
\newtheorem{thd}[df]{Theorem-Definition}
\newtheorem{pp}[df]{Proposition}
\newtheorem{co}[df]{Corollary}
\newtheorem{lm}[df]{Lemma}
\newtheorem{cond}{Condition}

\newcommand{\fk}{\mathfrak}
\newcommand{\mc}{\mathcal}
\newcommand{\wtd}{\widetilde}
\newcommand{\wht}{\widehat}
\newcommand{\wch}{\widecheck}
\newcommand{\ovl}{\overline}
\newcommand{\tr}{\mathrm{t}} 
\newcommand{\End}{\mathrm{End}} 
\newcommand{\id}{\mathrm{id}}
\newcommand{\Hom}{\mathrm{Hom}}
\newcommand{\Conf}{\mathrm{Conf}}

\newcommand{\ev}{\mathrm{ev}}

\newcommand{\Rep}{\mathrm{Rep}}
\newcommand{\diag}{\mathrm{diag}}
\newcommand{\Dom}{\scr D}
\newcommand{\loc}{\mathrm{loc}}

\newcommand{\uni}{\mathrm{u}}
\newcommand{\ssp}{\mathrm{ss}}

\newcommand{\Diffp}{\mathrm{Diff}^+}
\newcommand{\Diff}{\mathrm{Diff}}
\newcommand{\PSU}{\mathrm{PSU}}

\newcommand{\Witt}{\mathscr W}

\newcommand{\bk}[1]{\langle {#1}\rangle}
\newcommand{\GA}{\mathscr G_{\mathcal A}}
\newcommand{\vs}{\varsigma}
\newcommand{\Vect}{\mathrm{Vec}}
\newcommand{\Vectc}{\mathrm{Vec}^{\mathbb C}}
\newcommand{\scr}{\mathscr}
\newcommand{\sjs}{\subset\joinrel\subset}
\newcommand{\Jtd}{\widetilde{\mathcal J}}
\newcommand{\gk}{\mathfrak g}
\newcommand{\hk}{\mathfrak h}
\newcommand{\hr}{\mathfrak h_{\mathbb R}}
\newcommand{\Ad}{\mathrm{Ad}}
\newcommand{\DHR}{\mathrm{DHR}_{I_0}}
\newcommand{\Repi}{\mathrm{Rep}_{\wtd I_0}}
\newcommand{\mbb}{\mathbb}

\numberwithin{equation}{section}

\title{Categorical extensions of conformal nets}
\author{{\sc Bin Gui}
}
\date{}
\begin{document}\sloppy 
	\pagenumbering{arabic}
	\setcounter{section}{0}
	\setcounter{equation}{6}

	\maketitle

\tableofcontents

\newpage

\begin{abstract}
An important goal in studying the relations between unitary VOAs and conformal nets is to prove the equivalence of their ribbon categories. In this article, we prove this conjecture for many familiar examples. Our main idea is to construct new structures associated to conformal nets: the categorical extensions.


Let $V$ be a strongly-local unitary regular  VOA of CFT type, and assume that all $V$-modules are unitarizable. Then $V$ is associated with a conformal net $\mc A_V$ by \cite{CKLW18}. Let $\Rep^\uni(V)$  and $\Rep^\ssp(\mc A_V)$ be the braided tensor categories of  unitary $V$-modules and semisimple $\mc A_V$-modules respectively. We show that if one can find enough intertwining operators of $V$ satisfying the strong intertwining property and the strong braiding property, then any unitary $V$-module $W_i$ can be integrated to an $\mc A_V$-module $\mc H_i$, and the functor $\fk F:\Rep^\uni(V)\rightarrow\Rep^\ssp(\mc A_V),  W_i\mapsto \mc H_i$ induces an equivalence of the ribbon categories $\Rep^\uni(V)\xrightarrow\simeq \fk F(\Rep^\uni(V))$. This, in particular, shows that $\fk F(\Rep^\uni(V))$ is a modular tensor category.

We apply the above result to all unitary $c<1$ Virasoro VOAs (minimal models), many unitary affine VOAs (WZW models), and all even lattice VOAs. In the case of Virasoro VOAs and affine VOAs, one further knows that $\fk F(\Rep^\uni(V))=\Rep^\ssp(\mc A_V)$. So we've proved the equivalence of the unitary modular tensor categories $\Rep^\uni(V)\simeq\Rep^\ssp(\mc A_V)$. In the case of lattice VOAs, besides the equivalence of $\Rep^\uni(V)$ and $\fk F(\Rep^\uni(V))$, we also prove the strong locality of $V$ and the strong integrability of all (unitary) $V$-modules. This solves a conjecture in \cite{CKLW18}.




\end{abstract}

\section{Introduction}

\subsubsection*{Backgrounds}

A systematic study of the relations between vertex operator algebras (VOAs) and conformal nets, two major mathematical formulations  of chiral conformal field theories, was initiated by \cite{CKLW18}. A main point in \cite{CKLW18} is that, given a unitary CFT-type VOA $V$ satifying certain nice analytic properties (energy-bounds condition and strong locality), one defines $\mc A_V(I)$ to be the von Neumann algebra generated by all smeared vertex operators localized in the open interval $I\subset S^1$. Then the collection of all these $\mc A_V(I)$ form a conformal net $\mc A_V$. The energy-bounds condition is needed to show the preclosedness of smeared vertex operators and the weak commutativity (Wightmann-locality) of casually disjoint smeared vertex operators. To show that $\mc A_V$ satisfies the locality axiom of a conformal net, one  requires that causally disjoint smeared operators also commute strongly, in the sense that the von Neumann algebras generated by them commute. This is the meaning of strong locality. Energy bounds condition and strong locality are natural requirements on VOAs, which can be verified for many important examples. Indeed, it is believed that all unitary VOAs satisfy these two properties.

After building a bridge between VOAs and conformal nets, the next natural step is to understand the relations between their representations.  Let $V$ be a unitary (energy-bounded and) strongly local CFT-type VOA, and assume that all irreducible $V$-modules are unitarizable.\footnote{The unitarizability of all $V$-modules is known to be true for many well known rational CFT models. For some other examples, this problem is related to constructing a $C^*$-tensor structure on $\Rep^\uni(V)$. We will discuss this topic in future work.} Since our main interest is in rational CFTs, we assume that $V$ is regular \cite{DLM95}, so that there exists a modular tensor categorical structure on the category $\Rep^\uni(V)$ of unitary $V$-modules \cite{Hua08b}. We also have a braided $C^*$-tensor categorical structure on the representation\footnote{In this article, we assume that all conformal net modules are seperable and (hence) locally normal.} category $\Rep(\mc A_V)$ of $\mc A_V$ by Doplicher-Haag-Roberts (DHR) superselection theory \cite{DHR71,DHR74,FRS89,FRS92}. Now the whole project of relating the representation theories of ``rational" VOAs and conformal nets can be described by answering the following questions:
\begin{enumerate}
	\item Can we define a ``natural" fully faithful $*$-functor $\fk F:\Rep^\uni(V)\rightarrow\Rep(\mc A_V)$?
	\item Is $\fk F$ essentially surjective?
	\item Does $\fk F$ preserve the braided tensor categorical structures?
\end{enumerate}
Once these problems are solved, we can conclude that the category $\Rep^\ssp(\mc A_V)$ of semisimple $\mc A_V$-modules is a braided tensor subcategory of $\Rep(\mc A_V)$ admitting a ribbon fusion categorical structure, the modular tensor category $\Rep^\uni(V)$ admits a unitary (i.e, $C^*$-) structure, and $\fk F:\Rep^\uni(V)\rightarrow\Rep^\ssp(\mc A_V)$ is an equivalence of unitary modular tensor categories. As an important application, the Reshetikhin-Turaev 3d topological quantum field theory (cf. \cite{RT91,Tur94}) constructed from $\Rep^\uni(V)$ and from $\Rep^\ssp(\mc A_V)$ are the same.

Problem 1 is the main subject of \cite{CWX}. That paper shows that for many nice examples of $V$, any unitary $V$-module $(W_i,Y_i)$ is energy bounded, and can be ``integrated" to an $\mc A_V$-module $(\mc H_i,\pi_i)$, in the sense that $\pi_i(Y(v,f))=Y_i(v,f)$ for any smeared vertex operators $Y(v,f)$ and $Y_i(v,f)$. This condition is called \emph{strong integrability}.  One can thus define $\fk F(W_i)=\mc H_i$. By semisimpleness, any morphism $F$ of unitary $V$-modules is bounded. Thus  $\fk F(F)$ can be defined  to be the closure of $F$. Then \cite{CWX} shows that $\fk F$ is fully faithful. (See also \cite{Gui17b} chapter 4 for relevant results.)\footnote{Besides using smeared vertex operators, one can also use Segal CFTs and a geometric interpolation procedure to construct conformal nets from unitary VOAs, and to define the $*$-functor $\fk F$. See \cite{Ten16,Ten18} for more details.} A detailed study of problem 2 can be found in \cite{CW}. In the case of unitary affine VOAs, problem 2 was completely solved by \cite{Hen19}. For $c<1$ unitary Virasoro VOAs, problem 2 can be solved by combining the results of \cite{Xu00a} and \cite{KL04} (see section \ref{lb50}).

So far, the studies of problem 3 have been focusing mainly on comparing fusion rules. The following results are known: If $V$ is a type $A$ unitary affine VOA, or a $c<1$ unitary Virasoro VOA, then $\Rep^\uni(V)$ and $\Rep^\ssp(\mc A_V)$ have the same fusion rules \cite{Was98,Loke94}. When $V$ is of affine type $D$, the tensor subcategory $\mc C$ of unitary $V$-modules corresponding to the single-valued representations of  $SO(2n)$ has the same fusion rules as $\fk F(\mc C)$ \cite{TL04}. On the other hand, the  equivalence of the braided tensor categories is unknown except when $V$ is affine $\fk {sl}_2$, in which case the braided tensor categorical structures are determined by the fusion rules and the twist operators according to \cite{FK93} proposition 8.2.6.\footnote{This argument is due to  Marcel Bischoff. See \cite{Hen17} the paragraphs after conjecture 3.} Even for general affine $\fk{sl}_n$, proving the equivalence of the braided tensor categories has long been an open problem.

\subsubsection*{Categorical extensions of conformal nets}

One of the main goals in this paper is to give a systematic treatment of problem 3. We shall not only show the equivalence of fusion rules for more examples, but also provide a new perspective on conformal nets and VOAs from which the equivalence of the braided tensor categorical structures becomes quite natural: we shall define a new structure associated to conformal nets, called \emph{categorical extensions}. 

An ordinary extension of a conformal net $\mc A$ is just a conformal net  $\mc B$ containing $\mc A$ as a (finite-index) subnet. It is a fermionic extension when $\mc B$ is a super-conformal net, but it can also be anyonic in general. Full CFT and boundary CFT can also be regarded as extensions of conformal nets. A categorical extension $\scr E$ of $\mc A$, on the other hand, is a universal, free, categorical, and anyonic extension of $\mc A$.\footnote{Indeed,  both ``universal extensions" and ``anyonic conformal nets" were candidates for the name of this new structure.} By ``universal", we mean that $\scr E$ contains any sort of extensions of $\mc A$ as  sub-systems. Roughly speaking, $\scr E$ is defined to be the $*$-extension generated ``freely" by the intertwining operators of $\mc A$ (or its corresponding VOA). $\scr E$ is free of relations, but any extension of $\mc A$, which is a sub-system of $\scr E$, is described by a set of relations, i.e, by a \emph{Frobenius algebra}. As intertwining operators do not form an algebra  in general (except when the braidings are abelian \cite{DL93}), there seems to be no single Hilbert space $\mc H$ on which all intertwining operators could act freely. Therefore, we consider tensor categories of Hilbert spaces instead of single Hilbert spaces. As extensions of $\mc A$ are in general anyonic, $\scr E$ is anyonic.

Let us outline some key features of categorical extensions. Note that for a conformal net $\mc A$, given an open interval $I$, we have state-field correspondence between an operator $x\in\mc A(I)$ and a vector $x\Omega$. Vectors in $\mc H_0(I)=\mc A(I)\Omega$ are called $I$-bounded vectors. Then the actions of $\mc A(I)$ on $\mc H_0$ can be regarded as multiplications $\mc H_0(I)\otimes\mc H_0\rightarrow\mc H_0$. With over simplification, we regard the vacuum module $\mc H_0$  as both a vector space and  an algebra. Now for general $\mc A$-modules $\mc H_i,\mc H_j$, their multiplications are in neither $\mc H_i$ nor $\mc H_j$, but in a tensor (fusion) product $\mc H_i\boxdot\mc H_j$. More precisely, for any open interval $I$, we denote by $I^c$ the complement of its closure in $S^1$, and set $\mc H_i(I)$ to be the subspace of all $\xi\in\mc H_i$ satisfying that the linear map defined by $x\Omega\in\mc A(I^c)\Omega\mapsto x\xi\in\mc H_i$ is bounded. We call such $\xi$ an $I$-bounded vector. We then have a multiplication $\mc H_i(I)\otimes\mc H_j\rightarrow\mc H_i\boxdot\mc H_j,\xi\otimes\eta\mapsto\xi\cdot\eta$. Let $L(\xi,I)$ denote this left action of $\xi$ on $\mc H_j$. Then we require that $L(\xi,I)$ is a bounded operator intertwining the actions of $\mc A(I^c)$, i.e., $L(\xi,I)\in\Hom_{\mc A_V(I^c)}(\mc H_j,\mc H_i\boxdot\mc H_j)$. 

The above formulation is reminiscent of Connes fusion products (Connes relative tensor products) \cite{Con80}. Indeed, Connes fusion
is a major way to construct categorical extensions, in which case the tensor product $\mc H_i\boxdot\mc H_j$ is just the Connes fusion product $\mc H_i\boxtimes\mc H_j$, and the multiplication is the natural one. On the other hand, the standard Connes fusion theory for bimodules tells us nothing about how  the fusion products over different intervals could be related. If we want to consider a net of left actions $\{L(\cdot,I)\}$, we need to take into account the monodromy behaviors of them. So $L(\xi,I)$ should depend not only on $I$, but also on a preferred branch of $I$ in the universal covering space of $S^1$. Equivalently, we should equip $I$ with a (continuous) argument function $\arg_I$ on $I$, set $\wtd I=(I,\arg_I)$, and write $L(\xi,I)$ as $L(\xi,\wtd I)$ instead. Similarly, for any arg-valued interval $\wtd J=(J,\arg _J)$ and $\eta\in\mc H_j(J)$, we also have a right action $R(\eta,\wtd J)\in\Hom_{\mc A(J^c)}(\mc H_i,\mc H_i\boxdot\mc H_j)$. 

Locality is the most important axiom of categorical extensions, which we now state. Suppose that $I$ and $J$ are disjoint, and the arg function $\arg_I$ of $I$ is chosen to be anticlockwise to $\arg_J$, in the sense that $\arg_J(\zeta)<\arg_I(z)<\arg_J(\zeta)+2\pi$ for any $z\in I,\zeta\in J$. In this case we say that $\wtd I=(I,\arg_I)$ is anticlockwise to $\wtd J=(J,\arg_J)$. Now the locality axiom says that  for any $\mc A$-modules $\mc H_i,\mc H_j,\mc H_k$,  any arg-valued intervals $\wtd I,\wtd J$ with $\wtd I$ anticlockwise to $\wtd J$, and any $\xi\in\mc H_i(I),\eta\in\mc H_j(J)$, the  diagram   
\begin{align*}
\begin{CD}
\mc H_k @> \quad R(\eta,\wtd J)\quad   >> \mc H_k\boxdot\mc H_j\\
@V L(\xi,\wtd I)   V  V @V L(\xi,\wtd I) VV\\
\mc H_i\boxdot\mc H_k @> \quad R(\eta,\wtd J) \quad  >> \mc H_i\boxdot\mc H_k\boxdot\mc H_j
\end{CD}
\end{align*}
commutes \emph{adjointly}, in the sense that the following diagram also commutes.
\begin{align*}
\begin{CD}
\mc H_k @> \quad R(\eta,\wtd J)\quad   >> \mc H_k\boxdot\mc H_j\\
@A L(\xi,\wtd I)^*   AA @A L(\xi,\wtd I)^* AA\\
\mc H_i\boxdot\mc H_k @> \quad R(\eta,\wtd J) \quad  >> \mc H_i\boxdot\mc H_k\boxdot\mc H_j
\end{CD}
\end{align*}
Due to locality, the $C^*$-tensor categorical structure defined by $\boxdot$ is remembered by the categorical extension, and is naturally equivalent to the one defined by Connes fusion $\boxtimes$. Moreover, if the left and right actions are related by a braiding $\ss$, in the sense that  there always exists a functorial isomorphism $\ss_{i,j}:\mc H_i\boxdot\mc H_j\rightarrow\mc H_j\boxdot\mc H_i$ such that $R(\xi,\wtd I)\eta=\ss_{i,j}L(\xi,\wtd I)\eta$ for any arg-valued $I$, $\xi\in\mc H_i(I)$, and $\eta\in\mc H_j$ (the braiding axiom), then the braid structure is also remembered. Therefore, once we have shown that the braided tensor category $\Rep^\uni(V)$ is unitary (i.e., a braided $C^*$-tensor category),\footnote{Though solving problem 3 will prove the unitarity of $\Rep^\uni(V)$, in our theory we have to first prove the unitarity in order to construct vertex categorical extensions and show the equivalence of the braided tensor categories. This is one of our main motivations for studying the unitarity of $\Rep^\uni(V)$ in  \cite{Gui17a,Gui17b}.} and construct a categorical extension $\scr E$ of $\mc A_V$ using the intertwining operators of $V$  (the \emph{vertex categorical extension}), then $\Rep^\uni(V)$ will be automatically equivalent to a braided tensor subcategory of $\Rep(\mc A_V)$ under the $*$-functor $\fk F$. 

\subsubsection*{The strong intertwining and braiding properties}

To construct a vertex categorical extension, locality is also the most difficult to verify.\footnote{Our situation is similar to that of \cite{CKLW18}.} Our previous works \cite{Gui17a,Gui17b} show that the unitarity of the braided tensor category $\Rep^\uni(V)$ follows from the strong locality of $V$ and the \emph{strong intertwining property} for the  intertwining operators of $V$ (see remark \ref{lb45}).\footnote{In \cite{Gui17b} chapter 4 we (essentially) showed that the strong integrability of $V$ follows also from these two properties, hence providing an  answer to problem 1 alternative to the work of \cite{CWX}. See also theorem \ref{lb31}.} The strong intertwining property says that if $\mc Y_\alpha$ is a type ${k\choose i~j}={W_k\choose W_i W_j}$ intertwining operator, then for any homogeneous $v\in V,w^{(i)}\in W_i$, disjoint intervals $I,J$ with $I$ arg-valued, and smooth functions $\wtd f,g$ supported in $\wtd I,J$ respectively, the smeared intertwining operator $\mc Y_\alpha(w^{(i)},\wtd f)$ commutes strongly with the smeared vertex operator $Y_{j\oplus k}(v,g)$.\footnote{The strong intertwining property for intertwining operators is parallel to the notion of localized intertwining operators in \cite{Ten18}.} (See definition \ref{lb51} for more details.) Unfortunately, these two properties are not enough to verify the locality axiom of categorical extensions. One also requires that there exist enough \footnote{The meaning of ``enough" will be given in theorem \ref{lb36}.} intertwining operators satisfying the \emph{strong braiding property}, whose meaning is explained below.

Choose unitary $V$-modules $W_i,W_j$. For any $W_k\in\Rep^\uni(V)$, we have a distinguished intertwining operator $\mc L_i$ of type ${ik\choose i~j}={W_i\boxtimes W_k\choose W_i~W_k}$, such that any intertwining operator of type $l\choose i~j$ (where $W_l\in\Rep^\uni(V)$) factors through $\mc L_i$. $\mc L_i$ may act on different $W_k$ to denote intertwining operators of different types.  The type $jk\choose j~k$ intertwining operator $\mc L_j$ is defined in a similar way. Now we define a type $kj\choose j~k$ intertwining operator $\mc R_j$ acting on each $k$ to be $\mc R_j=\ss_{j,k}\mc L_j$, where $\ss$ denotes the braiding of $V$-modules. Assume that $\mc L_i$ and $\mc R_i$ are energy-bounded. Then one can show (see theorem \ref{lb38}) that  for any homogeneous $w^{(i)}\in W_i,w^{(j)}\in W_j$, arg-valued $\wtd I$ (disjoint and) anticlockwise to $\wtd J$, and smooth functions $\wtd f,\wtd g$ supported in $\wtd I,\wtd J$ respectively, the following two diagrams commute in the sense of braiding of smeared intertwining operators.
\begin{align*}
\begin{CD}
\mc H^\infty_k @> \quad\mc R_j(w^{(j)},\wtd g)\quad >> \mc H^\infty_{kj}\\
@V \mc L_i(w^{(i)},\wtd f) VV @V \mc L_i(w^{(i)},\wtd f) VV\\
\mc H^\infty_{ik} @> \quad\mc R_j(w^{(j)},\wtd g)\quad>> \mc H^\infty_{ikj}
\end{CD}\qquad\qquad\qquad\qquad
\begin{CD}
\mc H^\infty_k @> \quad\mc R_j(w^{(j)},\wtd g)\quad >> \mc H^\infty_{kj}\\
@A \mc L_i(w^{(i)},\wtd f)^\dagger AA @A \mc L_i(w^{(i)},\wtd f)^\dagger AA\\
\mc H^\infty_{ik} @> \quad\mc R_j(w^{(j)},\wtd g)\quad>> \mc H^\infty_{ikj}
\end{CD}
\end{align*}
(Here, for example, $\mc H_{ik}=\fk F(W_{ik})$ is the $\mc A_V$-module integrated from $W_{ik}=W_i\boxtimes W_k$, and $\mc H_{ik}^\infty$ is its subspace of smooth vectors.) One can roughly say that the preclosed operators $\mc L_i(w^{(i)},\wtd f)$ and $\mc R_i(w^{(j)},\wtd g)$ commute adjointly. Now, we say that the actions $w^{(i)},w^{(j)}\curvearrowright\Rep^\uni(V)$ satisfies the \emph{strong braiding property}, if for any $W_k\in\Rep^\uni(V)$ and $\wtd I,\wtd J,\wtd f,\wtd g$ as above, the preclosed operators $\mc L_i(w^{(i)},\wtd f)$ and $\mc R_i(w^{(j)},\wtd g)$ commute strongly, in the sense that the von Neumann algebras generated by (the closures of) them commute.

Thus strong braiding is the crucial condition for $\Rep^\uni(V)$ and $\fk F(\Rep^\uni(V))$ to have the same braided $C^*$-tensor categorical structure, just as strong locality is crucial for constructing conformal nets from VOAs, and the strong intertwining property is required to construct conformal net modules from VOA modules, and to show the unitarity of $\Rep^\uni(V)$. Indeed,  these three properties should be treated as a whole: together they guarantee the existence of vertex categorical extensions. We strongly believe that constructing vertex categorical extensions is a more fundamental question than proving the equivalence of the modular tensor categories, as the latter only reflect the topological data of CFTs, while categorical extensions contain both analytic and topological data. We summarize our philosophy: \emph{categorical extensions of conformal nets are analytic enrichments of braided $C^*$-tensor categories.}

\subsubsection*{Analytic properties for VOA extensions}

Another motivation for studying categorical extensions is to understand the relations between various types of ``rational"   VOA extensions  and  conformal net extensions (including full and boundary CFTs), as well as the relations between their tensor categories. A general theory on this topic will be left to future works. In this paper, we use even lattice VOAs as examples to demonstrate that categorical extensions are powerful tools for studying functional analytic properties of VOA extensions.

We first explain why strong locality is not easy to prove for lattice VOAs (and for many other VOAs). The starting point of proving the strong locality of an energy-bounded unitary VOA $V$ is the $1$-st order energy bounds (linear energy bounds) condition. If $f$ and $g$ are supported in disjoint open intervals, and one of $Y(u,x)$ and $Y(v,x)$ satisfies $1$-order energy bounds, then using results from \cite{TL99} (see also lemma \ref{lb39}), we know that $Y(u,f)$ and $Y(v,g)$ commute strongly. Unfortunately (or fortunately?), $1$-st order energy bounds are not necessary conditions for strong commutativity. \cite{CKLW18} theorem 8.1 tells us that if $V$ is generated (in the vertex-algebraic sense) by a set of quasi-primary vectors, among which the strong commutativity of causally disjoint smeared vertex operators holds, then $V$ is strong local. So for instance, if $V$ is a unitary affine VOA, then $V$ is generated by quasi-primary vertex operators satisfying $1$-st order energy bounds. Therefore $V$ is strongly local. But we can easily choose $u,v\in V$ whose vertex operators do not satisfy $1$-st order energy bounds. 

The above example suggests a useful way to prove the strong locality of a VOA $V$ which is not necessarily generated by vertex operators satisfying $1$-st order energy bounds. Suppose that we can embed $V$ into a larger unitary VOA $\fk V$ (conformal embedding is not necessarily required), and if $\fk V$ is generated by quasi-primary vertex operators satisfying $1$-st order energy bounds, then $\fk V$ is strongly local. This proves the strong locality of $V$. Indeed,  all  examples in \cite{CKLW18} (see chapter 8) were proved in this way.

Now the issue for a lattice VOA $V$ is the lack of such a larger VOA $\fk V$ containing $V$. Nor is the situation much better if we allow $\fk V$ to be a super VOA. In order to contain $V$, $\fk V$ has to be a highly anyonic vertex algebra, say, a generalized vertex algebra in the sense of \cite{DL93}. However, the problem with this approach is the difficulty of generalization to non-abelian intertwining operators. Therefore, to take general cases into consideration, one has to study categorical vertex algebras, whose corresponding categorical conformal nets are the  categorical extensions of ordinary (bosonic) conformal nets.

Let us explain the idea of the proof in more details. Let $U$ be a conformal unitary sub-VOA of $V$.\footnote{In principle $U$ is required to be regular, but we also allow $U$ to be a Heisenberg VOA.} Then the categorical vertex algebra $\fk V$ for $U$-intertwining operators contains $V$. Similar to \cite{CKLW18} theorem 8.1, one can show that if $\fk V$ is generated by $U$-intertwining operators satisfying $1$-st order energy bounds (and hence satisfying the strong braiding property), then all fields of $\fk V$, including those of $V$, satisfy the strong braiding property. (See theorems \ref{lb43} and \ref{lb40}.) This proves the strong locality of $V$. In the case that $V$ is an even lattice VOA, this method works by choosing  $U$ to be the corresponding Heisenberg sub-VOA.

\subsubsection*{Outline of the paper}

In chapter 2 we present a new approach to  Connes fusions of conformal net modules. The idea of using Connes fusion products to construct (braided) $C^*$-tensor categories for conformal nets is not new (see \cite{Was98,BDH15,BDH17}). Our approach differs from \cite{Was98} by emphasizing the global aspects of Connes fusions. On the other hand, unlike $\cite{BDH15,BDH17}$, many of our results  do not require conformal covariance. Thus they can be easily applied to M\"obius covariant nets. We also avoid the technical assumption of strong additivity.

In section 2.1 we review some of the basic facts about conformal nets and their representations. In section 2.2 we define the notion path continuations, which plays a  centrally important role in our theory. As we will see, the braid operator is a special path continuation. In section 2.3 we use path continuations to define the action of a conformal net $\mc A$ on the Connes fusion $\mc H_i\boxtimes\mc H_j$ of $\mc A$-modules $\mc H_i$ and $\mc H_j$. In section 2.4 we describe the conformal structure of $\mc H_i\boxtimes\mc H_j$ in terms of those of $\mc H_i$ and $\mc H_j$. Connes fusions of three (or more) representations are discussed in section 2.5. In section 2.6, we define the $C^*$-tensor categorical structure on $\Rep(\mc A)$ using our theory of Connes fusions. We will also define braiding in this section, which will be shown (proposition \ref{lb52}) to be the same as the one defined in \cite{Was98} section 33. However, a direct verification of the Hexagon axioms could be very complicated. We prove the Hexagon axioms in chapter 3 after categorical extensions are introduced.

Categorical extensions of conformal nets are defined in section 3.1. In section 3.2, we use Connes fusions to construct categorical extensions (called Connes categorical extensions). Then, in section 3.3, we use this machinery to prove the Hexagon axioms for $\Rep(\mc A)$. The next two sections are devoted to the uniqueness of categorical extensions. In section 3.4, we show that if $\scr E$ is a categorical extension of $\mc A$ over a braided $C^*$-tensor category $\scr C$, where $\scr C$ is also a full subcategory of $\Rep(\mc A)$, then $\scr C$ is equivalent to the corresponding braided $C^*$-tensor category defined by Connes fusions. In section 3.5, we show that $\scr E$ can be extended to a unique maximal categorical extension $\ovl{\scr E}$ defined also over $\scr C$. This maximal categorical extension $\ovl{\scr E}$ is naturally equivalent to a Connes categorical extension. We say that $\ovl{\scr E}$ is the closure of $\scr E$. The relation between $\ovl{\scr E}$ and $\scr E$ is similar to that between a von Neumann algebra $\mc M$ and a  subset $E\subset \mc M$ which densely spans $\mc M$. However, in applications one quite often starts with a subset $E$ which $*$-algebraically (but not just linearly) generates a dense subspace of $\mc M$. The situation is similar in the construction of vertex categorical extensions (as we will see in chapter 4): if a categorical extension $\scr E$ over $\scr C$ is regarded as a $\scr C$-$\scr C$ bimodule ${}_{\scr C}\scr C_{\scr C}$, then, more often, one begins with an $\mc F$-$\mc F$ bimodule ${}_{\mc F}\scr C_{\mc F}$, where $\mc F$ is a set of objects in $\scr C$ which tensor-generates $\scr C$. Then one can use ${}_{\mc F}\scr C_{\mc F}$ to generate ${}_{\scr C}\scr C_{\scr C}$. Such ${}_{\mc F}\scr C_{\mc F}$ is called a categorical local extension. In section 3.6 we show that a categorical local extension $\scr E^\loc$ generates a categorical extension $\scr E$. Moreover, we show that if $A$ (resp. $B$) commutes with the right (resp. left) action of $\mc F$ on $\scr C$, (In this case $A$ (resp. $B$) is called a left (resp. right) operator of $\scr E^\loc$.)  then $A$ and $B$ commute adjointly (see theorem \ref{lb43}). This  theorem is crucial for proving the strong braiding property of certain intertwining operators not satisfying $1$-st order energy bounds.

The goal of chapter 4 is to construct vertex categorical extensions using smeared intertwining operators. Most of the material in sections 4.1 and 4.3-4.5 is not new. In section 4.1 we review  Huang-Lepowsky's construction of ribbon categories for VOA-modules. Unitary structures on these tensor categories, which were introduced in \cite{Gui17a,Gui17b}, are reviewed in section 4.3. In section 4.4 we review the energy bounds conditions and smeared intertwining operators. Constructions of conformal nets and their representations from VOAs their modules are discussed in section 4.5. What's new in this chapter is the construction of the intertwining operators $\mc L_i,\mc R_i$ (for any VOA module $W_i$), which are closely related to the left and right actions $L,R$ in categorical extensions. The adjoint commutativity of $\mc L$ and $\mc R$ (in the sense of braiding) is proved in sections 4.2 and 4.3. These braid relations are crucial for verifying the locality axiom of categorical extensions. In section 4.4 we prove the adjoint commutativity of the smeared $\mc L$ and $\mc R$. Finally, in section 4.6 we use these smeared intertwining operators to construct vertex categorical extensions. 

Applications to various examples are given in chapter 5. In section 5.1, we show that if $V$ is a $c<1$ unitary Virasoro VOA, or a unitary affine VOA of type $A,C,G_2$, then problem 3 is completely solved: $\Rep^\ssp(\mc A_V)$ is equivalent to $\Rep^\uni(V)$ as unitary modular tensor categories. If $V$ is an affine VOA of type $B$ or $D$, then a partial result exists: Let $\mc C$ be the monoidal subcategory of $\Rep^\uni(V)$ tensor-generated by the smallest non-vacuum irreducible $V$-module. Then $\mc C$ is equivalent to $\fk F(\mc C)$ as unitary ribbon fusion categories. (The braided tensor categorical structure on $\fk F(\mc C)$ is defined using Connes fusions.) In section 5.2, we prove the equivalence of the ribbon fusion categories $\Rep^\uni(V)$ and $\fk F(\Rep^\uni(V))$ when $V$ is a unitary Heisenberg VOA. (In this case $\Rep^\uni(V)$ is defined to be the tensor category of semisimple unitary $V$-modules.) More importantly, we prove  the strong intertwining and braiding properties for all intertwining operators of unitary Heisenberg VOAs. This result is used in section 5.3 to prove the strong intertwining and braiding properties for all intertwining operators of an even lattice VOA $V$. The strong localilty of $V$, the strong integrability of all $V$-modules, and the equivalence of the unitary modular tensor categories $\Rep^\uni(V)\simeq\fk F(\Rep^\uni(V))$ thus follow. 

In the literature of conformal nets, the braided tensor categories are more often defined using DHR superselection theory. It is well known (at least when the conformal nets are strongly additive) that Connes fusions and  DHR theory define the same monoidal structures. However, it is not clear why these two theories give the same braidings.  In chapter 6, we clarify the relation between these two theories, and show that the braided $C^*$-tensor categories defined by them are equivalent.\\

\noindent
\emph{Note.} When $V$ is a unitary affine VOA of type $A$, the equivalence of $\Rep^\uni(V)$ and $\Rep^\ssp(\mc A_V)$ was also proved in a  recent work \cite{CCP} using completely different methods. For  affine VOAs and lattice VOAs, \cite{CCP}  proved the unitarity of $\Rep^\uni(V)$ using methods different from those in \cite{Gui17a,Gui17b}.

\subsubsection*{Acknowledgment}
The author would like to express his gratitude to Vaughan Jones for the constant support and encouragement during his research. The author would also like to thank Sebastiano Carpi and Andre Henriques for many helpful  discussions.

\section{Connes fusion products}

\subsection{Conformal nets and their representations}\label{lb57}
We first recall some basic facts about $\Diffp(S^1)$,  the group of orientation-preserving diffeomorphisms of $S^1$. Convergence in $\Diffp(S^1)$ means uniform convergence of all derivatives.  $\Diffp(S^1)$ contains the subgroup $\PSU(1,1)$ of M\"obius transformations of $S^1$. For any Hilbert space $\mathcal H$, we let $\mathcal U(\mathcal H)$ be the group of unitary operators on $\mathcal H$, equipped with the strong operator topology. We let $P\mathcal U(\mathcal H)$ be the projective group of $\mathcal U(\mathcal H)$. Then a projective representation of $\Diffp(S^1)$ on $\mathcal H$ is a continuous homomorphism $U:\Diffp(S^1)\rightarrow P\mathcal U(\mathcal H)$.

Let $\mathcal J$ be the set of open intervals in $S^1$, i.e., the set of non-empty non-dense connected open subset of $S^1$. If $I\in\mathcal J$, we let $I^c$ be the interior of the complement of $I$ in $S^1$, which is again an open interval. For any $I$, we let $\Diff_I(S^1)$ be the subgroup of all $g\in\Diffp(S^1)$ which fixes points in the closure of $I^c$.

Let $\Vect(S^1)=C^\infty(S^1,\mathbb R)$ be the Lie algebra of real vector fields on $S^1$, where, for any $X,Y\in\Vect(S^1)$, $[X,Y]$ is the negative of the usual bracket for vector fields. Then $\Vect(S^1)$ is the Lie algebra of $\Diffp(S^1)$. We let $\Vectc(S^1)$ be the complexification of $\Vect(S^1)$. For each $n\in\mathbb Z$, we the $L_n$ be the complex vector field on $S^1$ defined by
\begin{align*}
L_n(e^{i\theta})=-ie^{in\theta}\frac d{d\theta}.
\end{align*}
Then these $L_n$ form the Witt algebra $\Witt$, which is a dense Lie subalgebra of $\Vectc(S^1)$. We define a $*$ structure on $\Witt$ by setting $L_n^*=L_{-n}$. An element $X=\sum_{n} a_nL_n$ in $\Vectc(S^1)$ is self-adjoint (i.e., fixed by $*$), if and only if $\overline{a_n}=a_{-n}$ for all $a$, if and only if $iX\in\Vect(S^1)$. For such $X$, we can therefore consider the one parameter group $\exp_{iX}:t\in\mathbb R\mapsto \exp(itX)$ in $\Diffp(S^1)$. In particular, $\exp_{iL_0}$ is the rotation subgroup. In general, for any $X\in\Vect(S^1)$, we let $\exp_X$ be the one parameter subgroup of $\Diffp(S^1)$ generated by $X$.

It will be convenient to consider another type of localized diffeomorphism groups. We let $\Diff^0_I(S^1)$ be the subgroup of $\Diffp(S^1)$ (algebraically) generated by $\exp(\Vect_I(S^1))$, where $\Vect_I(S^1)$ is the subspace of  vector fields supported in $I$. Then by the proof of \cite{Loke94} proposition V.2.1, for any $J\sjs I$ (i.e., $J\in\mathcal J$, and  $\overline J\subset I$) we have
\begin{align}
\Diff^0_J(S^1)\subset\Diff_I(S^1).\label{eq92}
\end{align}
So $\Diff^0_I(S^1)$ is large enough.\\

A \textbf{conformal net} $\mathcal A$ associates to each $I\in\mathcal J$ a von Neumann algebra $\mathcal A(I)$ acting on a fixed separable Hilbert space $\mathcal H_0$, such that the following conditions hold:\\
(a) (Isotony) If $I_1\subset I_2\in\mathcal J$, then $\mathcal A(I_1)$ is a von Neumann subalgebra of $\mathcal A(I_2)$.\\
(b) (Locality) If $I_1,I_2\in\mathcal J$ are disjoint, then $\mathcal A(I_1)$ and $\mathcal A(I_2)$ commute.\\
(c) (Conformal covariance) We have a strongly continuous projective unitary representation $U$ of $\Diffp(S^1)$ on $\mathcal H_0$, such that for any $g\in \Diffp(S^1),I\in\mathcal J$, and any representing element $V\in\mathcal U(\mathcal H_0)$ of $U(g)$,
\begin{align*}
V\mathcal A(I)V^*=\mathcal A(gI).
\end{align*}
Moreover, if $g\in\Diff_I(S^1)$ and $x\in\mathcal A(I^c)$, then
\begin{align*}
VxV^*=x.
\end{align*}
(d) (Positivity of energy) The generator of the restriction of $U$ to $S^1$ is positive.\\
(e) There exists a  unique (up to scalar) unit vector $\Omega\in\mathcal H_0$ (the vacuum vector), such that $U(g)\Omega\in\mathbb C\Omega$ for any $g\in\PSU(1,1)$. Moreover, $\Omega$ is  cyclic under the action of $\bigvee_{I\in\mathcal J}\mathcal M(I)$ (the von Neumann algebra generated by all $\mathcal M(I)$).

Note that by the up to phase invariance of $\Omega$ under the projective action of $\PSU(1,1)$, one may fix an actual representation of $\PSU(1,1)$ on $\mc H_0$ such that $g\Omega=\Omega$ for any $g\in\PSU(1,1)$. It is also well known that a conformal net $\mathcal A$ satisfies the following properties (cf. for example, \cite{GL96} and the reference therein):\\
(1) (Additivity) $\mathcal A(I)=\bigvee_\alpha\mathcal A(I_\alpha)$ if $\{I_\alpha\}$ is a set of open intervals whose union is $I$.\\
(2) (Haag duality) $\mathcal A(I)'=\mathcal A(I^c)$. As a consequence, any representation element $V$ of $U(g)$ is in $\mathcal A(I)$ if $g\in\Diff_I(S^1)$.\\ 
(3) (Reeh-Schlieder theorem) $\mathcal A(I)\Omega$ is dense in $\mathcal H_0$ for any $I\in\mathcal J$\\
(4) For each $I\in\mathcal J$, $\mathcal A(I)$ is a type III$_1$ factor.

Let $\mathcal H_i$ be a separable Hilbert space. We say that $(\mathcal H_i,\pi_i)$ (or simply $\mathcal H_i$) is  a  \textbf{represention} of the  $\mathcal A$ (or a \textbf{$\mathcal A$-module}), if for any $I\in\mathcal J$, we have a normal unital *-representation $\pi_{i,I}:\mathcal A(I)\rightarrow B(\mathcal H_i)$, such that for any $I_1,I_2\in\mathcal J$ satisfying $I_1\subset I_2$, and any $x\in\mathcal A(I_1)$, we have $\pi_{i,I_1}(x)=\pi_{i,I_2}(x)$, which will be written as $\pi_i(x)$ when no confusion arises. Given a vector $\xi^{(i)}\in\mathcal H_i$, we often write $\pi_i(x)\xi^{(i)}$
as $x\xi^{(i)}$. Note that $\mathcal H_0$ itself is an $\mathcal A$-module, called the \textbf{vacuum module}.\\

Next, we discuss conformal structures on $\mc A$-modules. Let $\mathscr G=\widetilde{\Diffp(S^1)}$ be the simply connected covering group of $\Diffp(S^1)$, and consider the projective representation  $\mathscr G\curvearrowright \mathcal H_0$ lifted from $U:\Diffp(S^1)\curvearrowright\mathcal H_0$. This projective representation is also denoted by $U$. Define a  topological group
\begin{align}
\GA=\{(g,V)\in\mathscr G\times \mathcal U(\mathcal H_0)| V \textrm{ is a representing element of } U(g) \},\label{eq98}
\end{align}
called the \textbf{central extension of $\mathscr G$ associated to $\mathcal A$}. The topology of $\scr G_{\mc A}$ inherits from those of $\scr G$ and $\mc U(\mc H_0)$. Then we have a representation $U$ of $\GA$ on $\mathcal H_0$ defined by $U(g,V)=V$. We have an exact sequence
\begin{align*}
1\rightarrow U(1)\rightarrow \GA\rightarrow\scr G\rightarrow 1
\end{align*}
where $U(1)=\{ (1,V)\in\mathscr G\times \mathcal U(\mathcal H_0)| V \textrm{ is a representing element of } U(1)\}$. Clearly $U(1)$ is acting as scalars on $\mc H_0$. If $\mathcal H$ is a Hilbert space, then  a \textbf{(unitary) representation} of $\GA$ on $\mathcal H$ is, by definition, a homomorphism $\GA\rightarrow\mc U(\mathcal H)$ which restricts to a  homomorphism $\mathscr G\rightarrow P\mathcal U(\mathcal H)$ (i.e., a projective representation of $\mathscr G$). The standard action $\GA\curvearrowright\mc H_0$ is a unitary representation which is clearly continuous. 

\begin{rem}
The (equivalence class of the) central extension $\scr G_{\mc A}$  depends only on the central charge $c$ of the conformal net $\mc A$. In fact, by \cite{TL99} proposition 5.3.1, the Lie algebra of $\scr G_A$ is equivalent to the Virasoro algebra. The universal cover of $\GA$ can be identified with the group $\Diff_+^{\mbb R\times\mbb Z}(S^1)$ considered in \cite{Hen19}. Moreover, $\GA$ can indeed be recovered from the central extension $1\rightarrow\mbb R\rightarrow\Diff_+^{\mbb R\times\mbb Z}(S^1)\rightarrow\scr G\rightarrow 1$ by taking the quotient of $\Diff_+^{\mbb R\times\mbb Z}(S^1)$ by the central subgroup $c\mbb Z$ of $\mbb R$. We will not use these facts in the present article, and content ourselves with the explicit construction \eqref{eq98} of $\GA$.
\end{rem}

It is important to consider local diffeomorphism subgroups of $\mathscr G$ and $\GA$. For each $X\in\Vect(S^1)$, we define  $\wtd\exp_X:\mathbb R\rightarrow\mathscr G$ to be the one parameter subgroup of $\mathscr G$ lifted from $\exp_X$. We then set $\wtd\exp(X)=\wtd\exp_X(1)$. Define  $\mathscr G^0(I)$ to be the  subgroup of $\mathscr G$ algebraically generated by $\wtd\exp(\Vect_I(S^1))$, which can be identified with the connected branch of the inverse image of $\Diff_I^0(S^1)$ in $\scr G$ containing the identity. Similarly,  we let $\scr G(I)$ be the branch of the inverse image of $\Diff_I(S^1)$ in $\scr G$ containing the identity. Since $\Diff_I(S^1)$ is contractible, $\scr G(I)$ is homeomorphic to $\Diff_I(S^1)$ under the covering map $\scr G\rightarrow\Diffp(S^1)$.  From \eqref{eq92}, we know that if $J\sjs I$ then $\scr G(J)\subset\scr G^0(I)\subset\scr G(I)$.
Finally, we let  $\GA^0(I)$ and $\GA(I)$ be the respectively the inverse images of $\scr G^0(I)$ and $\scr G(I)$  in $\GA$. Then, we also have
\begin{align}
\GA(J)\subset\GA^0(I)\subset\GA(I)\label{eq94}
\end{align}
when $J\sjs I$. In this article, we will be mainly interested in $\GA(I)$ instead of $\GA^0(I)$. The only place we use $\GA^0(I)$ is in the proof of proposition \ref{lb34}.

Note that by conformal covariance of $\mathcal A$, $U(g)\in\mathcal A(I)$ for any $g\in\GA(I)$.

\begin{thm}\label{lb59}
Any representation $\mathcal H_i$ of $\mathcal A$ is  \textbf{conformal (covariant)} in the sense that there is a unique (unitary) representation $U_i$ of $\GA$ on $\mathcal H_i$ satisfying for any $I\in\mathcal J,g\in\GA(I)$ that
\begin{align}
U_i(g)=\pi_{i,I}(U(g)).\label{eq15}
\end{align}
Moreover, this representation is continuous.
\end{thm}

\begin{proof}
In \cite{Hen19},   it was shown in the proof of theorem 12 that the collection of inclusions $\{\scr G(I)\subset \scr G \}_{I\in\mc J}$ (which by theorem 11 is equivalent to $\{\scr G(I)\subset \mathrm{colim}_{J\in\mc J}\scr G(J) \}_{I\in\mc J}$) satisfies the assumption in the first paragraph of proposition 2 of that article.\footnote{Our $\scr G(I)$ is the same as $\Diff_0(I)$ in \cite{Hen19}.} Thus, by the second paragraph of that proposition, the canonical map
\begin{align}
\mathrm{colim}_{I\in\mc J}\GA(I)\rightarrow\GA\label{eq93}
\end{align}
induced by inclusions is an isomorphism of topological groups.  Thus, the collection of continuous representations $\{\pi_{i,I}\circ U:\GA(I)\curvearrowright\mc H_i\}_{I\in\mc J}$  gives rise to a continuous representation $\GA(I)\curvearrowright\mc H_i$ satisfying \eqref{eq15}. As \eqref{eq93} is surjective, the representation satisfying \eqref{eq15} is unique. 
\end{proof}

\begin{rem}
Our notion of a conformal covariant representation $\mc H_i$ is stronger than the usual one in the literature, which requires that $\mc H_i$ admits a projective representation of $\Diffp(S^1)$ on $\mc H_i$ satisfying \eqref{eq16}, and that the generator $L_0$ of the rotation subgroup is positive when acting on $\mc H_i$. Indeed, the positivity of $L_0$ is redundant by \cite{Wei06}; condition \eqref{eq16} follows from \eqref{eq15} by corollary \ref{lb7}.
\end{rem}


We rephrase the surjectivity of  \eqref{eq93} as follows, which also follows directly from \cite{Hen19} lemma 17-(ii).

\begin{pp}\label{lb6}
$\GA$ is (algebraically) generated by $\{\GA(I)\}_{I\in\mathcal J}$.  
\end{pp}
\begin{rem}\label{lb58}
By \eqref{eq94} and the above proposition, $\GA$ is  generated by $\{\GA^0(I)\}_{I\in\mathcal J}$. Thus, similarly, $\scr G$ is also generated by $\{\scr G^0(I)\}_{I\in\mathcal J}$.
\end{rem}

\begin{co}\label{lb7}
For any $g\in\GA$, $U_i(g)\in\bigvee_{I\in\mc J}\pi_{i,I}(\mc A(I))$, and 
\begin{align}
U_i(g)\pi_{i,I}(x)U_i(g)^*=\pi_{i,gI}(U(g)xU(g)^* )\label{eq16}
\end{align}
for any $I\in\mc J,x\in\mc A(I)$.
\end{co}
\begin{proof}
Clearly $U_i(g)\in\bigvee_{I\in\mc J}\pi_i(\mc A(I))$ when $g\in\GA(J)$ for some $J\in\mathcal J$. Thus it holds in general by proposition \ref{lb6}. On the other hand, if we fix $J\in\mathcal J$ and $g\in\GA(J)$, then \eqref{eq16} holds whenever $x\in\mc A(I)$ and $I$ is small enough such that $I$ and $J$ can be covered by an open interval in $S^1$. Thus, by the additivity of $\mc A$, \eqref{eq16} holds for any $I\in\mc J$ and $x\in\mc A(I)$  and the given $g\in\GA(J)$. Again, by proposition \ref{lb6}, equation \eqref{eq16} holds for any $g\in\GA$.
\end{proof}

Note that $\mathscr G$ restricts to the universal covering space $\wtd\PSU(1,1)$ of $\PSU(1,1)$, which is generated by $\wtd\exp(iX)$ where $X=a_1L_1+a_0L_0+a_{-1}L_{-1}$ is a self adjoint . By \cite{Bar54}, if we restrict the projective representation of $\mathscr G$ on $\mc H_i$ to a projective representation of $\wtd\PSU(1,1)$, then the latter can be lifted uniquely to a (continuous) unitary representation of $\wtd\PSU(1,1)$, also denoted by $U_i$. This shows that any conformal $\mathcal A$-module $\mathcal H_i$ is M\"obius covariant, in the sense that besides the positivity of $L_0$, there exists a (continuous) unitary representation of $\wtd\PSU(1,1)$ on $\mathcal H_i$ such that \eqref{eq16} holds for any $g\in\wtd\PSU(1,1)$.

\subsection{Connes fusion $\mathcal H_i(I)\boxtimes\mathcal H_j(J)$}

Starting from this section, we use Connes fusion to study the tensor category of the representations of conformal nets. Except in section 2.4, most of the discussions in this  and the following chapters  do not rely on the conformal structures of conformal net modules. Thus the results are also true for M\"obius covariant nets and their (normal) representations.

For any $\mathcal A$-modules $\mathcal H_i$, $\mathcal H_j$, we let $\Hom_{\mathcal A}(\mathcal H_i,\mathcal H_j)$ be the vector space of  bounded linear operators $T:\mathcal H_i\rightarrow\mathcal H_j$, such that $T\pi_i(x)=\pi_j(x)T$ for any $I\in\mathcal J$ and $x\in\mc A(I)$. Similarly, given $I\in\mathcal J$, we let $\Hom_{\mathcal A(I)}(\mathcal H_i,\mathcal H_j)$ be the vector space of  bounded linear operators $\mathcal H_i\rightarrow\mathcal H_j$ intertwining only the actions of $\mathcal A(I)$. Since $\mathcal A(I)$ is a type III factor,  $\mathcal H_i$ and $\mathcal H_j$  are equivalent as $\mathcal A(I)$-modules if they are both non-trivial. Therefore $\Hom_{\mathcal A(I)}(\mathcal H_i,\mathcal H_j)$ has unitary operators.

\begin{df}
Let $\mathcal H_i$ be an $\mathcal A$-module.  Given $I\in\mathcal J$, we say that a vector $\xi\in\mathcal H_i$ is \textbf{$I$-bounded}, if there exists $A\in\Hom_{\mathcal A(I^c)}(\mathcal H_0,\mathcal H_i)$, such that $A\Omega=\xi$.
\end{df}
Since $\mathcal A(I^c)\Omega$ is dense in $\mathcal H_0$ by Reeh-Schlieder theorem, such $A$, if exists, must be unique, and we will denote this operator by $Z(\xi,I)$. Let $\mathcal H_i(I)$ be the set of $I$-bounded vectors in $\mathcal H_i$. In other words $\mathcal H_i(I)=\Hom_{\mathcal A(I^c)}(\mathcal H_0,\mathcal H_i)\Omega$. Then clearly $\mathcal H_0(I)=\mathcal A(I)\Omega$ by Haag duality. In particular $\mathcal H_0(I)$ is dense. Since there exist unitary operators in  $\Hom_{\mathcal A(I^c)}(\mathcal H_0,\mathcal H_i)$, $\mathcal H_i(I)$ is also dense in $\mathcal H_i$.\footnote{One  can indeed prove the density without appeal to the type III property. See \cite{Tak02} chapter IX lemma 3.3 (iii).}

We now define the Connes fusion product(s) of two $\mathcal A$-modules $\mathcal H_i,\mathcal H_j$. Choose disjoint $I,J\in\mathcal J$. We define a positive sesquilinear form $\bk{\cdot|\cdot}$ (antilinear on the second variable) on $\mathcal H_i(I)\otimes\mathcal H_j(J)$  by setting, for any $\xi_1,\xi_2\in\mathcal H_i(I),\eta_1,\eta_2\in\mathcal H_j(J)$,
\begin{align}
\bk{\xi_1\otimes\eta_1|\xi_2\otimes\eta_2}=\bk{Z(\eta_2,J)^*Z(\eta_1,J)Z(\xi_2,I)^*Z(\xi_1,I)\Omega|\Omega}.
\end{align}
The positivity of $\bk{\cdot|\cdot}$ is easy to show (see for example \cite{Tak02} proposition IX.3.15). Since $Z(\xi_2,I)^*Z(\xi_1,I)\in\Hom_{\mathcal A(I^c)}(\mathcal H_0,\mathcal H_0)=\mathcal A(I^c)'=\mathcal A(I)$ and, similarly, $Z(\eta_2,J)^*Z(\eta_1,J)\in\mathcal A(J)$, we also have
\begin{align}
\bk{\xi_1\otimes\eta_1|\xi_2\otimes\eta_2}=\bk{Z(\xi_2,I)^*Z(\xi_1,I)Z(\eta_2,J)^*Z(\eta_1,J)\Omega|\Omega}.
\end{align}

\begin{df}
Define a Hilbert space $\mathcal H_i(I)\boxtimes\mathcal H_j(J)$ to be the completion of $\mathcal H_i(I)\otimes\mathcal H_j(J)$ under  $\bk{\cdot|\cdot}$. This Hilbert space is called the \textbf{Connes fusion (product) of $\mathcal H_i,\mathcal H_j$ over the intervals $I,J$.}
\end{df}
For simplicity, we  let $\xi\otimes \eta\in\mc H_i(I)\otimes\mc H_j(J)$ also denote the corresponding vector in $\mc H_i(I)\otimes\mc H_j(J)$.

Note that the order of Connes fusion doesn't matter: we can identify $\mathcal H_i(I)\boxtimes\mathcal H_j(J)$ with $\mathcal H_j(J)\boxtimes\mathcal H_i(I)$ by the canonical map $\mathcal H_i(I)\boxtimes\mathcal H_j(J)\ni\xi\otimes\eta\mapsto\eta\otimes\xi$.\\

We now relate Connes fusions over different intervals. Note that by the intertwining properties of these $Z$'s, we clearly have
\begin{align}\label{eq1}
\bk{\xi_1\otimes\eta_1|\xi_2\otimes\eta_2}=\bk{\pi_j(Z(\xi_2,I)^*Z(\xi_1,I))\eta_1|\eta_2}=\bk{\pi_i(Z(\eta_2,J)^*Z(\eta_1,J))\xi_1|\xi_2}.
\end{align}
Here $Z(\xi_2,I)^*Z(\xi_1,I)$ and $Z(\eta_2,J)^*Z(\eta_1,J)$ are regarded respectively as elements in $\mc A(I)$ and $\mc A(J)$. From these relations, one easily sees that $H\otimes K$ is dense in $\mathcal H_i(I)\boxtimes\mathcal H_j(J)$ under the inner product $\bk{\cdot|\cdot}$ if $H$ and $K$ are dense subspaces of $\mathcal H_i(I),\mathcal H_j(J)$ respectively. In particular, if $I_1\subset I$, $J_1\subset J$ are open intervals, then, as $\mathcal H_i(I_1)$ is dense in $\mathcal H_i$ (and therefore in $\mathcal H_i(I)$) and $\mathcal H_j(J_1)$ is dense in $\mathcal H_j(J)$, $\mathcal H_i(I_1)\boxtimes\mathcal H_j(J_1)$ is the same as $\mathcal H_i(I)\boxtimes\mathcal H_j(J)$.

\begin{df}
Let $I_1\subset I$ and $J_1\subset J$ be open intervals. By  \textbf{canonical equivalence} (or  \textbf{canonical map}) $\mathcal H_i(I_1)\boxtimes\mathcal H_j(J_1)\overset{\simeq}{\rightarrow }\mathcal H_i(I)\boxtimes\mathcal H_j(J)$ we mean the unitary map defined by $\xi\otimes\eta\mapsto\xi\otimes\eta$, where $\xi\in\mathcal H_i(I_1),\eta\in\mathcal H_j(J_1)$. Its inverse map is called the \text{canonical equivalence} $\mathcal H_i(I)\boxtimes\mathcal H_j(J)\overset{\simeq}{\rightarrow}\mathcal H_i(I_1)\boxtimes\mathcal H_j(J_1)$.
\end{df}

Next, we shall relate $\mathcal H_i(I_1)\boxtimes\mathcal H_j(J_1)$ and $\mathcal H_i(I_2)\boxtimes\mathcal H_j(J_2)$ when $I_1,I_2$ and $J_1,J_2$ are in general positions. In this case the equivalence maps will depend on the homotopy classes of paths relating these two pairs of intervals. This is where braid groups enter our story. To this end, we first define, for any distinct points $z,\zeta\in S^1$, the \textbf{Connes fusion of $\mathcal H_i,\mathcal H_j$ over $z,\zeta$}, to be a Hilbert space
\begin{align*}
\mathcal H_i(z)\boxtimes\mathcal H_j(\zeta)=\varinjlim_{(I,J)\ni(z,\zeta)}\mathcal H_i(I)\boxtimes\mathcal H_j(J)=\Bigg(\coprod_{(I,J)\ni(z,\zeta)}\mathcal H_i(I)\boxtimes\mathcal H_j(J)\Bigg)\Bigg/\simeq,
\end{align*}
where the subscript open intervals  $I,J\in\mathcal J$ are disjoint, and the equivalence relation is given by the canonical equivalence.\footnote{It will be interesting to compare our definition with the $P(z)$-tensor products in \cite{HL95a}.} Then for any fixed disjoint open intervals $I,J\in\mathcal J$ containing $z,\zeta$ respectively, we have an obvious canonical map
\begin{align*}
\mathcal H_i(I)\boxtimes\mathcal H_j(J)\overset{\simeq}{\rightarrow}\mathcal H_i(z)\boxtimes\mathcal H_j(\zeta)
\end{align*}
as well as its inverse map.

Now let $\Conf_2(S^1)=\{(z,\zeta\in S^1):z\neq \zeta\}$.  Let $\gamma(t)=(\alpha(t),\beta(t))$ ($0\leq t\leq 1$) be a path in $\Conf_2(S^1)$ with initial point $(z_1,\zeta_1)=\gamma(0)$ and end point $(z_2,\zeta_2)=\gamma(1)$. We shall use this path to define a unitary map $\mathcal H_i(z_1)\boxtimes\mathcal H_j(\zeta_1){\rightarrow}\mathcal H_i(z_2)\boxtimes\mathcal H_j(\zeta_2)$. First, we say that $\gamma$ is small if there exist disjoint open intervals $I,J\in\mathcal J$ such that the image of $\gamma$ is included in $I\times J$. Then the map $\gamma^\bullet:\mathcal H_i(z_1)\boxtimes\mathcal H_j(\zeta_1){\rightarrow}\mathcal H_i(z_2)\boxtimes\mathcal H_j(\zeta_2)$ is defined using the canonical equivalences
\begin{align*}
\mathcal H_i(z_1)\boxtimes\mathcal H_j(\zeta_1)\overset{\simeq}{\rightarrow}\mathcal H_i(I)\boxtimes\mathcal H_j(J)  \overset{\simeq}{\rightarrow}\mathcal H_i(z_2)\boxtimes\mathcal H_j(\zeta_2).
\end{align*}
For a general path $\gamma$, we choose $0=t_0<t_1<t_2<\cdots<t_n=1$, such that $\gamma|_{[t_{k-1},t_k]}$ is small for any $k=1,2,\dots,n$. This is called a partition of $\gamma$. Let $\gamma_k$ be a reparametrization of $\gamma|_{[t_{k-1},t_k]}$ such that the variable $t$ is again defined on $[0,1]$. We then define a unitary map
\begin{gather*}
\gamma^\bullet:\mathcal H_i(z_1)\boxtimes\mathcal H_j(\zeta_1){\rightarrow}\mathcal H_i(z_2)\boxtimes\mathcal H_j(\zeta_2),\\
\gamma^\bullet=\gamma^\bullet_n\gamma^\bullet_{n-1}\cdots \gamma^\bullet_1.
\end{gather*}
Obviously, finer partitions give the same result. Therefore the map $\gamma^\bullet$ is independent of the partitions. We call it the \textbf{path continuation $\mathcal H_i(z_1)\boxtimes\mathcal H_j(\zeta_1)\overset{\simeq}{\rightarrow}\mathcal H_i(z_2)\boxtimes\mathcal H_j(\zeta_2)$ induced by $\gamma$}. 

Now we return to the Connes fusions over open intervals. Suppose we have two pairs of mutually disjoint open intervals $I_1,J_1$ and $I_2,J_2$ in $ S^1$. Choose a path $\gamma$  in $\Conf_2(S^1)$ such that $\gamma(0)\in I_1\times J_1$, $\gamma(1)\in I_2\times J_2$. Let $(z_1,\zeta_1)=\gamma(0)$ and $(z_2,\zeta_2)=\gamma(1)$. We then define the \textbf{path continuation $\mathcal H_i(I_1)\boxtimes\mathcal H_j(J_1)\overset{\simeq}{\rightarrow }\mathcal H_i(I_2)\boxtimes\mathcal H_j(J_2)$ induced by $\gamma$} to be the map $\gamma^\bullet$ defined by
\begin{align*}
\mathcal H_i(I_1)\boxtimes\mathcal H_j(J_1)\overset{\simeq}{\rightarrow}\mathcal H_i(z_1)\boxtimes\mathcal H_j(\zeta_1) \overset{\gamma^\bullet}{\rightarrow} \mathcal H_i(z_2)\boxtimes\mathcal H_j(\zeta_2)\overset{\simeq}{\rightarrow}\mathcal H_i(I_2)\boxtimes\mathcal H_j(J_2).
\end{align*}
The following obvious lemma provides a practical way of calculating $\gamma^\bullet$.

\begin{lm}\label{lb1}
Choose disjoint $I_1,J_1\in\mathcal J$, and $I_2,J_2\in\mathcal J$. Let $\gamma$ be a path in $\Conf_2(S^1)$ from $I_1\times J_1$ to $I_2\times J_2$. If $\gamma([0,1])\subset I_1\times J_1$ or $\gamma([0,1])\subset I_2\times J_2$, and $I_1\cap I_2,J_1\cap J_2\in\mathcal J$ (see figure \ref{fig1}), then $\gamma^\bullet:\mathcal H_i(I_1)\boxtimes\mathcal H_j(J_1)\rightarrow\mathcal H_i(I_2)\boxtimes\mathcal H_j(J_2)$ equals the map
\begin{align*}
\mathcal H_i(I_1)\boxtimes\mathcal H_j(J_1)\overset{\simeq}{\rightarrow}\mathcal H_i(I_1\cap I_2)\boxtimes\mathcal H_j(J_1\cap J_2)\overset{\simeq}{\rightarrow}\mathcal H_i(I_2)\boxtimes\mathcal H_j(J_2).
\end{align*}
\begin{figure}[h]
	\centering
	\includegraphics[width=0.25\linewidth]{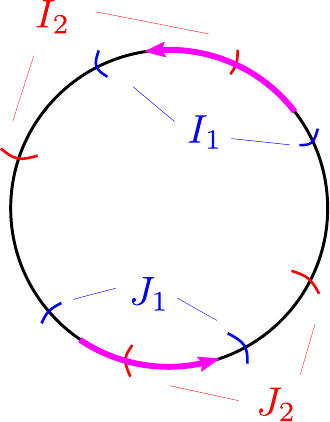}
	\caption{}
	\label{fig1}
\end{figure}
\end{lm}

We now show that homotopic paths induce the same map.
\begin{pp}\label{lb4}
Let $\gamma,\widetilde\gamma$ be two paths in $\Conf_2(S^1)$ with $\gamma(0),\widetilde\gamma(0)\in I_1\times J_1$ and  $\gamma(1),\widetilde\gamma(1)\in I_2\times J_2$. Suppose that there exists a homotopy map $\Gamma:[0,1]\times[0,1]\rightarrow\Conf_2(S^1)$  connecting the two paths $\gamma=\Gamma(\cdot,0)$ and $\widetilde\gamma=\Gamma(\cdot,1)$. Assume moreover that $\Gamma(0,[0,1])\subset I_1\times J_1,\Gamma(1,[0,1])\subset I_2\times J_2$. Then $\gamma^\bullet=\widetilde\gamma^\bullet$.
\end{pp}
\begin{proof}
Choose $0=t_0<t_1<\cdots<t_m=1, 0=s_0<s_1<\cdots<s_n=1$ such that for any $a=0,1,\cdots,m,b=0,1\cdots,n$, there exists a pair of disjoint open intervals $I_{a,b},J_{a,b}$ in $S^1$ satisfying the following conditions:\\
(1)  $\Gamma([t_{a-1},t_a]\times[s_{b-1},s_b])\subset I_{a,b}\times J_{a,b}$ when $a,b>0$.\\
(2) $I_{0,b}=I_1, J_{0,b}=J_1,I_{m,b}=I_2,J_{m,b}=J_2$.\\
(3) When $a>0$,  $I_{a-1,b}\cap I_{a,b}$ and $J_{a-1,b}\cap J_{a,b}$ are open intervals in $S^1$.\\
Then  for any $b$, the map $\Gamma(\cdot,s_b)^\bullet:\mathcal H_i(I_1)\boxtimes\mathcal H_j(J_1){\rightarrow }\mathcal H_i(I_2)\boxtimes\mathcal H_j(J_2)$ induced by the path $\Gamma(\cdot,s_b)$  is, by lemma \ref{lb1}, equal to the map $R_b$ defined by
\begin{align*}
&\mathcal H_i(I_1)\boxtimes\mathcal H_j(J_1)=\mathcal H_i(I_{0,b})\boxtimes\mathcal H_j(J_{0,b})\overset{\simeq}{\rightarrow}\mathcal H_i(I_{0,b}\cap I_{1,b})\boxtimes\mathcal H_j(J_{0,b}\cap J_{1,b})\\
\overset{\simeq}{\rightarrow}&\mathcal H_i(I_{1,b})\boxtimes\mathcal H_j(J_{1,b})
\overset{\simeq}{\rightarrow}\mathcal H_i(I_{1,b}\cap I_{2,b})\boxtimes\mathcal H_j(J_{1,b}\cap J_{2,b})\overset{\simeq}{\rightarrow}\mathcal H_i(I_{2,b})\boxtimes\mathcal H_j(J_{2,b})\\
\overset{\simeq}{\rightarrow}&\cdots\overset{\simeq}{\rightarrow}\mathcal H_i(I_{m,b})\boxtimes\mathcal H_j(J_{m,b})=\mathcal H_i(I_2)\boxtimes\mathcal H_j(J_2).
\end{align*}
Similarly $\Gamma(\cdot,s_b)^\bullet$ also equals $R_{b+1}$. Therefore
\begin{align*}
\gamma^\bullet=\Gamma(\cdot,s_0)^\bullet=R_1=\Gamma(\cdot,s_1)^\bullet=R_2=\Gamma(\cdot,s_2)^\bullet=\cdots=\Gamma(\cdot,s_n)^\bullet=\widetilde\gamma^\bullet.
\end{align*} 
\end{proof}

We close this section with a brief discussion of Connes fusions defined on a single interval. For any $I\in\mathcal J$ we define the Hilbert space $\mathcal H_i(I)\boxtimes\mathcal H_j$ to be the closure of $\mathcal H_i(I)\otimes\mathcal H_j$ under the positive sesquilinear form $\bk{\cdot|\cdot}$  defined by
\begin{align}
\bk{\xi_1\otimes\eta_1|\xi_2\otimes\eta_2}=\bk{\pi_j(Z(\xi_2,I)^*Z(\xi_1,I))\eta_1|\eta_2}
\end{align}
for any $\xi_1,\xi_2\in\mathcal H_i(I),\eta_1,\eta_2\in\mathcal H_j$. Then clearly $H\otimes K$ is dense in $\mathcal H_i(I)\boxtimes\mathcal H_j$ when $H$ is dense in $\mathcal H_i(I)$ and $K$ is dense in $\mathcal H_j$. In particular we can take $H=\mathcal H_i(I)$ and $K=\mathcal H_j(J)$ where $J\in\mathcal J$ is disjoint from $I$. Therefore, by \eqref{eq1}, we have a \textbf{canonical equivalence} $\mathcal H_i(I)\boxtimes\mathcal H_j(J)\overset{\simeq}{\rightarrow}\mathcal H_i(I)\boxtimes\mathcal H_j$ defined by $\xi\otimes\eta\mapsto\xi\otimes\eta$ ($\xi\in\mathcal H_i(I),\eta\in\mathcal H_j(J)$). Its inverse is also called the \textbf{canonical equivalence} $\mathcal H_i(I)\boxtimes\mathcal H_j\overset{\simeq}{\rightarrow}\mathcal H_i(I)\boxtimes\mathcal H_j(J)$.

Now for any $z\in S^1$, one can define $\mathcal H_i(z)\boxtimes\mathcal H_j$, in a similar to way, to be  $\varinjlim_{I\ni z}\mathcal H_i(I)\boxtimes\mathcal H_j$. One therefore has a canonical equivalence between $\mathcal H_i(z)\boxtimes\mathcal H_j$ and $\mathcal H_i(I)\boxtimes\mathcal H_j$ for any $I\in\mathcal J$ containing $z$. If $\alpha$ is a path in $S^1$ from $z_1$ to $z_2$, one can define the \textbf{path continuation $\alpha^\bullet:\mathcal H_i(z_1)\boxtimes\mathcal H_j{\rightarrow}\mathcal H_i(z_2)\boxtimes\mathcal H_j$ induced by $\alpha$} in a similar way. One can furthermore use this map to define, for any path $\alpha$  in $S^1$ from $I_1\in\mathcal J$ to $I_2\in\mathcal J$, a map $\alpha^\bullet:\mathcal H_i(I_1)\boxtimes\mathcal H_j{\rightarrow}\mathcal H_i(I_2)\boxtimes\mathcal H_j$, also called the \textbf{equivalence induced by $\alpha$}. Homotopic paths induce the same map. 

Path continuations of Connes fusions over single intervals can be related to those over two intervals by the following proposition.
\begin{pp}
Let $\gamma=(\alpha,\beta)$ be a path in $\Conf_2(S^1)$ from $I_1\times J_1$ to $I_2\times J_2$, where the open intervals $I_1,J_1$ are disjoint, and $I_2,J_2$ are disjoint. Then the equivalence $\alpha^\bullet:\mathcal H_i(I_1)\boxtimes\mathcal H_j{\rightarrow}\mathcal H_i(I_2)\boxtimes\mathcal H_j$ equals the map
\begin{align}
\mathcal H_i(I_1)\boxtimes\mathcal H_j\overset{\simeq}{\rightarrow}\mathcal H_i(I_1)\boxtimes\mathcal H_j(J_1)\overset{\gamma^\bullet}{\rightarrow}\mathcal H_i(I_2)\boxtimes\mathcal H_j(J_2)\overset{\simeq}{\rightarrow}\mathcal H_i(I_2)\boxtimes\mathcal H_j.\label{eq2}
\end{align} 
\end{pp}

\begin{proof}
Let $R$ denote the unitary map described by \eqref{eq2}. We first assume that $\gamma$ is small enough, such that $\alpha([0,1])\subset I_1,\beta([0,1])\subset J_1$, and $I_1\cap I_2,J_1\cap J_2\in\mathcal J$.  Then by lemma \ref{lb1} (and its variant for single interval fusions), $\alpha^\bullet$ and $R$ coincide when acting on the dense subspace $\mathcal H_i(I_1\cap I_2)\otimes\mathcal H_j(J_1\cap J_2)$ of  $\mathcal H_i(I)\boxtimes\mathcal H_j$. Therefore $\alpha^\bullet=R$.

For a general $\gamma$, we consider a partition $0=t_0<t_1<\cdots<t_n=1$, such that for each $a=0,1,2,\dots,n$, there exists a pair of disjoint open intervals $I'_a,J'_a$ in $S^1$ satisfying the following conditions:\\
(1) $\gamma([t_{a-1},t_a])\subset I'_a$ when $a>0$.\\
(2) $I'_0=I_1,J'_0=J_1,I'_n=I_1,J'_n=J_1$.\\
(3) When $a>0$, $I'_{a-1}\cap I'_a$ and $J'_{a-1}\cap J'_a$ are open intervals in $S^1$.\\
For each $a>0$, choose a path $\gamma_a=(\alpha_a,\beta_a)$ defined on $[0,1]$ to be a reparametrization of $\gamma|_{[t_{a-1},t_a]}$. Then $\gamma_a$ is small and, from the last paragraph, the map $\alpha_a^\bullet:\mathcal H_i(I_{a-1}')\boxtimes\mathcal H_j\rightarrow\mathcal H_i(I_a')\boxtimes\mathcal\mathcal H_j$ equals the map $R_a$ defined by
\begin{align}
\mathcal H_i(I_{a-1}')\boxtimes\mathcal H_j\overset{\simeq}{\rightarrow}\mathcal H_i(I_{a-1}')\boxtimes\mathcal H_j(J_{a-1}')\overset{\gamma_a^\bullet}{\rightarrow}\mathcal H_i(I_a')\boxtimes\mathcal H_j(J_a')\overset{\simeq}{\rightarrow}\mathcal H_i(I_a')\boxtimes\mathcal H_j.
\end{align} 
Now $\alpha^\bullet=R$ follows from the fact that $\alpha^\bullet=\alpha_n^\bullet\alpha_{n-1}^\bullet\cdots\alpha_1^\bullet$ and $R=R_nR_{n-1}\cdots R_1$.
\end{proof}

Similar properties also hold for $\mathcal H_i\boxtimes\mathcal H_j(J)$.

\subsection{Actions of conformal nets}

Assume as usual that $I,J\in\mathcal J$ are disjoint. In this section, we equip $\mathcal H_i(I)\boxtimes\mathcal H_j(J)$ with an $\mathcal A$-module structure, and show that the action of $\mc A$ commutes with path continuations. First, note that we have  natural representations of $\mathcal A(I)$ and $\mathcal A(J)$ on $\mathcal H_i(I)\boxtimes\mathcal H_j(J)$ defined by
\begin{gather*}
x(\xi\otimes\eta)=x\xi\otimes\eta,\qquad y(\xi\otimes\eta)=\xi\otimes y\eta
\end{gather*}
for any $\xi\in\mathcal H_i(I),\eta\in\mathcal H_j(J),x\in\mathcal A(I),y\in\mathcal A(J)$. If $K\subset I$ (resp. $K\subset J$) then the above representation of $\mathcal A(I)$ (resp. $\mathcal A(J)$) restricts to one of $\mathcal A(K)$. We would like to define a natural action of $\mathcal A(K)$ on $\mathcal H_i(I)\boxtimes\mathcal H_j(J)$ even when $K$ is not contained inside $I$ or $J$. The following proposition gives us a clue on how to define it.

\begin{pp}\label{lb2}
Let $\gamma$ be a path in $\Conf_2(S^1)$ from $I\times J$ to $J\times I$. Then for the path continuation $\gamma^\bullet:\mathcal H_i(I)\boxtimes\mathcal H_j(J)\rightarrow\mathcal H_i(J)\boxtimes\mathcal H_j(I)$, we have
\begin{align}
x=(\gamma^\bullet)^{-1}x\gamma^\bullet \label{eq3}
\end{align}
for any $x\in\mathcal A(I)$. A similar result holds for any $y\in\mathcal A(J)$.
\end{pp}
Note that the $x$ on the left hand side of \eqref{eq3} is acting on vectors in $\mathcal H_i$, whereas on the right hand side, $x$ is acting on vectors in $\mathcal H_j$. To prove this proposition we first need a lemma.

\begin{lm}\label{lb3}
Choose disjoint $I_1,J_1\in\mathcal J$, and $I_2,J_2\in\mathcal J$. If $I_1\cap I_2,J_1\cap J_2\in\mathcal J$, then the map
\begin{align*}
\mathcal H_i(I_1)\boxtimes\mathcal H_j(J_1)\overset{\simeq}{\rightarrow}\mathcal H_i(I_1\cap I_2)\boxtimes\mathcal H_j(J_1\cap J_2)\overset{\simeq}{\rightarrow}\mathcal H_i(I_2)\boxtimes\mathcal H_j(J_2)
\end{align*}
intertwines the actions of $\mathcal A(I_1\cap I_2)$.
\end{lm}
\begin{proof}
Denote the above map by $R$. Choose an arbitrary $x\in\mathcal A(I_1\cap I_2)$. Then for any $\xi\in\mathcal H_i(I_1\cap I_2)$ and $\eta\in \mathcal H_j(J_1\cap J_2)$, $xR(\xi\otimes\eta)$ clearly equals $Rx(\xi\otimes\eta)$, which must be $x\xi\otimes\eta\in\mathcal H_i(I_2)\boxtimes\mathcal H_j(J_2)$. Therefore, by density, $xR=Rx$.
\end{proof}

\begin{proof}[Proof of proposition \ref{lb2}]
Since by additivity we have $\mathcal A(I)=\bigvee_{K\subset\joinrel\subset I} \mathcal A(K)$  ($K\subset\joinrel\subset I$ means $K\in\mathcal J$ and $I$ contains the closure of $K$), it suffices to verify $x=(\gamma^\bullet)^{-1}x\gamma^\bullet$ for any fixed $K\subset\joinrel\subset I$ and $x\in\mathcal A(K)$. Another way to achieve this is to replace, by density, $J$ by a smaller $J_1\subset\joinrel\subset J$, and let $J_1^c,J_1,I$ be the new $I,J,K$.

Let $\gamma=(\alpha,\beta)$. We first assume that $\gamma$ is small enough, such that $\alpha([0,1])$ and $\beta([0,1])$ can both be covered by open intervals in $S^1$. For example, $\gamma$ can be a clockwise or 
anticlockwise rotation not exceeding $2\pi$. We now choose a pair of disjoint open intervals $I_1,J_1$ in $S^1$ satisfying\\
(a) $I_1\cap I,I_1\cap J,J_1\cap I, J_1\cap J\in\mathcal J$,\\
(b) $K\subset J_1$,\\
(c) $\gamma$ is homotopic (in the sense of proposition \ref{lb4}) to a path $\widetilde\gamma=(\widetilde\alpha,\widetilde\beta)$ from $I\times J$ to $J\times I$, such that $\widetilde\alpha([0,1])$ is covered by $I\cup I_1\cup J$ and $\widetilde\beta([0,1])$ is covered by $I\cup J_1\cup J$.\\
(See figure \ref{fig2}.) 
\begin{figure}[h]
	\centering
	\includegraphics[width=0.6\linewidth]{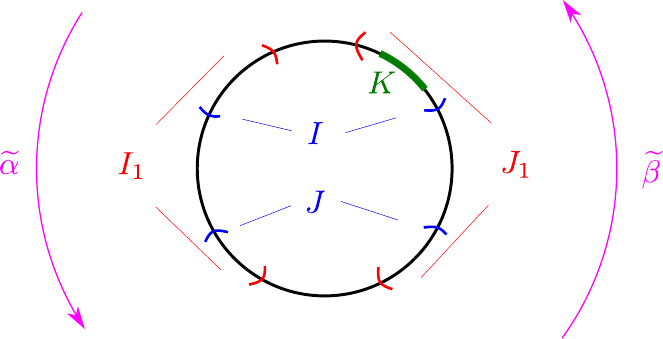}
	\caption{}
	\label{fig2}
\end{figure}
Now consider  two unitary maps
\begin{gather*}
R:\mathcal H_i(I)\boxtimes\mathcal H_j(J)\overset\simeq\rightarrow\mathcal H_i(I\cap I_1)\boxtimes\mathcal H_j(J\cap J_1)\overset\simeq\rightarrow\mathcal H_i(I_1)\boxtimes\mathcal H_j(J_1),\\
S:\mathcal H_i(I_1)\boxtimes\mathcal H_j(J_1)\overset\simeq\rightarrow\mathcal H_i(I_1\cap J)\boxtimes\mathcal H_j(J_1\cap I)\overset\simeq\rightarrow\mathcal H_i(J)\boxtimes\mathcal H_j(I).
\end{gather*}
Then from lemma \ref{lb1} it is easy to see that $\gamma^\bullet=SR$. According to lemma \ref{lb3}, $Sx=xS$. Therefore it suffices to show $x=R^*xR$. 

Choose any $\xi\in\mathcal H_i(I\cap I_1),\eta\in\mathcal H_j(J\cap J_1)$. Then such $\xi\otimes\eta$ span a dense subspace of $\mathcal H_i(I)\boxtimes\mathcal H_j(J)$. Clearly $R(\xi\otimes\eta)=\xi\otimes \eta\in\mathcal H_i(I_1)\boxtimes\mathcal H_j(J_1)$. Now since $x\in\mathcal A(K)$  and $K\subset J_1$, we have $xR(\xi\otimes\eta)=\xi\otimes x\eta\in\mathcal H_i(I_1)\boxtimes\mathcal H_j(J_1)$. Choose arbitrary $\xi'\in\mathcal H_i(I\cap I_1),\eta'\in\mathcal H_j(J\cap J_1)$. We also have $R(\xi'\otimes\eta')=\xi'\otimes\eta'$. Therefore
\begin{align}
&\bk{R^*xR(\xi\otimes\eta)|\xi'\otimes\eta'}=\bk{xR(\xi\otimes\eta)|R(\xi'\otimes\eta')}=\bk{\xi\otimes x\eta|\xi'\otimes\eta'}\nonumber\\
=&\bk{\pi_j(Z(\xi',I_1)^*Z(\xi,I_1))x\eta|\eta'}=\bk{\pi_j(Z(\xi',I_1)^*Z(\xi,I_1)x)\eta|\eta'}.\label{eq4}
\end{align}
Since $\xi,\xi'\in\mathcal H_i(I\cap I_1)$, we actually have $Z(\xi,I_1)=Z(\xi,I\cap I_1)=Z(\xi,I)$ and, similarly, $Z(\xi',I_1)=Z(\xi',I)$. Note also that $x\in\mathcal A(K)$ and $K\subset I$. So $Z(\xi,I)x\in\Hom_{\mathcal A(I^c)}(\mathcal H_0,\mathcal H_i)$. As $K$ is disjoint from $I\cap I_1$, by locality, $Z(\xi,I)x\Omega=xZ(\xi,I)\Omega=x\xi$. Therefore $Z(\xi,I)x=Z(x\xi,I)$. So \eqref{eq4} equals
\begin{align*}
&\bk{\pi_j(Z(\xi',I)^*Z(\xi,I)x)\eta|\eta'}=\bk{\pi_j(Z(\xi',I)^*Z(x\xi,I))\eta|\eta'}\\
=&\bk{x\xi\otimes\eta|\xi'\otimes\eta'}=\bk{x(\xi\otimes\eta)|\xi'\otimes\eta'}.
\end{align*}
This proves $R^*xR=x$, and hence $x=(\gamma^\bullet)^{-1}x\gamma^\bullet$.

We now prove \eqref{eq3} for more general $\gamma$. Let $\gamma_1$ be small path (in the same sense as above) from $I\times J$ to $J\times I$, and $\gamma_2$ another small path from $J\times I$ to $I\times J$, such that $\gamma_2*\gamma_1$ is homotopic to an anticlockwise rotation by $2\pi$. Then there exists $n\in\mathbb Z$ such that $\gamma$ is homotopic to $\gamma_1*(\gamma_2*\gamma_1)^{(*n)}$. So $\gamma^\bullet=\gamma_1^\bullet(\gamma_2^\bullet\gamma_1^\bullet)^n$. From what we've shown, both $\gamma_1^\bullet$ and $\gamma_2^\bullet$ intertwines $x$. Therefore $\gamma^\bullet x=x\gamma^\bullet$.
\end{proof}

\begin{thd}
Let $\mathcal H_i,\mathcal H_j$ be $\mathcal A$-modules, and choose disjoint $I,J\in\mathcal J$. \\
(a) There exists a (unique) representation $\pi_{i\boxtimes j}^l$ of $\mathcal A$ on $\mathcal H_i(I)\boxtimes\mathcal H_j(J)$ satisfying the following condition: If $K,L\in\mathcal J$ are disjoint, $\gamma$ is a path in $\Conf_2(S^1)$ from $I\times J\rightarrow K\times L$, and $x\in\mathcal A(K)$, then 
\begin{align}
\pi_{i\boxtimes j}^l(x)=(\gamma^\bullet)^{-1}x\gamma^\bullet,\label{eq5}
\end{align}
where $\gamma^\bullet:\mathcal H_i(I)\boxtimes\mathcal H_j(J)\rightarrow\mathcal H_i(K)\boxtimes\mathcal H_j(L)$ is the equivalence induced by $\gamma$.\\
(b) There exists a (unique) representation $\pi_{i\boxtimes j}^r$ of $\mathcal A$ on $\mathcal H_i(I)\boxtimes\mathcal H_j(J)$ satisfying the following condition: If $K,L\in\mathcal J$ are disjoint, $\varsigma$ is a path in $\Conf_2(S^1)$ from $I\times J\rightarrow L\times K$, and $x\in\mathcal A(K)$, then
\begin{align}
\pi_{i\boxtimes j}^r(x)=(\varsigma^\bullet)^{-1}x\varsigma^\bullet,
\end{align}
where $\varsigma^\bullet:\mathcal H_i(I)\boxtimes\mathcal H_j(J)\rightarrow\mathcal H_i(L)\boxtimes\mathcal H_j(K)$ is the equivalence induced by $\vs$.\\
(c) $\pi_{i\boxtimes j}^l$ and $\pi_{i\boxtimes j}^r$ are equal. Therefore $\pi_{i\boxtimes j}:=\pi_{i\boxtimes j}^l=\pi_{i\boxtimes j}^r$ gives $\mathcal A$ a natural representation on $\mathcal H_i(I)\boxtimes\mathcal H_j(J)$.\\
(d) The unitary maps induced by inclusions of intervals, restrictions of intervals, and path continuations are equivalences of $\mathcal A$-modules, i.e., they intertwine the actions of $\mathcal A$.
\end{thd}

\begin{proof}
Choose disjoint $K,L\in\mathcal J$, and choose $\gamma,\varsigma$ as in (a) and (b). Then $\varsigma*\gamma^{-1}$ is a path from $K\times L$ to $L\times K$. Apply proposition \ref{lb2} to $\varsigma*\gamma^{-1}$ which induces $\vs^\bullet(\gamma^\bullet)^{-1}$, we obtain
\begin{align}
(\gamma^\bullet)^{-1}x\gamma^\bullet=(\vs^\bullet)^{-1}x\vs^\bullet\label{eq6}
\end{align}
for any $x\in\mathcal A(K)$. Now we define $\pi^l_{i\boxtimes j}$ using relation \eqref{eq5}. We need to show that this definition is independent of $\gamma$. If $\widetilde\gamma$ is another path in $\Conf_2(S^1)$ from $I\times J\rightarrow K\times L$, then by \eqref{eq6}, $(\gamma^\bullet)^{-1}x\gamma^\bullet=(\vs^\bullet)^{-1}x\vs^\bullet=(\widetilde\gamma^\bullet)^{-1}x\widetilde\gamma^\bullet$. This proves the well-definedness of $\pi^l_{i\boxtimes j}$. Thus (a) is proved. (b) can be proved in a similar way. (c) follows directly from \eqref{eq6}. (d) is obvious. 
\end{proof}

\begin{df}
Let $\mathcal H_i,\mathcal H_{i'},\mathcal H_j,\mathcal H_{j'}$ be $\mathcal A$-modules, $I,J\in\mathcal J$ are disjoint, $F\in\Hom_{\mathcal A}(\mathcal H_i,\mathcal H_{i'})$, $G\in\Hom_{\mathcal A}(\mathcal H_j,\mathcal H_{j'})$. We define a (clearly bounded) map $F\otimes G:\mathcal H_i(I)\boxtimes\mathcal H_j(J)\rightarrow\mathcal H_{i'}(I)\boxtimes\mathcal H_{j'}(J)$ such that for any $\xi\in\mathcal H_i(I),\eta\in\mathcal H_j(J)$,
\begin{align}
(F\otimes G)(\xi\otimes\eta)=F\xi\otimes G\eta.\label{eq19}
\end{align}
\end{df}
In the future, when  several different fusion products are considered simultaneously, we will write $F\otimes G$ as $F\boxtimes G$ to avoid ambiguity.

The following properties are easy to show.
\begin{pp}\label{lb8}
$F\otimes G$ commutes with the canonical equivalences induced by inclusions and restrictions of intervals, and hence commutes with path continuations. Moreover, $F\otimes G$ is an $\mathcal A$-module homomorphism, i.e., $F\otimes G\in\Hom_{\mathcal A}(\mathcal H_i(I)\boxtimes\mathcal H_j(J),\mathcal H_{i'}(I)\boxtimes\mathcal H_{j'}(J))$.
\end{pp}
\begin{proof}
The first statement is easy. The second one follows from the first one and the easy fact that $F\otimes G$ intertwines the actions of $\mathcal A(I)$ or $\mathcal A(J)$.
\end{proof}

Actions of $\mathcal A$ on single interval fusions can be defined in a similar way. Let $I\in\mathcal J$. If $K\in\mathcal J$ is a sub-interval of $I$, we let $\mathcal A(K)$ act on $\mathcal H_i(I)\boxtimes \mathcal H_j$ by setting $x(\xi\otimes\eta)=x\xi\otimes\eta$ for any $x\in\mathcal A(K)$, $\xi\in\mathcal H_i(I),\eta\in\mathcal H_j$. For general $K$, we choose a path $\alpha$ in $S^1$ from $I$ to $K$, and let $\mathcal A(K)$ act on $\mathcal H_i(I)\boxtimes \mathcal H_j$ by setting $\pi_{i\boxtimes j}(x)=(\alpha^\bullet)^{-1}x\alpha^\bullet$ for any $x\in\mathcal A(K)$. This action is independent of the path chosen, and hence makes $\mathcal H_i(I)\boxtimes\mathcal H_j$ a natural $\mathcal A$-module. $\mathcal H_i\boxtimes\mathcal H_j(J)$ can be treated in a similar way. Alternatively, one can use the action of $\mathcal A$ on $\mathcal H_i(I)\boxtimes\mathcal H_j(J)$, together with the canonical equivalence $\mathcal H_i(I)\boxtimes\mathcal H_j{\rightarrow}\mathcal H_i(I)\boxtimes\mathcal H_j(J)$, to define the action of $\mathcal A$ on $\mathcal H_i(I)\boxtimes\mathcal H_j$. These two ways give the same definitions. In particular,  $x(\xi\otimes\eta)=\xi\otimes x\eta$ for any $\xi\in\mc H_i(I),\eta\in\mc H_j,x\in\mathcal A(I^c)$. Tensor products of homomorphisms can also be defined using \eqref{eq19}.\\

We now show that a given $\mathcal A$-module $\mathcal H_i$ can be identified with its fusion with the vacuum module $\mathcal H_0$. Define a linear map $\natural_i:\mathcal H_i(I)\otimes\mathcal H_0\rightarrow\mathcal H_i$ satisfying $\natural_i(\xi\otimes y\Omega)=y\xi$ for any $\xi\in\mathcal H_i(I),y\in\mathcal A(I^c)$. It is easy to check that $\natural_i$ is an isometry with dense range. Therefore $\natural_i$ extends to a unitary map $\natural_i:\mathcal H_i(I)\boxtimes\mathcal H_0\rightarrow\mathcal H_i$. Clearly $\natural_i$ preserves  the canonical equivalences induced by restrictions and inclusions of intervals, and hence preserves path continuations.  $\natural_i$ also commutes with the action of $\mathcal A(I)$. Therefore $\natural_i$ intertwines the actions of $\mathcal A$. We thus conclude:

\begin{thm}\label{lb5}
There exists a (unique) unitary $\mathcal A$-module isomorphism $\natural_i:\mathcal H_i(I)\boxtimes\mathcal H_0\rightarrow \mathcal H_i$ satisfying
\begin{align*}
\natural_i(\xi\otimes y\Omega)=y\xi
\end{align*}
for any $\xi\in\mathcal H_i(I),y\in\mathcal A(I^c)$. Moreover $\natural_i$ preserves path continuations, i.e., $\natural_i\alpha^\bullet=\natural_i$ for any path $\alpha:[0,1]\rightarrow S^1$ from $I$ to another open interval in $S^1$.
\end{thm}
Similar results hold for $\mathcal H_i(I)\boxtimes\mathcal H_0(J)$.

\subsection{Conformal structures}

Let $\mathcal H_i,\mathcal H_j$ be $\mc A$-modules, and choose $I\in\mc J$. Then we know that $\mathcal H_i,\mathcal H_j,\mc H_i(I)\boxtimes\mc H_j$ are all conformal $\mc A$-modules. In this section we describe the action $U_{i\boxtimes j}$ of $\GA$ on $\mc H_i(I)\boxtimes\mc H_j$ in terms of those on $\mathcal H_i$ and $\mathcal H_j$. This result  will be used in the next chapter to study the conformal structures of categorical extensions. (See theorems \ref{lb16} and \ref{lb35}.)

Choose $g\in\GA$. Choose a path $\lambda:[0,1]\rightarrow\GA$ from $1$ to $g$ (i.e. $\lambda(0)=1$ and $\lambda(1)=g$). We require only that $\lambda$  descends to a continuous path $[\lambda]$ in $\scr G$; the continuity of $\lambda$ in $\GA$ is not necessary.  Note that the homotopy class of $[\lambda]$ is uniquely determined by $g$. Consider the action of $\scr G$ on $S^1$ lifted from $\mathscr G$ (and hence from $\Diffp(S^1)$). Choose arbitrary $z\in I$. Then the map $\lambda_z$ defined by
\begin{gather*}
\lambda_z:[0,1]\rightarrow S^1,\qquad t\mapsto \lambda(t)z
\end{gather*}
is a path from $I$ to $gI$ with initial point $z$ and end point $gz$. The homotopy class of $\lambda_z$ is clearly determined by that of $[\lambda]$ and hence by $g$. Thus,   by proposition \ref{lb4}, $\lambda_z^\bullet$ depends only on $g$ but not on the choice of $\lambda$. For instance, if $g\in\GA$ is a lift of $\wtd\exp(2i\pi L_0)\in\scr G$, then $\gamma_z^\bullet$ is the path continuation induced by an anticlockwise rotation by $2\pi$.

We now let $g$ act on $\mathcal H_i(I)\boxtimes\mathcal H_j$ by setting
\begin{align}
U_{i\boxtimes j}'(g)(\xi\otimes\eta)=(\lambda_z^\bullet)^{-1} \big(gZ(\xi,I)g^{-1}\Omega\otimes g\eta\big)\label{eq7}
\end{align}
for any $\xi\in\mathcal H_i(I),\eta\in\mathcal H_j$. (Here the actions of $g$ on $\mc H_i(I),\mc H_j$ are the standard ones.) Note that since $Z(\xi,I)$ intertwines the actions of $\mathcal A(I^c)$, $gZ(\xi,I)g^{-1}$ intertwines the actions of $g\mathcal A(I^c)g^{-1}=\mathcal A(gI^c)$. So $gZ(\xi,I)g^{-1}\in\Hom_{\mathcal A(gI^c)}(\mathcal H_0,\mathcal H_i)$. Accordingly, $(\lambda_z^\bullet)^{-1}:\mc H_i(gI)\boxtimes\mc H_j\xrightarrow{\simeq}\mc H_i(I)\boxtimes\mc H_j$ is the path continuation induced by the path $\lambda_z^{-1}$ from $gI$ to $I$. Set
\begin{align}
g\xi g^{-1}=gZ(\xi,I)g^{-1}\Omega,
\end{align}
which is a vector in $\mathcal H_i(gI)$. Equivalently,
\begin{align}
Z(g\xi g^{-1},gI)=gZ(\xi,I)g^{-1}.
\end{align}
Then \eqref{eq7} can be simplified as
\begin{align}
U_{i\boxtimes j}'(g)(\xi\otimes\eta)=(\lambda_z^\bullet)^{-1} \big(g\xi g^{-1}\otimes g\eta\big).\label{eq8}
\end{align}

Note that although \eqref{eq8} only gives a linear map $U_{i\boxtimes j}'(g):\mathcal H_i(I)\otimes\mathcal H_j\rightarrow\mathcal H_i(I)\boxtimes\mathcal H_j$, where $\mathcal H_i(I)\otimes\mathcal H_j$ (or more precisely, the quotient of $\mathcal H_i(I)\otimes\mathcal H_j$ over its subspace annihilated by $\bk{\cdot|\cdot}$) is regarded as a dense subspace of $\mathcal H_i(I)\boxtimes\mathcal H_j$, one computes, for any $\xi'\in\mathcal H_i(I),\eta'\in\mathcal H_j$, that
\begin{align*}
&\bk{U_{i\boxtimes j}'(g)(\xi\otimes\eta)|U_{i\boxtimes j}'(g)(\xi'\otimes\eta')}=\bk{g\xi g^{-1}\otimes g\eta|g\xi' g^{-1}\otimes g\eta'}\\
=&\bk{Z(g\xi'g^{-1},gI)^*Z(g\xi g^{-1},gI)g\eta|g\eta'}=\bk{gZ(\xi',I)^*Z(\xi,I)g^{-1}\cdot g\eta|g\eta'}\\
=&\bk{Z(\xi',I)^*Z(\xi,I)\eta|\eta'}=\bk{\xi\otimes\eta|\xi'\otimes\eta'}.
\end{align*}
Therefore $U_{i\boxtimes j}'(g)$ is an isometry. As the image of $\lambda_z^\bullet U_{i\boxtimes j}'(g)$ is clearly $\mathcal H_i(gI)\otimes\mathcal H_j$, which is dense in $\mathcal H_i(gI)\boxtimes\mathcal H_j$, $U_{i\boxtimes j}'(g)$ extends to a unitary map on  $\mathcal H_i(I)\boxtimes\mathcal H_j$.

\begin{lm}
If \eqref{eq8} defines a unitary representation $U_{i\boxtimes j}'$ of $\GA$ on $\mc H_i(I)\boxtimes\mc H_j$, then $U_{i\boxtimes j}'$ equals the standard one $U_{i\boxtimes j}$.
\end{lm}

\begin{proof}
Suppose that $U_{i\boxtimes j}'$ is a representation. By the uniqueness statement in theorem \ref{lb59}, it remains to check \eqref{eq15} for $U_{i\boxtimes j}'$, i.e., to check that for any $K\in\mathcal J$ and $g\in\GA(K)$,
\begin{align}
U_{i\boxtimes j}'(g)=\pi_{i\boxtimes j,K}(U(g))\label{eq17}
\end{align}
when acting on $\mathcal H_i(I)\boxtimes \mathcal H_j$.

Note first of all that from the definition \eqref{eq8}, it is clear that the action $U_{i\boxtimes j}'$ of $\GA$ on Connes fusions commutes with the canonical maps induced by restrictions and inclusions  of intervals. So it also commutes with path continuations. Therefore, by adjusting $I$, it suffices to prove \eqref{eq17} when $K\subset I^c$. In this case $g\in\GA(I^c)$, which implies that $g$ commutes with $Z(\xi,I)$ (i.e. $g\xi g^{-1}=\xi$), and that the path $\lambda_z$ considered in \eqref{eq8} is a constant. Hence, for any $\xi\in\mathcal H_i(I),\eta\in\mathcal H_j$,
\begin{align}
U_{i\boxtimes j,K}'(g)(\xi\otimes\eta)=\xi\otimes g\eta.\label{eq18}
\end{align}
One the other hand, from the canonical equivalence $\mathcal H_i(I)\boxtimes\mathcal H_j\xrightarrow\simeq\mathcal H_i(I)\boxtimes\mathcal H_j(I^c)$ and the way we define the action of $\mathcal A(K)$ on $\mathcal H_i(I)\boxtimes\mathcal H_j(I^c)$, one easily sees that $\pi_{i\boxtimes j}(U(g))(\xi\otimes\eta)$ also equals $\xi\otimes g\eta$. Hence \eqref{eq17} is proved.
\end{proof}

\begin{lm}\label{lb60}
\eqref{eq8} defines a unitary representation $U_{i\boxtimes j}'$ of $\GA$ on $\mc H_i(I)\boxtimes\mc H_j$. Namely, for any $\xi\in\mathcal H_i(I),\eta\in\mathcal H_j,g,h\in\GA$,
\begin{align}
U_{i\boxtimes j}'(g)U_{i\boxtimes j}'(h)(\xi\otimes\eta)=U_{i\boxtimes j}'(gh)(\xi\otimes\eta).\label{eq9}
\end{align}
\end{lm}

\begin{proof}
We write $U_{i\boxtimes j}'$ as $U'$ for brevity. By \cite{Hen19} lemma 17-(ii), $\scr G$ is (algebraically) 	generated by $\scr G(J)$ for all $J\in\mc J$ whose  length $|J|$  is less than $|I|$. Thus, $\GA$ is generated by  $\GA(J)$ for all $J$ satisfying $|J|<|I|$. Therefore, it suffices to verify \eqref{eq9} when $h$ belongs to $\GA(J)$ satisfying $|J|<|I|$. Choose $I_0\in\mc J$ to be a sub-interval of $I$ disjoint from $J$. Since $\mc H_i(I_0)\otimes\mc H_j$ is dense in $\mc H_i(I)\boxtimes\mc H_j$, it suffices to assume that $\xi\in\mc H_i(I_0)$. In that case, as argued near \eqref{eq18}, we have $h\xi h^{-1}=\xi$ and hence $U'(h)(\xi\otimes\eta)=\xi\otimes h\eta$. Therefore, we need to check 
\begin{align}
U'(g)(\xi\otimes h\eta)=U'(gh)(\xi\otimes\eta).\label{eq99}
\end{align}

Choose a path $\lambda$ in $\GA$ from $1$ to $g$. Again, we assume the continuity only for the projection $[\lambda]$ of $\lambda$ in $\scr G$. Choose $z\in I_0$. Then the left hand side of \eqref{eq99} is
\begin{align*}
U'(g)(\xi\otimes h\eta)=(\lambda_z^\bullet)^{-1}(g\xi g^{-1}\otimes gh\eta).
\end{align*}
Choose a path $\mu$ in $\GA(J)$ from $1$ to $h$. (This is possible since $\scr G(J)$ is clearly contractible.)  Then $\lambda\mu=\lambda(t)\mu(t)$ is a path in $\GA$ from $1$ to $gh$. Since $z\in I_0$ is outside $J$, $\mu(t)z=z$ for any $t\in [0,1]$. It follows that $(\lambda\mu)_z=\lambda_z$. Using $h\xi h^{-1}=\xi$, we compute the right hand side of \eqref{eq99}:
\begin{align*}
U'(gh)(\xi\otimes\eta)=((\lambda\mu)_z^\bullet)^{-1}(gh\xi h^{-1}g^{-1}\otimes gh\eta)=(\lambda_z^\bullet)^{-1}(g\xi g^{-1}\otimes gh\eta).
\end{align*}
This proves \eqref{eq99}.
\end{proof}

The above two lemmas imply the following main result of this section.

\begin{thm}
Let $\mathcal H_i,\mathcal H_j$ be  $\mathcal A$-modules. Then for any $I\in\mathcal J$,  the unitary representation $U_{i\boxtimes j}$ of $\GA$ defining the conformal structure of $\mc H_i(I)\boxtimes\mc H_j$ can be described as follows. For any $g\in\GA$, we choose a map $\lambda:[0,1]\rightarrow\GA$ satisfying $\lambda(0)=1,\lambda(1)=g$ such that $\lambda$ descends to a (continuous) path in $\scr G$. Choose any $z\in I$, and let $\lambda_z$ be the path $t\in[0,1]\mapsto \lambda(t)z$ in $S^1$. Then for any $\xi\in\mathcal H_i(I),\eta\in\mathcal H_j$,
\begin{align}
g(\xi\otimes\eta)=(\lambda_z^\bullet)^{-1} \big(g\xi g^{-1}\otimes g\eta\big)\label{eq25}
\end{align}
where $g\xi g^{-1}=gZ(\xi,I)g^{-1}\Omega$, and $\lambda_z^\bullet:\mc H_i(I)\boxtimes\mc H_j\rightarrow \mc H_i(gI)\boxtimes\mc H_j$ is the path continuation induced by $\lambda_z$.
\end{thm}

Using \eqref{eq25}, one can easily describe the M\"obius structure on $\mathcal H_i(I)\boxtimes\mathcal H_j$. For any $\mc A$-module $\mathcal H_k$ and a self-adjoint vector field $X=\sum_{s=1,0,-1}a_sL_s$, we let the operator  $e^{iX}$ denote the action of $\wtd\exp(iX)\in\wtd\PSU(1,1)$ on $\mathcal H_k$. Now define the path $\lambda_X:[0,1]\rightarrow\wtd\PSU(1,1)$ to be $\lambda_X(t)=\wtd\exp(itX)$. For an arbitrary $z\in I$, let $\lambda_{X,z}=(\lambda_X)_z$ be the path $t\in[0,1]\mapsto \lambda_X(t)z$. Then for any $\xi\in\mathcal H_i(I),\eta\in\mathcal H_j$, we have the formula
\begin{align}
e^{iX}(\xi\otimes\eta)=(\lambda_{X,z}^\bullet)^{-1}(e^{iX}\xi\otimes e^{iX}\eta).\label{eq26}
\end{align}
Here we use the fact that $e^{iX}\xi=e^{iX}\xi e^{-iX}$ since $\Omega$ is fixed by $\PSU(1,1)$.

The action of $\GA$ on $\mathcal H_i\boxtimes\mathcal H_j(J)$ can be described in a similar way.

\subsection{Associativity}

In this section we study Connes fusion of more than two $\mathcal A$-modules. For simplicity, our discussion is restricted to the case of 3 modules. The general cases can be treated in a similar way.

We first discuss Connes fusions over three disjoint intervals. There are two equivalent ways to define them. Let $\mathcal H_i,\mathcal H_j,\mathcal H_k$ be $\mathcal A$-modules, and $I,J,K$ be disjoint open intervals in $S^1$.  Let $\bk{\cdot|\cdot}$ be a positive sesquilinear form  on $\mathcal H_i(I)\otimes\mathcal H_j(J)\otimes\mathcal H_j(K)$ satisfying
\begin{align}
&\bk{\xi_1\otimes\eta_1\otimes\chi_1|\xi_2\otimes\eta_2\otimes\chi_2}\nonumber\\
=&\bk{Z(\xi_2,I)^*Z(\xi_1,I)Z(\eta_2,J)^*Z(\eta_1,J)Z(\chi_2,K)^*Z(\chi_1,K)\Omega|\Omega}\label{eq20}
\end{align}
for any $\xi_1,\xi_2\in\mathcal H_i(I),\eta_1,\eta_2\in\mathcal H_j(J),\chi_1,\chi_2\in\mathcal H_j(K)$. The \textbf{Connes fusion} $\mathcal H_i(I)\boxtimes\mathcal H_j(J)\boxtimes\mathcal H_k(K)$ is defined to be the Hilbert space completion of $\mathcal H_i(I)\otimes\mathcal H_j(J)\otimes\mathcal H_k(K)$ under $\bk{\cdot|\cdot}$. Canonical equivalences induced by restrictions and inclusions of open intervals, and path continuations can be defined in a similar way. We have  natural actions of $\mathcal A(I),\mathcal A(J),\mathcal A(K)$ on $\mathcal H_i(I)\boxtimes\mathcal H_j(J)\boxtimes\mathcal H_k(K)$. These actions can be extended to a representation of $\mathcal A$ using path continuations.

One can also defined fusions of three modules as iterated fusions of two modules. For example,  consider $\mathcal H_i(I)\boxtimes(\mathcal H_j(J)\boxtimes\mathcal H_k(K))$. This expression is a combination of a fusion over two intervals  with a fusion over one interval: we first take a fusion of $\mathcal H_j,\mathcal H_k$ over $J,K$, and treat this fusion as a single $\mathcal A$-module $\mathcal H_1=\mathcal H_j(J)\boxtimes\mathcal H_k(K)$; then we take $\mathcal H_i(I)\boxtimes\mathcal H_1$ as a fusion over $I$. It is easy to check that the right hand side of \eqref{eq20} also gives the formula for the positive sesquilinear form of the iterated Connes fusion. Therefore we have a unitary map
\begin{gather}
\mathcal H_i(I)\boxtimes(\mathcal H_j(J)\boxtimes\mathcal H_k(K))\rightarrow\mathcal H_i(I)\boxtimes\mathcal H_j(J)\boxtimes\mathcal H_k(K),\nonumber\\
\xi\otimes(\eta\otimes\chi)\mapsto \xi\otimes\eta\otimes\chi\qquad(\xi\in\mathcal H_i(I),\eta\in\mathcal H_j(J),\eta\in\mathcal H_k(K)).\label{eq21}
\end{gather}
Similarly one can define $(\mathcal H_i(I)\boxtimes\mathcal H_j(J))\boxtimes\mc H_k$, and  an equivalence $(\mathcal H_i(I)\boxtimes\mathcal H_j(J))\boxtimes\mathcal H_k(K)\xrightarrow{\simeq}\mathcal H_i(I)\boxtimes\mathcal H_j(J)\boxtimes\mathcal H_k(K)$ mapping $(\xi\otimes\eta)\otimes\chi$ to $\xi\otimes\eta\otimes\chi$. Therefore we have a natural unitary associativity map $(\mc H_i(I)\boxtimes\mc H_j(J))\boxtimes\mc H_k(K)\xrightarrow\simeq \mc H_i(I)\boxtimes(\mc H_j(J)\boxtimes\mc H_k(K))$.

For Connes fusions of three modules over two intervals, one also has similar isomorphisms between $\mathcal H_i(I)\boxtimes\mathcal H_j\boxtimes\mathcal H_k(K)$, $(\mathcal H_i(I)\boxtimes\mathcal H_j)\boxtimes\mathcal H_k(K)$, and $\mathcal H_i(I)\boxtimes(\mathcal H_j\boxtimes\mathcal H_k(K))$. Here, the second and the third fusions are iterations of two fusions over single intervals.

We now show that associativity maps are $\mc A$-module isomorphisms. We only prove this for fusions over two intervals. The three-interval cases can be proved in a similar way.  To show this, note that the above isomorphisms clearly commute with the actions of $\mc A(I)$. Hence it suffices to prove that they also commute with path continuations, as indicated by the following proposition.

\begin{pp}\label{lb53}
Let $\mc H_i,\mc H_j,\mc H_k$ be $\mc A$-modules. Choose two pairs of disjoint intervals $I,J\in\mc J,I',J'\in\mc J$. Let $\gamma=(\alpha,\beta):[0,1]\mapsto\Conf_2(S^1)$ be a path from $I\times J$ to $I'\times J'$. Then the following diagrams commute.
\begin{align}
\begin{CD}
(\mc H_i(I)\boxtimes\mc H_k)\boxtimes\mc H_j(J) @>\quad \beta^\bullet(\alpha^\bullet\otimes\id_j) \quad>> (\mc H_i(I')\boxtimes\mc H_k)\boxtimes\mc H_j(J')\\
@V\simeq VV @V\simeq VV\\
\mc H_i(I)\boxtimes\mc H_k\boxtimes\mc H_j(J) @>~~\qquad\gamma^\bullet~~\qquad>> \mc H_i(I')\boxtimes\mc H_k\boxtimes\mc H_j(J')\\
@V\simeq VV @V\simeq VV\\
\mc H_i(I)\boxtimes(\mc H_k\boxtimes\mc H_j(J)) @>\quad \alpha^\bullet(\id_i\otimes\beta^\bullet) \quad>> \mc H_i(I')\boxtimes(\mc H_k\boxtimes\mc H_j(J'))
\end{CD}\label{eq37}
\end{align}
\end{pp}
We remark that $\beta^\bullet$ commutes with $\alpha^\bullet\otimes\id_j$ and $\alpha^\bullet$ commutes with $\id_i\otimes\beta^\bullet$ by the functoriality of path continuations (see proposition \ref{lb8}).
\begin{proof}
We only prove the commutativity of the first diagram, as  the second one can be proved similarly. Let us first assume that $\gamma=(\alpha,\beta)$ is small in the sense that $\gamma([0,1])\subset I'\times J'$, and $I\cap I',J\cap J'\in\mc J$. Then it is easy to verify the  commutativity of the first diagram by considering the actions of these maps on any $(\xi\otimes\chi)\otimes\eta$, where $\xi\in\mc H_i(I\cap I'),\chi\in\mc H_k,\eta\in\mc H_j(J\cap J')$.

In the general case, we can divide $\gamma=(\alpha,\beta)$ into small paths in $\Conf_2(S^1)$: $\gamma=\gamma_n*\gamma_{n-1}*\cdots*\gamma_1$, and choose pairs of disjoint open intervals $I_0,J_0\in\mc J,I_1,J_1\in\mc J,\cdots, I_n,J_n\in\mc J$, such that $I_0=I,J_0=J,I_n=I',J_n=J'$, and that  for any $s=1,2,\dots,n$, we have $I_{s-1}\cap I_s,J_{s-1}\cap J_s\in\mc J$, and $\gamma_s([0,1])\subset I_s\times J_s$. Write $\gamma_s=(\alpha_s,\beta_s)$. Then by the first paragraph, for each $s$ the diagram
\begin{align*}
\begin{CD}
(\mc H_i(I_{s-1})\boxtimes\mc H_k)\boxtimes\mc H_j(J_{s-1}) @>\quad \beta_s^\bullet(\alpha_s^\bullet\otimes\id_j) \quad>> (\mc H_i(I_s)\boxtimes\mc H_k)\boxtimes\mc H_j(J_s)\\
@V\simeq VV @V\simeq VV\\
\mc H_i(I_{s-1})\boxtimes\mc H_k\boxtimes\mc H_j(J_{s-1}) @>~~\qquad\gamma^\bullet_s~~\qquad>> \mc H_i(I_s)\boxtimes\mc H_k\boxtimes\mc H_j(J_s)
\end{CD}
\end{align*}
commutes. Since $\beta_t^\bullet$ commutes  with $\alpha_s^\bullet\otimes\id_j$ for any $1\leq s,t\leq n$, and since $\alpha^\bullet=\alpha_n^\bullet\cdots\alpha_1^\bullet,\beta^\bullet=\beta_n^\bullet\cdots\beta_1^\bullet,\gamma^\bullet=\gamma_n^\bullet\cdots\gamma_1^\bullet$,  the  commutativity of the first diagram of \eqref{eq37} follows.
\end{proof}

\subsection{$C^*$-tensor categories}\label{lb10}

Let $\Rep(\mathcal A)$ be the $C^*$-category of $\mathcal A$-modules. In this section, we equip $\Rep(\mathcal A)$ with a unitary monoidal structure. More precisely, we want to define a tensor (fusion) $*$-bifunctor $\boxtimes:\Rep(\mathcal A)\times\Rep(\mathcal A)\rightarrow\Rep(\mathcal A)$, define unitary associativity isomorphisms which are functorial with respect to the tensor bifunctor, define a unit object, identify (unitarily) a module with the tensor (fusion) product of this module with the unit, and verify the triangle and pentagon axioms. (See \cite{Tur94,BK01,EGNO} for the general theory of tensor categories.)

Let $S^1_+=\{z\in S^1:\mathrm{Im} z>0 \}$, $S^1_-=\{z\in S^1:\mathrm{Im} z<0 \}$. For any $\mathcal H_i,\mathcal H_j\in\Rep(\mathcal A)$, we define their tensor product $\mathcal H_i\boxtimes\mathcal H_j$ to be $\mathcal H_i(S^1_+)\boxtimes\mathcal H_j(S^1_-)$. We also identify $\mathcal H_i\boxtimes\mathcal H_j$ with $\mathcal H_i(S^1_+)\boxtimes\mathcal H_j$ and $\mathcal H_i\boxtimes\mathcal H_j(S^1_-)$ through the canonical equivalences. Let us simplify our notations by writing $S^1_+$ and $S^1_-$ as $+$ and $-$ in Connes fusions. Then by our definition,

\begin{align}
\mathcal H_i\boxtimes\mathcal H_j:=\mathcal H_i(+)\boxtimes\mathcal H_j(-)=\mathcal H_i(+)\boxtimes\mathcal H_j=\mathcal H_i\boxtimes\mathcal H_j(-).
\end{align}
Since this definition of tensor bifunctor relies on two fixed open intervals, we do not have a natural identification of $\mathcal H_i\boxtimes\mathcal H_j$ and $\mathcal H_j\boxtimes\mathcal H_i$. If $F\in\Hom_{\mathcal A}(\mathcal H_i,\mathcal H_j),G\in\Hom(\mathcal H_{i'},\mathcal H_{j'})$, then the tensor product $F\otimes G:\mathcal H_i\boxtimes\mathcal H_j\rightarrow \mathcal H_{i'}\boxtimes\mathcal H_{j'}$ is defined using \eqref{eq19}. That $(F\otimes G)^*=F^*\otimes G^*$ is easy to verify, which shows that the bifunctor $\boxtimes$ preserves the $*$-structures.

For $\mathcal H_i,\mathcal H_j,\mathcal H_k\in\Rep(\mathcal A)$,\footnote {Although $\Rep(\mc A)$ is not a set, we still write $\mc H_i\in\Rep(\mc A)$ to mean that $\mc H_i$ is an object in $\Rep(\mc A)$. We will use similar notations for other categories.} we define the associativity isomorphism $(\mathcal H_i\boxtimes\mathcal H_j)\boxtimes\mathcal H_k\xrightarrow\simeq\mathcal H_i\boxtimes(\mathcal H_j\boxtimes\mathcal H_k)$ to be the one for
\begin{align}
(\mathcal H_i(+)\boxtimes\mathcal H_j)\boxtimes\mathcal H_k(-)\xrightarrow\simeq\mathcal H_i(+)\boxtimes(\mathcal H_j\boxtimes\mathcal H_k(-)),\label{eq22}
\end{align}
which is clearly functorial. The pentagon axiom (see figure \ref{fig3}) holds , since it can easily be verified for all $\xi^{(i)}\otimes\xi^{(j)}\otimes\xi^{(k)}\otimes\xi^{(l)}$, where  $\xi^{(i)}\in\mathcal H_i(S^1_+),\xi^{(j)}\in\mathcal H_j(S^1_+),\xi^{(k)}\in\mathcal H_k(S^1_-),\xi^{(l)}\in\mathcal H_l(S^1_-)$. (Note that $\xi^{(i)}\otimes\xi^{(j)}\in(\mc H_i\boxtimes\mc H_j)(S^1_+),\xi^{(k)}\otimes\xi^{(l)}\in(\mc H_k\boxtimes\mc H_l)(S^1_-)$.) (cf. \cite{Loke94} lemma VI.5.5.1.)
\begin{figure}[h]
	\centering
	\includegraphics[width=1\linewidth]{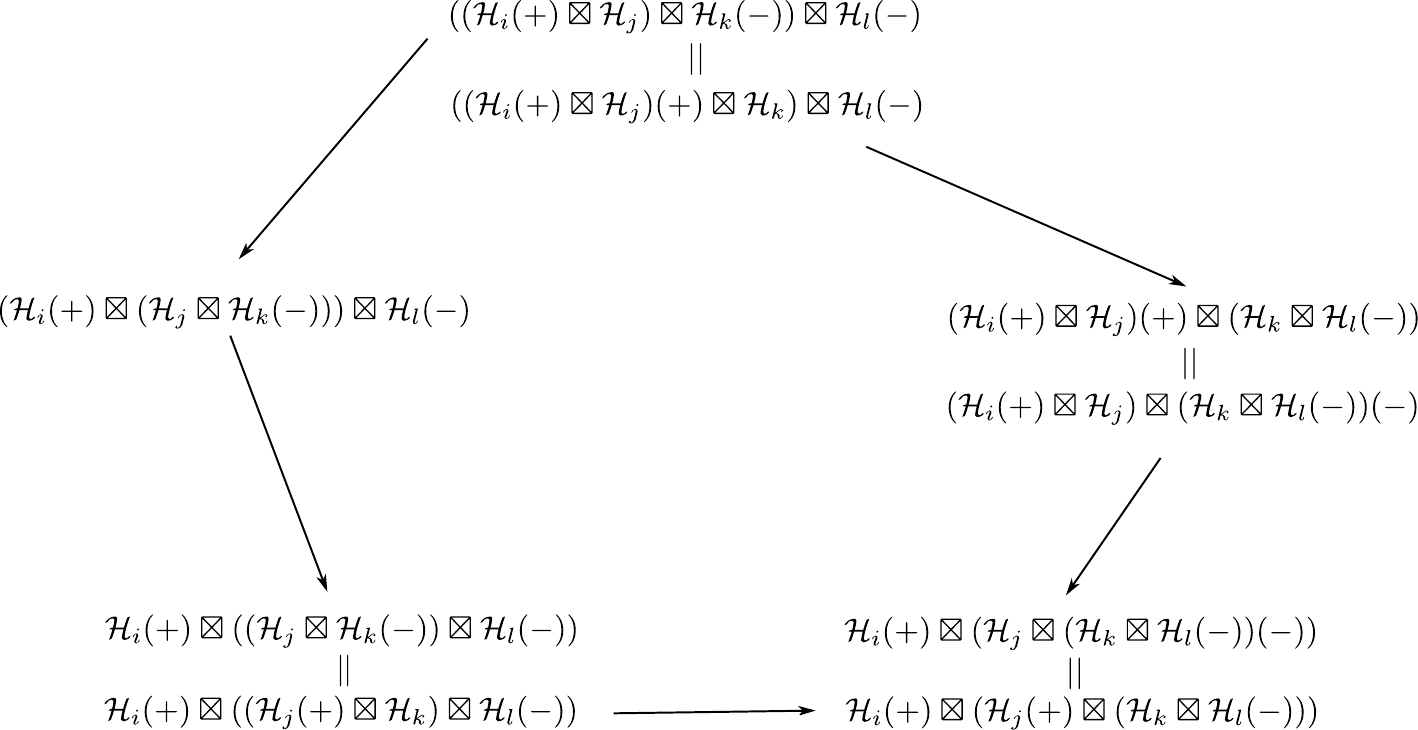}
	\caption{}
	\label{fig3}
\end{figure}
One can thus remove all the brackets in an iterated fusion. For example, all the iterated fusions in figure \ref{fig3} can be identified as $\mathcal H_i\boxtimes\mathcal H_j\boxtimes\mathcal H_k\boxtimes\mathcal H_l$.

Let $\mc H_0$ be the unit object.  By theorem \ref{lb5}, we have unitary isomorphisms
\begin{gather*}
\sharp_i:\mathcal H_i\boxtimes\mathcal H_0=\mathcal H_i(S^1_+)\boxtimes\mathcal H_0\xrightarrow{\natural_i}\mc H_i,\\
\flat_j:\mc H_0\boxtimes\mc H_j=\mc H_0\boxtimes\mc H_j(S^1_-)=\mc H_j(S^1_-)\boxtimes\mc H_0\xrightarrow{\natural_j}\mc H_j.
\end{gather*}
The triangle axiom says that
\begin{align}
\sharp_i\otimes \id_{j}=\id_{i}\otimes \flat_j, \label{eq23}
\end{align}
where $\id_i$ (resp. $\id_j$) is the identity operator on $\mc H_i$ (resp. $\mc H_j$). To see this, note that both sides of \eqref{eq23} act on $\mc H_i\boxtimes \mc H_0\boxtimes \mc H_j=\mc H_i(+)\boxtimes \mc H_0\boxtimes \mc H_j(-)$. Let us choose $I\sjs S^1_+,J\sjs S^1_-$, and $K\in\mc J$ disjoint from $I,J$. Then for any $\xi\in\mc H_i(I),\eta\in\mc H_j(J),x\in\mc A(K)$, one computes that
\begin{gather*}
(\sharp_i\otimes\id_j)(\xi\otimes x\Omega\otimes\eta)=x\xi\otimes\eta\in\mc H_i\boxtimes\mc H_j(-),\\
(\id_i\otimes\flat_j)(\xi\otimes x\Omega\otimes\eta)=\xi\otimes x\eta\in\mc H_i(+)\boxtimes\mc H_j.
\end{gather*} 
One can easily construct path continuations to show that both $x\xi\otimes\eta$ and $\xi\otimes x\eta$ equal $x(\xi\otimes\eta)$. Therefore \eqref{eq23} holds.  A construction of  $C^*$-tensor categorical structure on $\Rep(\mc A)$ is now finished.\\

We close this section with a brief discussion of braiding in $\Rep(\mathcal A)$. Let $\varrho$ be a path of $180^\degree$ clockwise rotation from $S^1_+$ to $S^1_-$, e.g.,
\begin{gather}
\varrho:[0,1]\rightarrow S^1,\qquad t\mapsto e^{i\pi(\frac 12-t)}. \label{eq24}
\end{gather}
The braid operator $\mathbb B_{i,j}:\mc H_i\boxtimes\mc H_j\rightarrow\mc H_j\boxtimes\mc H_i$ is defined to be
\begin{align}
\mathbb B_{i,j}: \mc H_i\boxtimes\mc H_j= \mc H_i(S^1_+)\boxtimes\mc H_j\xrightarrow{\varrho^\bullet}\mc H_i(S^1_-)\boxtimes\mc H_j=\mc H_j(S^1_+)\boxtimes\mc H_i(S^1_-)=\mc H_j\boxtimes\mc H_i,
\end{align}
or written more simply, $\mathbb B=\varrho^\bullet$. This unitary map is clearly functorial. So $\Rep(\mc A)$  becomes a braided $C^*$-tensor category once we've proved the hexagon axioms for $\mathbb B_{i,j}$. The proof of hexagon axioms will be much easier after we introduce categorical extensions. So we leave the proof to the next chapter (see section \ref{lb11}). 

Note that our description of the $C^*$-tensor category $\Rep(\mc A)$ and the braiding  $\mathbb B$ (as well as the Hexagon axioms to be proved in the next chapter) do not rely on the conformal structures of $\mc A$ and $\mc A$-modules. Thus the above results hold when $\mc A$ is only a M\"obius covariant net. In \cite{Was98} section 33, A.Wassermann defines braiding in a different way (see equation \eqref{eq89}), which relies on the conformal structures. In the following we show that our definition of $\mathbb B$ agrees with that of Wassermann. Although this result is interesting in its own right, we will not use it in the rest of this paper. The following discussion can be skipped safely.

Choose  $\mc H_i,\mc H_j\in\Rep(\mc A)$.  Consider a path $\lambda$ in $\wtd\PSU(1,1)$ defined by $t\in[0,1]\mapsto\wtd\exp(-i\pi tL_0)$. Set $z=i$, and $\lambda_z:t\in [0,1]\mapsto \lambda(t)z\in S^1$. Then $\lambda_z=\varrho$ where $\varrho$ is defined by \eqref{eq24}. By \eqref{eq26}, for any $\xi\in\mathcal H_i(S^1_+),\eta\in\mathcal H_j(S^1_-)$, 
\begin{align}
(\lambda_z^\bullet)^{-1}(\eta\otimes\xi)=e^{-i\pi L_0}(e^{i\pi L_0}\eta\otimes e^{i\pi L_0}\xi)\label{eq27}
\end{align}
where we regard $\eta\otimes\xi\in\mc H_j(S^1_-)\boxtimes\mathcal H_i(S^1_+)$ and  $e^{i\pi L_0}\eta\otimes e^{i\pi L_0}\xi\in\mc H_j(S^1_+)\boxtimes\mc H_i(S^1_-)$. Under the identification $\mc H_j(S^1_-)\boxtimes\mc H_i(S^1_+)=\mc H_i(S^1_+)\boxtimes\mc H_j(S^1_-)=\mc H_i\boxtimes\mc H_j$ and $\mc H_j(S^1_+)\boxtimes\mc H_i(S^1_-)=\mc H_j\boxtimes\mc H_i$,  \eqref{eq27} becomes
\begin{align}
(\lambda_z^\bullet)^{-1}(\xi\otimes\eta)=e^{-i\pi L_0}(e^{i\pi L_0}\eta\otimes e^{i\pi L_0}\xi)
\end{align}
where we regard $\xi\otimes\eta\in\mc H_i\boxtimes\mc H_j,e^{i\pi L_0}\eta\otimes e^{i\pi L_0}\xi\in\mc H_j\boxtimes\mc H_i$. As $\lambda_z=\varrho$, we have $\mathbb B=\lambda_z^\bullet$. We thus conclude:

\begin{pp}\label{lb52}
If $\mc H_i,\mathcal H_j$ are conformal $\mc A$-modules, then the inverse of $\mathbb B_{j,i}:\mc H_j\boxtimes\mc H_i\rightarrow\mc H_i\boxtimes\mc H_j$ can be described by
\begin{align}\label{eq89}
\mathbb B_{j,i}^{-1}(\xi\otimes\eta)=e^{-i\pi L_0}(e^{i\pi L_0}\eta\otimes e^{i\pi L_0}\xi)\qquad(\forall \xi\in\mc H_i(S^1_+),\eta\in\mc H_j(S^1_-)).
\end{align}
\end{pp}

\section{Connes fusions and categorical extensions}

\subsection{Categorical extensions}\label{lb29}

Let $\mc A$ be a conformal net as usual. Let $\scr C$ be a full abelian ($C^*$-)subcategory of $\Rep(\mc A)$ containing $\mc H_0$. In other words, $\scr C$ is a class of $\mc A$-modules which, up to unitary equivalences, is closed under taking $\mc A$-submodules and finite direct sums.\footnote{Indeed it is not necessary to require $\scr C$ to be a full subcategory. We add this condition only to simplify discussions.}  Equip $\scr C$  with a tensor bifunctor $\boxdot$ (not necessarily the Connes fusion bifunctor $\boxtimes$), functorial unitary associativity isomorphisms $(\mc H_i\boxdot\mc H_j)\boxdot H_k\xrightarrow\simeq\mc H_i\boxdot(\mc H_j\boxdot\mc H_k)$, and unitary isomorphisms $\mc H_i\boxdot\mc H_0\xrightarrow\simeq\mc H_i,\mc H_0\boxdot\mc H_i\xrightarrow\simeq \mc H_i$, so that $\scr C$ becomes a $C^*$-tensor category with unit $\mc H_0$. We identify $(\mc H_i\boxdot\mc H_j)\boxdot H_k$, $\mc H_i\boxdot(\mc H_j\boxdot\mc H_k)$ as $\mc H_i\boxdot\mc H_j\boxdot \mc H_k$, and $\mc H_i\boxdot \mc H_0$, $\mc H_0\boxdot\mc H_i$ as $\mc H_i$.

An \textbf{arg function} defined on $I\in\mc J$ is a continuous function $\arg_I:I\rightarrow \mathbb R$ such that $z=e^{\log|z|+i\arg_I(z)}$ for any $z\in I$. We say that an open interval $I\in\mc J$ is \textbf{arg-valued}, if $I$ is equipped with an arg function $\arg_I$. An arg-valued $I$ is often denoted by $\wtd I$ or $(I,\arg_I)$. Two identical open intervals are regarded as different arg-valued intervals if they have different arg functions. We let $\Jtd$ be the set of all arg-valued open intervals in $S^1$. If $\wtd I=(I,\arg_I),\wtd J=(J,\arg_J)\in\Jtd$, we say that $\wtd I$ is an \textbf{(arg-valued) open subinterval} of $\wtd J$ if $I\subset J$ and $\arg_J|_I=\arg_I$. In this case we write  $\wtd I\subset\wtd J$. Given $\wtd J\in\Jtd$, we say $\wtd I\sjs \wtd J$ if $\wtd I\in\Jtd$, $\wtd I\subset\wtd J$, and $I\sjs J$. We say that $\wtd I$ and $\wtd J$ are \textbf{disjoint} if $I$ and $J$ are disjoint.  We say  $\wtd I$ is \textbf{anticlockwise} to $\wtd J$ (or $\wtd J$ is clockwise to $\wtd I$), if $\wtd I$ and $\wtd J$ are disjoint, and $\arg_J(\zeta)<\arg_I(z)<\arg_J(\zeta)+2\pi$ for any $z\in I,\zeta\in J$.

The group $\scr G$ acts on $\Jtd$ in a natural way: If $g\in\scr G$, $\wtd I=(I,\arg_I)\in\Jtd$, we choose a path $\lambda$ in $\scr G$ from $1$ to $g$. Then for any $z\in I$, we have a path $\lambda_z$ in $S^1$ defined by $\lambda_z(t):t\in[0,1]\mapsto\lambda(t)z$. Let us now extend $\arg_I(z)$ continuously to an argument of $gz$ along the path $\lambda_z$. More precisely, we define a continuous function $\arg_{\lambda_z}:[0,1]\rightarrow\mathbb R$, such that for any $t\in[0,1]$, $\arg_{\lambda_z}(t)$ is an argument of $\lambda_z(t)=\lambda(t)z$. Then we take $\arg_{\lambda_z}(1)$ to be the argument of $gz$. Now we let $\arg_{gI}$  be the arg function on $gI$ satisfying $\arg_{gI}(gz)=\arg_{\lambda_z}(1)$. We define $g\wtd I=(gI,\arg_{gI})$. It is easy to check that this definition is well defined, and that the action of $\scr G\curvearrowright\Jtd$ is a group action. We can easily lift this action to $\GA\curvearrowright \Jtd$.

If $\mc P,\mc Q, \mc R,\mc S$ are Hilbert spaces, and we have bounded linear operators $A:\mc P\rightarrow\mc R,B:\mc Q\rightarrow\mc S,C:\mc P\rightarrow\mc Q,D:\mc R\rightarrow\mc S$, we say that the  diagram
\begin{align}
\begin{CD}
\mc P @>C>> \mc Q\\
@V A VV @V B VV\\
\mc R @>D>> \mc S
\end{CD}\label{eq28}
\end{align}
\textbf{commutes adjointly}, if both this diagram and the diagram
\begin{align}
\begin{CD}
\mc P @>C>> \mc Q\\
@A {A^*} AA @A {B^*} AA\\
\mc R @>D>> \mc S
\end{CD}\label{eq29}
\end{align}
commute. Note that \eqref{eq29} commutes if and only if the following diagram \eqref{eq30} commutes.
\begin{align}
\begin{CD}
\mc P @<C^*<< \mc Q\\
@V A VV @V B VV\\
\mc R @<D^*<< \mc S
\end{CD}\label{eq30}
\end{align}
Note also that if either $A,B$  or $C,D$ are unitary, then the commutativity implies the adjoint commutativity of \eqref{eq28}.

\begin{df}\label{lb9}
Let $\fk H$ assign, to each $\wtd I\in\Jtd$ and $\mc H_i\in\scr C$, a set $\fk H_i(\wtd I)$  such that $\fk H_i(\wtd I_1)\subset\fk H_i(\wtd I_2)$  whenever $\wtd I_1\subset\wtd I_2$. A \textbf{categorical extension} $\scr E=(\mc A,\scr C,\boxdot,\fk H)$ of $\mc A$ associates, to any $\mc H_i,\mc H_k\in\scr C,\wtd I\in\Jtd,\fk a\in\fk H_i(\wtd I)$, bounded linear operators
\begin{gather*}
L(\fk a,\wtd I)\in\Hom_{\mc A(I^c)}(\mc H_k,\mc H_i\boxdot\mc H_k),\\
R(\fk a,\wtd I)\in\Hom_{\mc A(I^c)}(\mc H_k,\mc H_k\boxdot\mc H_i),
\end{gather*}
such that the following conditions are satisfied:\\
(a) (Isotony) If $\wtd I_1\subset\wtd I_2\in\Jtd$, and $\fk a\in\fk H_i(\wtd I_1)$, then $L(\fk a,\wtd I_1)=L(\fk a,\wtd I_2)$, $R(\fk a,\wtd I_1)=R(\fk a,\wtd I_2)$ when acting on any $\mc H_k\in\scr C$.\\
(b) (Functoriality) If $\mc H_{i},\mc H_k,\mc H_{k'}\in\scr C$, $F\in\Hom_{\mc A}(\mc H_k,\mc H_{k'})$,  the following diagrams commute for any $\wtd I\in\Jtd,\fk a\in\fk H_i(\wtd I)$.
\begin{gather}
\begin{CD}
\mc H_k @>F>> \mc H_{k'}\\
@V L(\fk a,\wtd I)  VV @V L(\fk a,\wtd I)  VV\\
\mc H_i\boxdot\mc H_k @> \id_i\otimes F>> \mc H_i\boxdot\mc H_{k'}
\end{CD}\qquad\qquad
\begin{CD}
\mc H_k @> R(\fk a,\wtd I)  >> \mc H_k\boxdot\mc H_i\\
@V F VV @V F\otimes\id_i  VV\\
\mc H_{k'} @>R(\fk a,\wtd I) >> \mc H_{k'}\boxdot\mc H_i
\end{CD}.
\end{gather}
(c) (Neutrality) For any $\mc H_i\in\scr C$, under the identifications $\mc H_i=\mc H_i\boxdot\mc H_0=\mc H_0\boxdot\mc H_i$, the relation
\begin{align}
L(\fk a,\wtd I)\Omega=R(\fk a,\wtd I)\Omega
\end{align}
holds for any $\wtd I\in\Jtd,\fk a\in\fk H_i(\wtd I)$.\\
(d) (Reeh-Schlieder property) If $\mc H_i\in\scr C,\wtd I\in \Jtd$, then under the identification $\mc H_i=\mc H_i\boxdot\mc H_0$, the set $L(\fk H_i(\wtd I),\wtd I)\Omega$ spans a dense subspace of $\mc H_i$.\\
(e) (Density of fusion products) If $\mc H_i,\mc H_k\in\scr C,\wtd I\in\Jtd$, then the set $L(\fk H_i(\wtd I),\wtd I)\mc H_k$ spans a dense subspace of $\mc H_i\boxdot\mc H_k$, and $R(\fk H_i(\wtd I),\wtd I)\mc H_k$ spans a dense subspace of $\mc H_k\boxdot\mc H_i$.\\
(f) (Locality) For any $\mc H_k\in\scr C$, disjoint $\wtd I,\wtd J\in\Jtd$ with $\wtd I$ anticlockwise to $\wtd J$, and any $\fk a\in\fk H_i(\wtd I),\fk b\in\fk H_j(\wtd J)$, the following diagram \eqref{eq31}  commutes \emph{adjointly}.
\begin{align}
\begin{CD}
\mc H_k @> \quad R(\fk b,\wtd J)\quad   >> \mc H_k\boxdot\mc H_j\\
@V L(\fk a,\wtd I)   V  V @V L(\fk a,\wtd I) VV\\
\mc H_i\boxdot\mc H_k @> \quad R(\fk b,\wtd J) \quad  >> \mc H_i\boxdot\mc H_k\boxdot\mc H_j
\end{CD}\label{eq31}
\end{align}
(g) (Braiding) There is a unitary linear map $\ss_{i,j}:\mc H_i\boxdot\mc H_j\rightarrow\mc H_j\boxdot \mc H_i$  for any $\mc H_i,\mc H_j\in\scr C$, such that  
\begin{align}
\ss_{i,j} L(\fk a,\wtd I)\eta=R(\fk a,\wtd I)\eta\label{eq33}
\end{align}
whenever $\wtd I\in\Jtd,\fk a\in\fk H_i(\wtd I)$, $\eta\in\mc H_j$.
\end{df}

\begin{df}
A categorical extension $\scr E$ is called \textbf{conformal (covariant)}, if for any $g\in\GA,\wtd I\in\Jtd,\mc H_i\in\scr C,\fk a\in\fk H_i(\wtd I)$, there exits an element $g\fk ag^{-1}\in \fk H_i(g\wtd I)$, such that
\begin{gather}
L(g\fk a g^{-1},g\wtd I)=gL(\fk a,\wtd I)g^{-1}\label{eq32}
\end{gather}
when acting on any $\mc H_j\in\scr C$.
\end{df}
Note that \eqref{eq32} is equivalent to
\begin{align}
R(g\fk a g^{-1},g\wtd I)=gR(\fk a,\wtd I)g^{-1}
\end{align}
by relation \eqref{eq33} and corollary \ref{lb7}.\\

We now derive some immediate consequences from the definition of a categorical extension $\scr E$. The first thing to notice is that for any $\mc H_i\in\scr C,\wtd I\in\Jtd$, and $\fk a\in\fk H_i(\wtd I)$, the operator $L(\fk a,\wtd I)$ (acting on all possible $\mc A$-modules in $\scr C$) is uniquely determined by the vector $L(\fk a,\wtd I)\Omega$. To see this, we choose an arbitrary $\mc H_j\in\scr C$, and choose $\wtd J\in\Jtd$ clockwise to $\wtd I$.
Then by locality, $L(\fk a,\wtd I)R(\fk b,\wtd J)\Omega=R(\fk b,\wtd J)L(\fk a,\wtd I)\Omega$ for any $\fk b\in\fk H_j(\wtd J)$. So the action of $L(\fk a,\wtd I)$ on $R(\fk H_j(\wtd J),\wtd J)\Omega$ is determined by $L(\fk a,\wtd I)\Omega$. By Reeh-Schlieder property, $R(\fk H_j(\wtd J),\wtd J)\Omega=L(\fk H_j(\wtd J),\wtd J)\Omega$ spans a dense subspace of $\mc H_i$. Therefore $L(\fk a,\wtd I)$ is uniquely determined by $L(\fk a,\wtd I)\Omega$. Note also that $L(\fk a,\wtd I)\Omega\in\mc H_i(I)$. Hence we may relabel $\scr E$ to satisfy the following condition:

\begin{df}\label{lb14}
A categorical extension $\scr E=(\mc A,\scr C,\boxdot,\fk H)$ is called \textbf{vector-labeled}, if for any $\mc H_i\in\scr C,\wtd I\in\Jtd$, the set $\fk H_i(\wtd I)$ is a subset of $\mc H_i(I)$, and for any $\fk a\in \fk H_i(\wtd I)$, the following creation property (state-field correspondence) holds:
\begin{align}
L(\fk a,\wtd I)\Omega=\fk a.
\end{align}
\end{df}

If $\scr E$ is vector-labeled, $x\in\mc A(I)$, and $x\Omega\in\fk H_0(\wtd I)$, then for any $\mc H_k\in\scr C$,   $L(x\Omega,\wtd I)=\pi_k(x)=R(x\Omega,\wtd I)$ when acting on $\mc H_k$. Indeed, we choose an arbitrary $\wtd J\in\Jtd$ clockwise to $\wtd I$, and $\fk b\in\fk H_k(\wtd J)$. Then by locality, $L(x\Omega,\wtd I)R(\fk b,\wtd J)\Omega=R(\fk b,\wtd J)L(x\Omega,\wtd I)\Omega=R(\fk b,\wtd J)x\Omega=\pi_k(x)R(\fk b,\wtd J)\Omega$. Now $L(x\Omega,\wtd I)=\pi_k(x)$ follows immediately from the Reeh-Schlieder property. Similar argument shows that $\pi_k(x)=R(x\Omega,\wtd I)$.

\subsection{Connes categorical extensions}

The main goal in this section is to use Connes fusions to construct a conformal categorical extension $\scr E_C=(\mc A,\Rep(\mc A),\boxtimes,\mc H)$
of $\mc A$. For any $\mc H_i\in\Rep(\mc A),\wtd I=(I,\arg_I)\in\Jtd$, we set $\mc H_i(\wtd I)=\mc H_i(I)$, which plays the role of $\fk H_i(\wtd I)$ in the definition of categorical extensions. Fix $z_+=i,z_-=-i$, and let $\varrho:[0,1]\rightarrow S^1$ be an $180^\degree$ clockwise rotation from $z_+$ to $z_-$ defined for instance by \eqref{eq24}. Choose a path $\alpha_{\wtd I}:[0,1]\rightarrow S^1$ from (a point in) $I$ to $z_+$, such that the arg value $\arg_I(\alpha_{\wtd I}(0))$ of $\alpha_{\wtd I}(0)$ changes continuously along this path to the arg value $\frac \pi 2$ of $z_+=i$. (For example, if $-1\in I$, and $\wtd I=(I,\arg_I)$ is defined in such a way that $\arg_I(-1)=5\pi$, then we can choose $\alpha_{\wtd I}$ to be a $(5-\frac 12)\cdot 180^\degree$ clockwise rotation from $-1$ to $i$.) Then we know that   $\arg_I(\alpha_{\wtd I}(0))$ changes continuously along the path $\varrho*\lambda_{\wtd I}$ to the arg value $-\frac \pi 2$ of $z_-=-i$. 

Now for any $\xi\in\mc H_i(I)$ and any $\mc H_k\in\Rep(\mc A)$, we let $Z(\xi,I)$ (which is originally a linear map $\mc H_0\rightarrow \mc H_i$) also be a bounded linear operator from $\mc H_k$ to $\mc H_i(I)\boxtimes\mc H_k$:
\begin{gather*}
Z(\xi,I):\mc H_k\rightarrow\mc H_i(I)\boxtimes\mc H_k,\qquad \chi\mapsto\xi\otimes\chi.
\end{gather*}
Clearly $Z(\xi,I)\in\Hom_{\mc A(I^c)}(\mc H_k,\mc H_i(I)\boxtimes\mc H_k)$. Now we define $L(\xi,\wtd I):\mc H_k\rightarrow\mc H_i\boxtimes\mc H_k$ to be
\begin{align}
L(\xi,\wtd I):\mc H_k\xrightarrow{Z(\xi,I)}\mc H_i(I)\boxtimes\mc H_k\xrightarrow{\alpha_{\wtd I}^\bullet}\mc H_i(S^1_+)\boxtimes\mc H_k=\mc H_i\boxtimes\mc H_k.
\end{align}
Since path continuations commute with the actions of $\mc A$, we have $L(\xi,\wtd I)\in\Hom_{\mc A(I^c)}(\mc H_k,\mc H_i\boxtimes\mc H_k)$. Similarly, we  define $R(\xi,\wtd I)\in\Hom_{\mc A(I^c)}(\mc H_k,\mc H_k\boxtimes\mc H_i)$ to be
\begin{align}
R(\xi,\wtd I):\mc H_k\xrightarrow{Z(\xi,I)}\mc H_i(I)\boxtimes\mc H_k\xrightarrow{(\varrho*\alpha_{\wtd I})^\bullet}\mc H_i(S^1_-)\boxtimes\mc H_k=\mc H_k\boxtimes\mc H_i.
\end{align}
Then equation \eqref{eq33} (with $\ss=\mathbb B$) follows directly from the fact that $\mathbb B=\varrho^\bullet$.
\begin{thm}
With the above constructions, $\scr E_C=(\mc A,\Rep(\mc A),\boxtimes,\mc H)$ is a vector-labeled categorical extension of $\mc A$. We call it the \textbf{Connes categorical extension} of $\mc A$.
\end{thm}

\begin{proof}
We only prove locality here. All the other axioms are easy to verify using the results obtained in the previous  chapter.

Step 1. We show that for any $\mc H_i,\mc H_j,\mc H_k\in\Rep(\mc A)$, disjoint $I,J\in\mc J$, and $\xi\in\mc H_i(I),\eta\in\mc H_j(J)$, the following diagram commutes adjointly:
\begin{align}
\begin{CD}
\mc H_k @> \quad Z(\eta,J)\quad  >> \mc H_k\boxtimes\mc H_j(J)\\
@V Z(\xi,I)  VV @V Z(\xi,I) VV\\
\mc H_i(I)\boxtimes\mc H_k @> \quad  Z(\eta,J)\quad  >> \mc H_i(I)\boxtimes\mc H_k\boxtimes\mc H_j(J)
\end{CD}.\label{eq35}
\end{align}
It is easy to show that this diagram commutes. Indeed, if we choose an arbitrary $\chi\in\mc H_k$, then clearly $Z(\eta,J)Z(\xi,I)\chi=\xi\otimes\chi\otimes\eta=Z(\xi,I)Z(\eta,J)\chi$. To prove the adjoint commutativity, we choose any $\xi'\in\mc H_i(I),\chi'\in\mc H_k$. Then it is easy to show
\begin{align}
Z(\xi,I)^*(\xi'\otimes\chi')=\pi_k(Z(\xi,I)^*Z(\xi',I))\chi'\label{eq34}
\end{align}
by evaluating both sides with an arbitrary vector $\chi\in\mc H_k$.
Similarly, we have
\begin{align*}
Z(\xi,I)^*Z(\eta,J)(\xi'\otimes\chi')=Z(\xi,I)^*(\xi'\otimes\chi'\otimes\eta)=\pi_{k\boxtimes j}(Z(\xi,I)^*Z(\xi',I))(\chi'\otimes\eta).
\end{align*}
Since $Z(\xi,I)^*Z(\xi',I)|_{\mc H_0}\in\mc A(I)$, the right hand side of the above expression equals
\begin{align*}
\pi_k(Z(\xi,I)^*Z(\xi',I))\chi'\otimes\eta=Z(\eta,J)\pi_k(Z(\xi,I)^*Z(\xi',I))\chi',
\end{align*}
which, by \eqref{eq34}, equals $Z(\eta,J)Z(\xi,I)^*(\xi'\otimes\chi')$. Thus we've proved $Z(\eta,J)Z(\xi,I)^*=Z(\xi,I)^*Z(\eta,J)$, and hence the adjoint commutativity of \eqref{eq35}.\\

Step 2. We prove the adjoint commutativity of \eqref{eq31} for $\fk a=\xi,\fk b=\eta$. Choose a path $\beta_{\wtd J}$ from $J$ to $z_-\in S^1_-$, such that the arg value $\arg_J(\beta_{\wtd J}(0))$ of  $\beta_{\wtd J}(0)$ changes continuously along this path to the arg value $-\frac\pi 2$ of $z_-$. Then clearly $R(\eta,\wtd J)=\beta_{\wtd J}^\bullet Z(\eta,J)$. By replacing $\alpha_{\wtd I}$ and $\beta_{\wtd J}$ with homotopic paths, we assume that $(\alpha_{\wtd I},\beta_{\wtd J})$ is a path in $\Conf_2(S^1)$, i.e., $\alpha_{\wtd I}(t)\neq \beta_{\wtd J}(t)$ for any $t\in[0,1]$. (It is here that the anticlockwiseness of $\wtd I$ to $\wtd J$ is used.) Consider the following $2\times 2$ matrix of diagrams:
\begin{align}
\begin{CD}
\mc H_k @>\quad Z(\eta,J) \quad >> \mc H_k\boxtimes\mc H_j(J) @> ~~~\quad  \beta_{\wtd J}^\bullet \quad~~~  >> \mc H_k\boxtimes\mc H_j(S^1_-)\\
@V Z(\xi,I)  VV @V Z(\xi,I) VV @V Z(\xi,I) VV\\
\mc H_i(I)\boxtimes\mc H_k @> \quad  Z(\eta,J) \quad  >> \mc H_i(I)\boxtimes\mc H_k\boxtimes\mc H_j(J)   @> \quad \id_i\otimes \beta_{\wtd J}^\bullet\quad  >>\mc H_i(I)\boxtimes( \mc H_k\boxtimes\mc H_j(S^1_-))\\
@V \alpha_{\wtd I}^\bullet  VV @V  \alpha_{\wtd I}^\bullet\otimes\id_j VV @V  \alpha_{\wtd I}^\bullet  VV\\
\mc H_i(S^1_+)\boxtimes\mc H_k @> \quad  Z(\eta,J) \quad  >> (\mc H_i(S^1_+)\boxtimes\mc H_k)\boxtimes\mc H_j(J)   @>~~~ \quad \beta_{\wtd J}^\bullet\quad~~~  >>\mc H_i(S^1_+)\boxtimes \mc H_k\boxtimes\mc H_j(S^1_-).
\end{CD}
\end{align}
If we can prove the adjoint commutativity of all these four diagrams, then \eqref{eq31} commutes adjointly. Now the $(1,1)$-diagram commutes adjointly due to step 1. It is easy to show that the $(2,1)$-diagram commutes when $\alpha_{\wtd I}^\bullet$ is more generally any morphism in $\Hom_{\mc A}(\mc H_i(I)\boxtimes\mc H_k,\mc H_i(S^1_+)\boxtimes\mc H_k)$. Therefore $(2,1)$ commutes. Since $\alpha_{\wtd I}^\bullet$ and $\alpha_{\wtd I}^\bullet\otimes\id_j$ are unitary, $(2,1)$ commutes adjointly. Similarly $(1,2)$ also commutes adjointly. The (adjoint) commutativity of the $(2,2)$-diagram follows from proposition \ref{lb53}.
\end{proof}

The above argument and result clearly hold  for any M\"obius covariant net. Now that we assume $\mc A$ to be conformal covariant, we can also show that the Connes categorical extension is conformal covariant.

\begin{thm}\label{lb16}
$\scr E_C$ is a  conformal vector-labeled categorical extension.
\end{thm}

\begin{proof}
Choose any $\mc H_i,\mc H_k\in\Rep(\mc A),\wtd I\in\Jtd,\xi\in\mc H_i(I),g\in\GA$. We show that
\begin{align}
gL(\xi,\wtd I)g^{-1}=L(g\xi g^{-1},g\wtd I).
\end{align}
Recall our notation  $g\xi g^{-1}=gZ(\xi,I)g^{-1}\Omega$. Choose $z\in I$, and let $\alpha_{\wtd I}$ be a path in $S^1$ from $z$ to $z_+=i$, along which $\arg_I(z)$ changes continuously to the argument $\frac \pi 2$ of $z_+$. Let furthermore $\lambda$ be a path in $\GA$ from $1$ to $g$, and let $\lambda_z:t\in[0,1]\mapsto\lambda(t)z\in S^1$. Therefore, if we write $g\wtd I=(gI,\arg_{gI})$, then $\alpha_{\wtd I}*(\lambda_z)^{-1}$ is a path in $S^1$ from $gz$ to $z_+$, along which the argument $\arg_{gI}(gz)$ of $gz$ changes continuously to the argument $\frac \pi 2$ of $z_+$. It follows that $L(\xi,\wtd I)=\alpha_{\wtd I}^\bullet Z(\xi,I)$ and $L(g\xi g^{-1},g\wtd I)=\alpha_{\wtd I}^\bullet (\lambda_z^\bullet)^{-1} Z(g\xi g^{-1},g I)$. By relation \eqref{eq25}, we have
\begin{align*}
gZ(\xi,I)=(\lambda_z^\bullet)^{-1}Z(g\xi g^{-1},gI)g.
\end{align*}
Using the fact that path continuations intertwine the actions of $\GA$ (since they intertwine the actions of $\mc A$. Note also corollary \ref{lb7}), we have
\begin{align*}
&gL(\xi,\wtd I)g^{-1}=g\alpha_{\wtd I}^\bullet Z(\xi,I)g^{-1}=\alpha_{\wtd I}^\bullet\cdot gZ(\xi,I)\cdot g^{-1}\\
=&\alpha_{\wtd I}^\bullet\cdot(\lambda_z^\bullet)^{-1}Z(g\xi g^{-1},gI)g\cdot g^{-1}=L(g\xi g^{-1},g\wtd I).
\end{align*}
\end{proof}

\subsection{Hexagon axioms}\label{lb11}

With the results obtained so far, we  give an easy proof of the hexagon axioms for the braid operators $\mathbb B$ defined in section \ref{lb10}. First we collect some useful formulas.  For any $\mc H_i,\mc H_j\in\Rep(\mc A)$, disjoint $\wtd I,\wtd J\in\Jtd$, and $\xi\in\mc H_i(I),\eta\in\mc H_j(J)$, if $\wtd I$ is anticlockwise to $\wtd J$, then
\begin{align}
L(\xi,\wtd I)\eta=\mathbb B_{j,i}L(\eta,\wtd J)\xi.
\end{align}
Indeed, $L(\xi,\wtd I)\eta=L(\xi,\wtd I)R(\eta,\wtd J)\Omega=R(\eta,\wtd J)L(\xi,\wtd I)\Omega=R(\eta,\wtd J)\xi=\mathbb B_{j,i}L(\eta,\wtd J)\xi$. Therefore, if $\wtd I$ is clockwise to $\wtd J$ then
\begin{align}
L(\xi,\wtd I)\eta=\mathbb B_{i,j}^{-1}L(\eta,\wtd J)\xi.
\end{align}
Next, if $F\in\Hom_{\mc A}(\mc H_i,\mc H_{i'}),G\in \Hom_{\mc A}(\mc H_j,\mc H_{j'})$, then it is easy to see that for any $\wtd I\in\Jtd,\xi\in\mc H_i(I),\eta\in\mc H_j$, 
\begin{align}
(F\otimes G)L(\xi,\wtd I)\eta=L(F\xi,\wtd I)G\eta.\label{eq49}
\end{align}
The following properties are parallel to the fusion and braid relations for intertwining operators of vertex operator algebras.

\begin{pp}\label{lb13}
Choose $\wtd I,\wtd J,\wtd O\in\Jtd$ such that $\wtd I,\wtd J\subset\wtd O$, and $\xi\in\mc H_i(I),\eta\in\mc H_j(J)$. Then $L(\xi,\wtd I)\eta\in(\mc H_i\boxtimes\mc H_j)(O)$, and
\begin{align}
L(\xi,\wtd I)L(\eta,\wtd J)=L(L(\xi,\wtd I)\eta,\wtd O)\label{eq39}
\end{align}
when acting on any $\mc H_k\in\Rep(\mc A)$.
\end{pp} 

\begin{proof}
Since $L(\xi,\wtd I)=L(\xi,\wtd O)$, $L(\eta,\wtd J)=L(\xi,\wtd O)$, we may assume that $\wtd I=\wtd J=\wtd O$. Since both $L(\xi,\wtd O)$ and $L(\eta,\wtd O)$ intertwines the actions of $\mc A(O^c)$, so does $L(\xi,\wtd O)L(\eta,\wtd O)$. Hence $L(\xi,\wtd O)L(\eta,\wtd O)\in\Hom_{\mc A(O^c)}(\mc H_0,\mc H_i\boxtimes\mc H_j)$. Since $L(\xi,\wtd O)L(\eta,\wtd O)\Omega=L(\xi,\wtd O)\eta$, we conclude $L(\xi,\wtd O)\eta\in(\mc H_i\boxtimes\mc H_j)(O)$.

Now we choose $\arg_{O^c}$ such that $\wtd O^c=(O^c,\arg_{O^c})$ is clockwise to $\wtd O$. Then for any $\mc H_k\in\Rep(\mc A)$ and $\chi\in\mc H_k(O^c)$,
\begin{align*}
&L(\xi,\wtd O)L(\eta,\wtd O)\chi=L(\xi,\wtd O)L(\eta,\wtd O)R(\chi,\wtd O^c)\Omega=R(\chi,\wtd O^c)L(\xi,\wtd O)L(\eta,\wtd O)\Omega\\
=&R(\chi,\wtd O^c)L(\xi,\wtd O)\eta=R(\chi,\wtd O^c)L(L(\xi,\wtd O)\eta,\wtd O)\Omega=L(L(\xi,\wtd O)\eta,\wtd O)R(\chi,\wtd O^c)\Omega\\
=&L(L(\xi,\wtd O)\eta,\wtd O)\chi.
\end{align*}
This proves \eqref{eq39}.
\end{proof}

\begin{pp}\label{lb17}
Suppose that $\wtd I$ is anticlockwise to $\wtd J$, and there exits $\wtd O\in\Jtd$ such that $\wtd I,\wtd J\subset\wtd O$. Then for any $\mc H_i,\mc H_j,\mc H_k\in\Rep(\mc A),\xi\in\mc H_i(I),\eta\in\mc H_j(J),\chi\in
	\mc H_k$, we have
	\begin{align}
	L(\xi,\wtd I)L(\eta,\wtd J)\chi=(\mathbb B_{j,i}\otimes\id_k)L(\eta,\wtd J)L(\xi,\wtd I)\chi.\label{eq55}
	\end{align}
\end{pp}

\begin{proof}
We compute 
\begin{align*}
&(\mathbb B_{j,i}\otimes\id_k)L(\eta,\wtd J)L(\xi,\wtd I)\eta=(\mathbb B_{j,i}\otimes\id_k) L(L(\eta,\wtd J)\xi,\wtd O)\chi=L(\mathbb B_{j,i}L(\eta,\wtd J)\xi,\wtd O)\chi\\
=&L(L(\xi,\wtd I)\eta,\wtd O)\chi=L(\xi,\wtd I)L(\eta,\wtd J)\chi.
\end{align*}
\end{proof}

The above proposition is equivalent to the braiding condition (g) of definition \ref{lb9}. Although \ref{lb9}-(g) looks simpler and is easier to verify than \eqref{eq55},  the latter has clearer physical meaning: it says that the left operators $L$ satisfy \textbf{braid statistics}, which generalize the usual boson and fermion statistics.\\

Now we prove the hexagon axioms for $\mathbb B$. In this paper, we always let $\mc H_{i\boxtimes j}$ or its subscript $i\boxtimes j$ denote $\mc H_i\boxtimes\mc H_j$.

\begin{thm}[Hexagon axioms for $\mathbb B$]
Choose any $\mc H_i,\mc H_j,\mc H_k\in\Rep(\mc A)$. Then for the braiding $\mathbb B$ defined by  $180^\degree$ clockwise rotations (see the end of section \ref{lb10}),  we have the following relations for morphisms $\mc H_i\boxtimes\mc H_j\boxtimes\mc H_k\rightarrow \mc H_k\boxtimes \mc H_i\boxtimes\mc H_j$:
\begin{gather}
(\mathbb B_{i,k}\otimes\id_j)(\id_i\otimes\mathbb B_{j,k})=\mathbb B_{i\boxtimes j,k},\label{eq40}\\
(\mathbb B_{k,i}^{-1}\otimes\id_j)(\id_i\otimes\mathbb B_{k,j}^{-1})=\mathbb B_{k,i\boxtimes j}^{-1}.
\end{gather}
\end{thm}

\begin{proof}
Since the proof of the two relations are similar, we only prove the first one. Choose disjoint $\wtd I,\wtd J,\wtd K\in\Jtd$ such that $\wtd I$ is clockwise to $\wtd J$, $\wtd J$ is clockwise to $\wtd K$, and $\wtd I,\wtd J,\wtd K$ can be covered by an arg-valued open interval in $S^1$. Choose $\wtd O\in\Jtd$ containing $\wtd I,\wtd J$ and clockwise to $\wtd K$. Then for any $\xi\in\mc H_i(I),\eta\in\mc H_j(J),\chi\in\mc H_k(K)$, the action of the left hand side of equation \eqref{eq40}  on $L(\xi,\wtd I)L(\eta,\wtd J)\chi$ is
\begin{align*}
&L(\xi,\wtd I)L(\eta,\wtd J)\chi\xrightarrow {(\id_i\otimes\mathbb B_{j,k})} L(\xi,\wtd I)\mathbb B_{j,k}L(\eta,\wtd J)\chi=L(\xi,\wtd I)L(\chi,\wtd K)\eta \xrightarrow{\mathbb B_{i,k}\otimes\id_j}L(\chi,\wtd K) L(\xi,\wtd I)\eta.
\end{align*}
On the other hand, the action of the right hand side of  \eqref{eq40}  on $L(\xi,\wtd I)L(\eta,\wtd J)\chi$ is
\begin{align*}
L(\xi,\wtd I)L(\eta,\wtd J)\chi=L(L(\xi,\wtd I)\eta,\wtd O)\chi\xrightarrow{\mathbb B_{i\boxtimes j,k}}\mathbb B_{i\boxtimes j,k} L(L(\xi,\wtd I)\eta,\wtd O)\chi=L(\chi,\wtd K)L(\xi,\wtd I)\eta.
\end{align*}
Hence \eqref{eq40} is proved.
\end{proof}

\begin{thm}
With the monoidal structure and the braid operators defined in section \ref{lb10}, $\Rep(\mc A)$ is a braided $C^*$-tensor category.
\end{thm}

\subsection{Uniqueness of tensor categorical structures}

Let $\scr E=(\mc A,\scr C,\boxdot,\fk H)$ be a categorical extension of $\mc A$ with braid operator $\ss$. In this section, we show that $(\scr C,\boxdot,\ss)$ is a braided $C^*$-tensor category (i.e., the unitary map $\ss$ is a functorial $\mc A$-module isomorphism satisfying the Hexagon axioms), and that there is a natural equivalent between $(\scr C,\boxdot,\ss)$ and a braided tensor subcategory of $(\Rep(\mc A),\boxtimes,\mathbb B)$.

To begin with, we let $\wht {\scr C}$ be the $C^*$-category of all $\mc A$-modules $\mc H_i$ such that $\mc H_i$ is  unitarily equivalent to some object in $\scr C$. We assume without loss of generality that $\scr E$ is vector-labeled. So, in particular, for each $\wtd I\in\Jtd,\mc H_i\in\scr C$, $\fk H_i(\wtd I)$ is a subset of $\mc H_i(I)$. We thus write $\xi,\eta,\dots$ instead of $\fk a,\fk b,\dots$ for elements in $\fk H_i(\wtd I)$. But then there is a conflict of notations, as $L(\xi,\wtd I)$ may denote a left action in either $\scr E$ or the Connes categorical extension $\scr E_C$. We avoid this issue by letting $L^\boxdot(\xi,\wtd I)$ and $L^\boxtimes(\xi,\wtd I)$ denote  left actions  in $\scr E=(\mc A,\scr C,\boxdot,\fk H)$ and $\scr E_C=(\mc A,\Rep(\mc A),\boxtimes,\mc H)$ respectively. Similar notations apply to right actions.

\begin{thm}\label{lb12}
Let $\scr E=(\mc A,\scr C,\boxdot,\fk H)$ be a vector-labeled categorical extension of $\mc A$. Then $(\scr C,\boxdot,\ss)$ is a braided $C^*$-tensor category, $\wht{\scr C}$ is closed under Connes fusion product $\boxtimes$, and there is a (unique) unitary functorial (i.e. natural) isomorphism
\begin{align}
\Phi_{i,j}:\mc H_i\boxtimes\mc H_j\rightarrow\mc H_i\boxdot\mc H_j\qquad(\forall\mc H_i,\mc H_j\in\scr C),
\end{align}
such that for any $\wtd I\in\Jtd,\mc H_i,\mc H_j\in\scr C,\xi\in\fk H_i(\wtd I),\eta\in\mc H_j$,
\begin{gather}
\Phi_{i,j}L^\boxtimes(\xi,\wtd I)\eta=L^\boxdot(\xi,\wtd I)\eta,\label{eq43}\\
\Phi_{j,i}R^\boxtimes(\xi,\wtd I)\eta=R^\boxdot(\xi,\wtd I)\eta.\label{eq44}
\end{gather}

Moreover, $\Phi$ induces an equivalence of braided $C^*$-tensor categories $(\wht{\scr C},\boxtimes,\mathbb B)\simeq(\scr C,\boxdot,\ss)$. More precisely, ``equivalence" means that for any $\mc H_i,\mc H_j,\mc H_k\in\scr C$, the following conditions are satisfied:\\
(a) The following diagram commutes.
\begin{gather}
\begin{CD}
\mc H_i\boxtimes\mc H_k\boxtimes\mc H_j @>\quad\id_i\otimes\Phi_{k,j}\quad>> \mc H_i\boxtimes(\mc H_k\boxdot\mc H_j)\\
@V \Phi_{i,k}\otimes\id_j VV @V \Phi_{i,k\boxdot j} VV\\
(\mc H_i\boxdot\mc H_k)\boxtimes\mc H_j @> ~~\quad\Phi_{i\boxdot k,j} \quad~~>> \mc H_i\boxdot\mc H_k\boxdot\mc H_j
\end{CD}.\label{eq45}
\end{gather}
(b) The following two maps equal $\id_i:\mc H_i\rightarrow\mc H_i$.
\begin{gather}
\mc H_i\simeq \mc H_0\boxtimes\mc H_i\xrightarrow{\Phi_{0,i}} \mc H_0\boxdot\mc H_i\simeq\mc H_i,\label{eq46}\\
\mc H_i\simeq \mc H_i\boxtimes\mc H_0\xrightarrow{\Phi_{i,0}} \mc H_i\boxdot\mc H_0\simeq\mc H_i. \label{eq47}
\end{gather}
(c) The following diagram commutes.
\begin{gather}
\begin{CD}
\mc H_i\boxtimes\mc H_j @> \quad\mathbb B_{i,j} \quad>> \mc H_j\boxtimes\mc H_i\\
@V \Phi_{i,j} VV @V \Phi_{j,i} VV\\
\mc H_i\boxdot\mc H_j @>\quad\ss_{i,j}\quad>> \mc H_j\boxdot\mc H_i
\end{CD}.\label{eq48}
\end{gather}
\end{thm}

Note that the functoriality of $\Phi$ means that for any objects $\mc H_i,\mc H_{i'},\mc H_j,\mc H_{j'}$ of $\scr C$ and any $F\in\Hom_{\mc A}(\mc H_i,\mc H_{i'}),G\in\Hom_{\mc A}(\mc H_j,\mc H_{j'})$, we have $\Phi_{i',j'}(F\boxtimes G)=(F\boxdot G)\Phi_{i,j}$, where $F\boxtimes G$ and $F\boxdot G$ are the tensor products  $F\otimes G$ in $(\wht{\scr C},\boxtimes)$ and in $(\scr C,\boxdot)$ respectively.

\begin{proof}
Step 1. Let $\mc H_i,\mc H_j\in\scr C$. Choose any disjoint $\wtd I,\wtd J\in\Jtd$ such that $\wtd I$ is anticlockwise to $\wtd J$. Then for any $\xi\in\fk H_i(\wtd I)$ and $x\in\mc A(I^c)$,
\begin{align}
L^\boxdot(\xi,\wtd I)x\Omega=xL^\boxdot(\xi,\wtd I)\Omega=x\xi=xL^\boxtimes(\xi,\wtd I)\Omega=L^\boxtimes(\xi,\wtd I)x\Omega,
\end{align}
which shows that $L^\boxdot(\xi,\wtd I)=L^\boxtimes(\xi,\wtd I)$ when acting on $\mc H_0$. Similar result holds for the right actions of modules on $\mc H_0$. Then using locality, we compute, for any  $x_1,x_2\in\mc A(I),\xi_1,\xi_2\in\fk H_i(\wtd I),\eta_1,\eta_2\in\fk H_j(\wtd J)$, that
\begin{align}
&\bk{x_1L^\boxdot(\xi_1,\wtd I)\eta_1|x_2L^\boxdot(\xi_2,\wtd I)\eta_2 }=\bk{x_1L^\boxdot(\xi_1,\wtd I)R^\boxdot(\eta_1,\wtd J)\Omega|x_2L^\boxdot(\xi_2,\wtd I)R^\boxdot(\eta_2,\wtd J)\Omega }\nonumber\\
=&\bk{R^\boxdot(\eta_2,\wtd J)^*L^\boxdot(\xi_2,\wtd I)^*x_2^*x_1L^\boxdot(\xi_1,\wtd I)R^\boxdot(\eta_1,\wtd J)\Omega|\Omega}\nonumber\\
=&\bk{R^\boxdot(\eta_2,\wtd J)^*R^\boxdot(\eta_1,\wtd J)L^\boxdot(\xi_2,\wtd I)^*x_2^*x_1L^\boxdot(\xi_1,\wtd I)\Omega|\Omega}.\label{eq41}
\end{align}
Note that on the right hand side of \eqref{eq41}, $L^\boxdot(\xi_1,\wtd I),L^\boxdot(\xi_2,\wtd I),R^\boxdot(\eta_1,\wtd J),R^\boxdot(\eta_2,\wtd J)$ all act on the vacuum module. Similarly,
\begin{align}
\bk{x_1L^\boxtimes(\xi_1,\wtd I)\eta_1|x_2L^\boxtimes(\xi_2,\wtd I)\eta_2 }=\bk{R^\boxtimes(\eta_2,\wtd J)^*R^\boxtimes(\eta_1,\wtd J)L^\boxtimes(\xi_2,\wtd I)^*x_2^*x_1L^\boxtimes(\xi_1,\wtd I)\Omega|\Omega},\label{eq42}
\end{align}
and on the right hand side of this equations, all the operators $R^\boxtimes$ are acting on $\mc H_0$. Therefore the left hand sides of \eqref{eq41} and \eqref{eq42} are equal. We thus conclude, by the density of fusion products and the Reeh-Schlieder property (conditions (d) and (e) of definition \ref{lb9}), that there is a unique unitary map $\Phi^{\wtd I,\wtd J}_{i,j}:\mc H_i\boxtimes\mc H_j\rightarrow\mc H_i\boxdot\mc H_j$ satisfying
\begin{gather*}
\Phi^{\wtd I,\wtd J}_{i,j}xL^\boxtimes(\xi,\wtd I)\eta=xL^\boxdot(\xi,\wtd I)\eta\qquad(\forall x\in\mc A(I), \xi\in\fk H_i(\wtd I),\eta\in\fk H_j(\wtd J)).
\end{gather*}
In particular, $\Phi^{\wtd I,\wtd J}_{i,j}$ intertwines the actions of $\mc A(I)$ on $\mc H_i\boxtimes\mc H_j$ and $\mc H_i\boxdot\mc H_j$. It is obvious that if $\wtd I_0,\wtd J_0\in\Jtd$ and $\wtd I_0\subset\wtd I,\wtd J_0\subset\wtd J$, then $\Phi^{\wtd I_0,\wtd J_0}_{i,j}=\Phi^{\wtd I,\wtd J}_{i,j}$. Therefore it is easy to see that  $\Phi^{\wtd I,\wtd J}_{i,j}=\Phi^{\wtd I',\wtd J'}_{i,j}$ for any $\wtd I',\wtd J'\in\Jtd$ such that $\wtd I'$ is anticlockwise to $\wtd J'$, i.e., $\Phi^{\wtd I,\wtd J}_{i,j}$ is independent of $\wtd I$ and $\wtd J$. This shows that $\Phi^{\wtd I,\wtd J}_{i,j}$ intertwines the actions of $\mc A(I')$ for any $I'\in\mc J$.  

We thus write $\Phi^{\wtd I,\wtd J}_{i,j}=\Phi_{i,j}$. Then $\Phi_{i,j}:\mc H_i\boxtimes\mc H_j\rightarrow\mc H_i\boxdot\mc H_j$ is a unitary $\mc A$-module isomorphism satisfying \eqref{eq43} for any $\xi\in\fk H_i(\wtd I),\eta\in\fk H_j(\wtd J)$. In particular, $\mc H_i\boxtimes\mc H_j\in\wht {\scr C}$ as $\mc H_i\boxdot\mc H_j\in\scr C$. By Reeh-Schlieder property, \eqref{eq43} holds for any $\xi\in\fk H_i(\wtd I),\eta\in\mc H_j$.  By \eqref{eq43} and the functoriality of $L^\boxtimes(\xi,\wtd I)$ and $L^\boxdot(\xi,\wtd I)$, it is clear that $\Phi_{i,j'}(\id_i\boxtimes G)=(\id_i\boxdot G)\Phi_{i,j}$ for any $\mc H_j'\in\scr C$ and $G\in\Hom_{\mc A}(\mc H_j,\mc H_{j'})$. Now assume $\xi\in\fk H_i(\wtd I),\eta\in\fk H_j(\wtd J)$. Then
\begin{align*}
\Phi_{i,j}R^\boxtimes(\eta,\wtd J)\xi=\Phi_{i,j}L^\boxtimes(\xi,\wtd I)\eta=L^\boxdot(\xi,\wtd I)\eta=R^\boxdot(\eta,\wtd J)\xi.
\end{align*}
Since $\fk H_i(\wtd I)$ spans a dense subspace of $\mc H_i$ by Reeh-Schlieder property, the left and the right hand sides of the above equation are equal for any $\xi\in\mc H_i$. Equivalently, we've proved condition \eqref{eq44}. A similar argument shows that $\Phi_{i',j}(F\boxtimes\id_j)=(F\boxdot\id_j)\Phi_{i,j}$ for any $\mc H_{i'}\in\scr C$ and $F\in\Hom_{\mc A}(\mc H_i,\mc H_{i'})$. Therefore $\Phi$ is functorial.\\

Step 2. We now show that $\Phi$ induces an equivalence of $C^*$-tensor categories. We first verify the commutativity of digram \eqref{eq45}. Choose any $\mc H_i,\mc H_j,\mc H_k\in\scr C,$ disjoint $\wtd I,\wtd J\in\Jtd$ such that $\wtd I$ is anticlockwise to $\wtd J$, and $\xi\in\fk H_i(\wtd I),\eta\in\fk H_j(\wtd J),\chi\in\mc H_k$. Then using conditions \eqref{eq43}, \eqref{eq44}, and the functoriality and locality of categorical extensions (conditions (b) and (f) of definition \ref{lb9}), we compute
\begin{align*}
&L^\boxtimes(\xi,\wtd I)R^\boxtimes(\eta,\wtd J)\chi\xrightarrow{\id_i\otimes\Phi_{k,j}} L^\boxtimes(\xi,\wtd I)\Phi_{k,j}R^\boxtimes(\eta,\wtd J)\chi\\
=&L^\boxtimes(\xi,\wtd I)R^\boxdot(\eta,\wtd J)\chi \xrightarrow{\Phi_{i,k\boxdot j}} L^\boxdot(\xi,\wtd I)R^\boxdot(\eta,\wtd J)\chi,
\end{align*}
and also
\begin{align*}
&L^\boxtimes(\xi,\wtd I)R^\boxtimes(\eta,\wtd J)\chi=R^\boxtimes(\eta,\wtd J)L^\boxtimes(\xi,\wtd I)\chi \xrightarrow{\Phi_{i,k}\otimes\id_j} R^\boxtimes(\eta,\wtd J)\Phi_{i,k}L^\boxtimes(\xi,\wtd I)\chi\\
=&R^\boxtimes(\eta,\wtd J)L^\boxdot(\xi,\wtd I)\chi \xrightarrow{\Phi_{i\boxdot k,j}}  R^\boxdot(\eta,\wtd J)L^\boxdot(\xi,\wtd I)\chi= L^\boxdot(\xi,\wtd I)R^\boxdot(\eta,\wtd J)\chi.
\end{align*}
Therefore diagram \eqref{eq45} commutes.

For condition (b), we choose any $\wtd I\in\Jtd,\xi\in\mc H_i$. Then
\begin{align*}
\xi=L^\boxtimes(\Omega,\wtd I)\xi\xrightarrow{\Phi_{0,i}} L^\boxdot(\Omega,\wtd I)\xi=\xi.
\end{align*}
Thus \eqref{eq46} equals identity. Similarly \eqref{eq47} also equals identity. Thus $\Phi$ is an equivalence of $C^*$-tensor categories.

Finally, choose any  $\xi\in\fk H_i(\wtd I),\eta\in\mc H_j$. Then
\begin{gather*}
L^\boxtimes(\xi,\wtd I)\eta\xrightarrow{\Phi_{i,j}} L^\boxdot(\xi,\wtd I)\eta \xrightarrow{\ss_{i,j}} R^\boxdot(\xi,\wtd I)\eta,
\end{gather*}
and also
\begin{align*}
L^\boxtimes(\xi,\wtd I)\eta\xrightarrow{\mathbb B_{i,j}} R^\boxtimes(\xi,\wtd I)\eta \xrightarrow{\Phi_{j,i}} R^\boxdot(\xi,\wtd I)\eta.
\end{align*}
Hence diagram \eqref{eq48} commutes. This shows that $\Phi$ intertwines $\mathbb B$ and $\ss$. Thus, due to the functoriality of $\Phi$, we conclude that $\ss$ is functorial and satisfies hexagon axioms since these are true for $\mathbb B$. Since $\Phi$ and $\mathbb B$ are isomorphisms of $\mc A$-modules, so is $\ss$.  Therefore $(\scr C,\boxdot,\ss)$ is a braided $C^*$-tensor category equivalent to $(\wht{\scr C},\boxtimes,\mathbb B)$ under the functorial map $\Phi$.
\end{proof}

\subsection{Uniqueness of maximal categorical extensions}

\begin{df}
Let $\scr E=(\mc A,\scr C,\boxdot,\fk H)$ and $\scr F=(\mc A,\scr C,\boxdot,\fk K)$ be  vector-labeled categorical extensions of $\mc A$. We say that $\scr F$ is a \textbf{small extension} of $\scr E$ if $\fk H_i(\wtd I)\subset \fk K_i(\wtd I)$ for any $\mc H_i\in\scr C,\wtd I\in\Jtd$,  and for any $\xi\in\fk H_i(\wtd I)$, the operator $L(\xi,\wtd I)$ (resp. $R(\xi,\wtd I)$) of $\scr E$ equals the one of $\scr F$. In this case we write $\scr E\subset \scr F$.
\end{df}

By density of fusion products, it is clear that $\scr E$ and $\scr F$ have the same braid operator $\ss$.

\begin{thm}\label{lb15}
Let $\scr E=(\mc A,\scr C,\boxdot,\fk H)$ be a vector-labeled categorical extension of $\mc A$. Then $\scr E$ has a unique maximal small extension $\overline{\scr E}=(\mc A,\scr C,\boxdot,\mc H)$, called the \textbf{closure} of $\scr E$. Moreover, $\ovl{\scr E}$ satisfies  that for any $\wtd I\in\Jtd$ and $\mc H_i,\mc H_j\in\scr C$:\\
(a) $\mc H_i(\wtd I)=\mc H_i(I)$.\\
(b) If $\xi\in\mc H_i(I),\eta\in\mc H_j(I),$ then $L(\xi,\wtd I)\eta\in(\mc H_i\boxdot\mc H_j)(I)$, and $L(L(\xi,\wtd I)\eta,\wtd I)=L(\xi,\wtd I)L(\eta,\wtd I)$.
\end{thm}

\begin{proof}
As in the last chapter, we let $L^\boxdot, R^\boxdot$ denote actions  in $\scr E$, and let $L^\boxtimes, R^\boxtimes$ denote actions  in the Connes categorical extension $\scr E_C$. Consider the functorial isomorphism $\Phi$ in theorem \ref{lb12}. We now define $\overline{\scr E}=(\mc A,\scr C,\boxdot,\mc H)$ such that $\mc H_i(\wtd I)=\mc H_i(I)$ for any $\mc H_i\in\scr C,\wtd I\in\Jtd$, and  for any $\mc H_i,\mc H_j\in\scr C$, the bounded linear operators $L^\boxdot(\xi,\wtd I)$ and $R^\boxdot(\xi,\wtd I)$ mapping $\mc H_j\rightarrow\mc H_i\boxdot\mc H_j$ and $\mc H_j\rightarrow\mc H_j\boxdot\mc H_i$ respectively are defined by
\begin{gather}
L^\boxdot(\xi,\wtd I)\eta=\Phi_{i,j}L^\boxtimes(\xi,\wtd I)\eta,\label{eq53}\\
R^\boxdot(\xi,\wtd I)\eta=\Phi_{j,i}R^\boxtimes(\xi,\wtd I)\eta
\end{gather}
for any $\eta\in\mc H_j$. Condition (a) is clearly satisfied. Condition (b) follows from proposition \ref{lb13}. This construction is clearly compatible with $\scr E$. So once we've proved that $\ovl{\scr E}$ satisfies the axioms of a categorical extension, then $\ovl{\scr E}$ is a small extension of $\scr E$.

We only check that $\ovl{\scr E}$ satisfies locality, as all the other axioms are easy to verify. Choose any $\wtd I\in\Jtd$ anticlockwise to $\wtd J\in\Jtd$. Choose $\mc H_i,\mc H_j\in\scr C,\xi\in\mc H_i(I),\eta\in\mc H_j(J)$. Consider the following $2\times 2$ matrix of diagrams:
\begin{align}
\begin{CD}
\mc H_k @> \quad R^\boxtimes(\eta,\wtd J)\quad   >> \mc H_k\boxtimes\mc H_j  @> ~~\quad\Phi_{k,j}\quad~~ >> \mc H_k\boxdot\mc H_j \\
@V L^\boxtimes(\xi,\wtd I)   V  V @V L^\boxtimes(\xi,\wtd I) VV   @V  L^\boxtimes(\xi,\wtd I) VV\\
\mc H_i\boxtimes\mc H_k @> \quad R^\boxtimes(\eta,\wtd J) \quad  >> \mc H_i\boxtimes\mc H_k\boxtimes\mc H_j  @> \quad\id_i\otimes\Phi_{k,j}\quad >>\mc H_i\boxtimes(\mc H_k\boxdot\mc H_j)\\
@ V \Phi_{i,k} VV  @V \Phi_{i,k}\otimes\id_j VV  @V \Phi_{i,k\boxdot j}VV\\
\mc H_i\boxdot\mc H_k  @> \quad R^\boxtimes(\eta,\wtd J)\quad >> (\mc H_i\boxdot \mc H_k)\boxtimes\mc H_j  @> ~~\quad\Phi_{i\boxdot k,j}\quad~~ >> \mc H_i\boxdot\mc H_k\boxdot\mc H_j.
\end{CD}
\end{align}
The $(1,1)$-diagram commutes adjointly due to the locality of $\scr E_C$. The $(2,1)$-diagram and the $(1,2)$-diagram both commute (and hence commute adjointly since the maps involved are unitary) due to the functoriality axiom of $\scr E_C$. The $(2,2)$-diagram is exactly diagram \eqref{eq45}. So it commutes (adjointly) by theorem \ref{lb12}. Therefore the largest diagram commutes adjointly, which proves the locality of $\ovl{\scr E}$.

Now suppose that $\scr F$ is a small extension of $\scr E$. If we construct $\ovl{\scr F}$ in a similar way, then by the uniqueness of $\Phi$ (which follows from the density of fusion products and the Reeh-Schlieder property of $\scr E$), we clearly have $\ovl{\scr E}=\ovl{\scr F}$. Therefore $\scr F\subset\ovl{\scr E}.$ Thus we've proved that any small extension of $\scr E$ is included in $\ovl{\scr E}$, which means that $\ovl{\scr E}$ is the unique maximal small extension of $\scr E$.
\end{proof}

The proof of theorem \ref{lb15} implies a very interesting consequence.
\begin{thm}\label{lb35}
Any categorical extension $\scr E=(\mc A,\scr C,\boxdot,\fk H)$  of $\mc A$ is conformal.
\end{thm}
\begin{proof}
Assume that $\scr E$ is vector-labeled. We want to check that for any $g\in\GA,\wtd I\in\Jtd,\mc H_i\in\scr C,\xi\in\fk H_i(\wtd I)$,
\begin{gather}
L^\boxdot(g\xi g^{-1},g\wtd I)=gL^\boxdot(\xi,\wtd I)g^{-1}\label{eq51}
\end{gather}
when acting on any $\mc H_j\in\scr C$. By theorem \ref{lb16}, we have
\begin{gather}
L^\boxtimes(g\xi g^{-1},g\wtd I)=gL^\boxtimes(\xi,\wtd I)g^{-1}.\label{eq52}
\end{gather}
Since $\mc A$-module homomorphisms intertwine the actions of $\GA$ by proposition \ref{lb6}, $g$ commutes with $\Phi$. Therefore relation \eqref{eq51} follows from \eqref{eq52} and \eqref{eq53}.
\end{proof}

We say that a vector-labeled categorical extension $\scr E$ is \textbf{closed} if $\scr E=\ovl{\scr E}$.

\subsection{Semisimple categorical extensions}

In this section we assume that the full abelian $C^*$-subcategory $\scr C$ of $\Rep(\mc A)$ is semisimple, i.e., any module $\mc H_i\in\scr C$ is unitarily equivalent to a finite direct sum of $\mc A$-modules in $\scr C$. We equip $\scr C$ with a braided $C^*$-tensor categorical structure $(\scr C,\boxdot,\ss)$.\footnote{As we have seen in theorem \ref{lb12},  Hexagon axioms and the functoriality of braidings are consequences of the existence of a categorical existence. However, for the categorical local extensions, we need to assume these two properties  at the very beginning: see step 2 of the proof of theorem \ref{lb18}.}  If $\mc F$ is a set of  $\mc A$-modules in $\scr C$, we say that $\mc F$ \textbf{generates} $\scr C$, if for any irreducible $\mc H_i\in\scr C$, there exist $\mc H_{i_1},\dots,\mc H_{i_n}\in\mc F$ such that $\mc H_i$ is equivalent to an (irreducible) $\mc A$-submodule of $\mc H_{i_1}\boxdot\mc H_{i_2}\boxdot\cdots\boxdot\mc H_{i_n}$.

\begin{df}
Assume $(\scr C,\boxdot,\ss)$ is semisimple and $\mc F$ is a generating set of  $\mc A$-modules in $\scr C$. Let $\fk H$ assign, to each $\wtd I\in\Jtd,\mc H_i\in\mc F$, a set $\fk H_i(\wtd I)$ such that $\fk H_i(\wtd I_1)\subset\fk H_i(\wtd I_2)$ whenever $\wtd I_1\subset\wtd I_2$. A \textbf{categorical local extension} $\scr E^\loc=(\mc A,\mc F,\boxdot,\fk H)$ of $\mc A$ associates, to any $\mc H_i\in\mc F,\mc H_k\in\scr C,\wtd I\in\Jtd,\fk a\in\fk H_i(\wtd I)$, bounded linear operators
\begin{gather*}
L(\fk a,\wtd I)\in\Hom_{\mc A(I^c)}(\mc H_k,\mc H_i\boxdot\mc H_k),\\
R(\fk a,\wtd I)\in\Hom_{\mc A(I^c)}(\mc H_k,\mc H_k\boxdot\mc H_i),
\end{gather*}
such that  the  axioms of definition \ref{lb9} are satisfied only for $\mc H_i\in\mc F$, and the locality (axiom (f)) holds only for $\mc H_i,\mc H_j\in\mc F$. We also assume that the unitary operator $\ss$ in the braiding axiom (see \eqref{eq33}) is the same as the one of the braided $C^*$-tensor category $\scr C$.
\end{df}

Similar to categorical extensions, if $\scr E^\loc$ is a categorical local extension, then for any $\mc H_i\in\mc F,\fk a\in\fk H_i(\wtd I)$, the operator $L(\fk a,\wtd I)$, acting on any $\mc A$-module in $\scr C$, is determined by $L(\fk a,\wtd I)\Omega$. Indeed, for any irreducible $\mc H_j\in\scr C$, we can find $\mc H_{j_1},\dots,\mc H_{j_n}\in\mc F$ and an isometric $G\in\Hom_{\mc A}(\mc H_j,\mc H_{j_1}\boxdot\cdots\boxdot\mc H_{j_n})$. Choose $\wtd J\in \Jtd$ clockwise to $\wtd I$. Then for any $\fk b_1\in\fk H_{j_1}(\wtd J),\dots \fk b_n\in\fk H_{j_n}(\wtd J)$,
\begin{align}
L(\fk a,\wtd I)G^*\cdot R(\fk b_n,\wtd J)\cdots R(\fk b_1,\wtd J)\Omega=(\id_i\otimes G^*)\cdot R(\fk b_n,\wtd J)\cdots R(\fk b_1,\wtd J)L(\fk a,\wtd I)\Omega.\label{eq50}
\end{align}
By Reeh-Schlieder property and density of fusion products, vectors of the form $G^*\cdot R(\fk b_n,\wtd J)\cdots R(\fk b_1,\wtd J)\Omega$ span a dense subspace of $\mc H_j$. Thus the action of $L(\fk a,\wtd I)$ on any irreducible $\mc H_j\in\scr C$ is determined by $L(\fk a,\wtd I)\Omega$. The general case follows from the functoriality of $\scr E^\loc$. Therefore  $\scr E^\loc$ is equivalent to a vector-labeled categorical local extension. (The meaning of ``vector-labeled" is understood in a similar way as in definition \ref{lb14}.)

\begin{thm}\label{lb18}
After relabeling, $\scr E^\loc=(\mc A,\mc F,\boxdot,\fk H)$ can be extended to a unique closed vector-labeled categorical extension $\ovl{\scr E}=(\mc A,\scr C,\boxdot,\mc H)$.
\end{thm}

\begin{proof}

Step 1. Assume without loss of generality that $\scr E^\loc$ is vector-labeled. We first prove the uniqueness. Let $\mc F^\oplus$ be the class of all $\mc H_i\in\scr C$ which is equivalent to a finite direct sum of $\mc A$-modules in $\mc F$. Assume $\ovl{\scr E}=(\mc A,\scr C,\boxdot,\mc H)$ is a closed categorical extension containing $\scr E^\loc$. Then for any $\mc H_{i_0}\in\mc F,\wtd I\in\Jtd,\xi\in\mc H_{i_0}(I)$, the operator $L(\xi,\wtd I)$ is uniquely determined by $\scr E^\loc$ due to equation \eqref{eq50} (with $\fk a$ replaced by $\xi$) and the fact that $\fk H_i(\wtd I)$ spans a dense subspace of $\mc H_i$. 

Now if $\mc H_i\in\mc F^\oplus$, we can find $\mc H_{i_1},\dots,\mc H_{i_m}\in\mc F$, and isometric $U_1\in\Hom_{\mc A}(\mc H_{i_1},\mc H_i),\dots,U_m\in\Hom_{\mc A}(\mc H_{i_m},\mc H_i)$, such that $U_1U_1^*,\dots,U_mU_m^*$ are orthogonal projections, and $U_1U_1^*+\dots+U_mU_m^*=\id_i$. Then by
\eqref{eq49}, for any $\mc H_k\in\scr C,\xi\in\mc H_i(I),\chi\in\mc H_k$,
\begin{align}
L(\xi,\wtd I)\chi=(U_1\otimes\id_k)L(U_1^*\xi,\wtd I)\chi+\cdots+(U_m\otimes\id_k)L(U_m^*\xi,\wtd I)\chi.\label{eq56}
\end{align}
Therefore, $L(\xi,\wtd I)$ is uniquely determined by $\scr E^\loc$.

Next, if $\mc H_{i_1},\dots,\mc H_{i_m}\in\mc F^\oplus$, then by theorem \ref{lb15}-(b), for any $\xi^{(i_1)}\in\mc H_{i_1}(I),\dots,\xi^{(i_m)}\in\mc H_{i_m}(I)$,
\begin{align}
L(L(\xi^{(i_m)},\wtd I)\cdots L(\xi^{(i_2)},\wtd I)\xi^{(i_1)},\wtd I)=L(\xi^{(i_m)},\wtd I)\cdots L(\xi^{(i_2)},\wtd I)L(\xi^{(i_1)},\wtd I).\label{eq57}
\end{align}
Therefore, since vectors of the form $L(\xi^{(i_m)},\wtd I)\cdots L(\xi^{(i_2)},\wtd I)\xi^{(i_1)}$ span a dense subspace of $\mc H_{i_m}\boxdot\cdots\boxdot\mc H_{i_1}$, the left actions of $\mc H_{i_m}\boxdot\cdots\boxdot\mc H_{i_1}$ on modules in $\scr C$ are determined by $\scr E^\loc$.

Finally, for any $\mc H_i\in\scr C$, we can find $\mc H_{i_1},\dots,\mc H_{i_m}\in\mc F^\oplus$ and an isometric $F\in\Hom_{\mc A}(\mc H_i,\mc H_{i_m}\boxdot\cdots\boxdot\mc H_{i_1})$. Since $F^*F=\id_i$,  by \eqref{eq49}, for any $\wtd I\in\Jtd,\mc H_k\in\scr C,\xi\in\mc H_i(I),\chi\in\mc H_k$,
\begin{align}
L(\xi,\wtd I)\chi=(F^*\otimes\id_k)L(F\xi,\wtd I)\chi.\label{eq58}
\end{align}
Therefore $L(\xi,\wtd I)$ is determined by $\scr E^\loc$. As $R(\xi,\wtd I)$ is related to $L(\xi,\wtd I)$ by $\ss$, $R(\xi,\wtd I)$ is also determined by $\scr E^\loc$. Thus the uniqueness is proved.\\

Step 2. We prove the existence. Let us first extend $\scr E^\loc$ to a categorical local extension  $\scr E^\loc_1=(\mc A,\mc F^\oplus,\boxdot,\fk K)$. For any $\mc H_i\in\mc F^\oplus$, we choose $\mc H_{i_1},\dots,\mc H_{i_m}\in\mc F$ and $U_1,\dots,U_m$ as in step 1. If it turns out that $\mc H_i\in\mc F$, then we choose $m=1,\mc H_{i_1}=\mc H_i,U_1=\id_i$. Now for any $\wtd I\in\Jtd$, we define $\fk K_i(\wtd I)=\fk H_{i_1}(\wtd I)\sqcup\cdots\sqcup\fk H_{i_m}(\wtd I)$. If $s=1,2,\dots,m$ and $\fk k=\fk a_s\in\fk H_{i_s}(\wtd I)\subset\fk K_i(\wtd I)$, we set, for any $\mc H_k\in\scr C,\chi\in\mc H_k$,
\begin{gather}
L(\fk k,\wtd I)\chi=(U_s\otimes\id_k)L(\fk a_s,\wtd I),\\
 R(\fk k,\wtd I)\chi=\ss_{i,k}L(\fk k,\wtd I)=(\id_k\otimes U_s)R(\fk a_s,\wtd I).\label{eq85}
\end{gather}
(Note that the functoriality of $\ss$ is used in the second equation of \eqref{eq85}.) Then one can easily check that $\scr E^\loc_1=(\mc A,\mc F^\oplus,\boxdot,\fk K)$ satisfies all the axioms of a categorical local extension.

Let $\mc P$ be the class of all $\mc H_i\in\scr C$ equaling $\mc H_{i_1}\boxdot\cdots\boxdot\mc H_{i_m}$ for some $m=1,2,\dots$ and $\mc H_{i_1},\dots,\mc H_{i_m}\in\mc F^\oplus$.     We now extend $\scr E^\loc_1$ to a categorical local extension $\scr E^\loc_2=(\mc A,\mc P,\boxdot,\fk M)$. For each $\mc H_i\in\mc P$, we choose $\mc H_{i_1},\dots,\mc H_{i_m}\in\mc F^\oplus$ such that $\mc H_i=\mc H_{i_m}\boxdot\cdots\boxdot\mc H_{i_1}$. For any $\wtd I\in\Jtd$, we define $\Jtd_m(\wtd I)$ to be the set of all $(\wtd I_1,\wtd I_2,\dots,\wtd I_m)\in\Jtd^{\times m}$ satisfying that $\wtd I_1,\dots,\wtd I_m\subset\wtd I$, and that $\wtd I_s$ is anticlockwise to $\wtd I_{s-1}$ for any $s=2,\dots,m$. We now set
\begin{align*}
\fk M_i(\wtd I)=\coprod_{(\wtd I_1,\dots,\wtd I_m)\in\Jtd_m(\wtd I)}\fk K_{i_1}(\wtd I_1)\times\cdots\times\fk K_{i_m}(\wtd I_m),
\end{align*}
and define, for any $\mc H_k\in\scr C$ and $\fk m=(\fk a_1,\cdots,\fk a_m)\in\fk K_{i_1}(\wtd I_1)\times\cdots\times\fk K_{i_m}(\wtd I_m)\subset\fk M_i(\wtd I) $, bounded linear operators $L(\fk m,\wtd I)\in\Hom_{\mc A(I^c)}(\mc H_k,\mc H_i\boxdot\mc H_k), R(\fk m,\wtd I)\in\Hom_{\mc A(T^c)}(\mc H_k,\mc H_k\boxdot\mc H_i)$, such that for any $\chi\in\mc H_k$,
\begin{gather}
L(\fk m,\wtd I)\chi=L(\fk a_m,\wtd I_m)\cdots L(\fk a_1,\wtd I_1)\chi,\\
R(\fk m,\wtd I)\chi=\ss_{i,k}L(\fk m,\wtd I)\chi.
\end{gather}
Then all the axioms of a categorical local extension, except the locality, are easy to verify for $\scr E^\loc_2$. We now show that
\begin{align}
R(\fk m,\wtd I)\chi=R(\fk a_1,\wtd I_1)\cdots R(\fk a_m,\wtd I_m)\chi.\label{eq54}
\end{align}
Then the locality of $\scr E^\loc_2$ follows immediately from that of $\scr E^\loc_1$.

Let us prove \eqref{eq54} when $m=3$. For general $m$ the argument is similar. By the coherence theorem for braided tensor categories, we have
\begin{align*}
\ss_{i,k}=\ss_{i_3\boxdot i_2\boxdot i_1,k}=(\ss_{i_3,k}\otimes\id_{i_2}\otimes\id_{i_1}) (\id_{i_3}\otimes\ss_{i_2,k}\otimes\id_{i_1}) (\id_{i_3}\otimes\id_{i_2}\otimes\ss_{i_1,k}).
\end{align*}
Therefore the action of $\ss_{i,k}$ on $L(\fk m,\wtd I)\chi$ is
\begin{align*}
&\quad L(\fk m,\wtd I)\chi=L(\fk a_3,\wtd I_3)L(\fk a_2,\wtd I_2) L(\fk a_1,\wtd I_1)\chi\xrightarrow{\id_{i_3}\otimes\id_{i_2}\otimes\ss_{i_1,k}} L(\fk a_3,\wtd I_3)L(\fk a_2,\wtd I_2) \ss_{i_1,k}L(\fk a_1,\wtd I_1)\chi\\
&=L(\fk a_3,\wtd I_3)L(\fk a_2,\wtd I_2)R(\fk a_1,\wtd I_1)\chi=R(\fk a_1,\wtd I_1)L(\fk a_3,\wtd I_3)L(\fk a_2,\wtd I_2)\chi\\
&\xrightarrow{\id_{i_3}\otimes\ss_{i_2,k}\otimes\id_{i_1}}  R(\fk a_1,\wtd I_1)(\id_{i_3}\otimes\ss_{i_2,k})L(\fk a_3,\wtd I_3)L(\fk a_2,\wtd I_2)\chi=R(\fk a_1,\wtd I_1)L(\fk a_3,\wtd I_3)\ss_{i_2,k}L(\fk a_2,\wtd I_2)\chi\\
&=R(\fk a_1,\wtd I_1)L(\fk a_3,\wtd I_3)R(\fk a_2,\wtd I_2)\chi=R(\fk a_1,\wtd I_1)R(\fk a_2,\wtd I_2)L(\fk a_3,\wtd I_3)\chi\\
&\xrightarrow{\ss_{i_3,k}\otimes\id_{i_2}\otimes\id_{i_1}}  R(\fk a_1,\wtd I_1)R(\fk a_2,\wtd I_2)\ss_{i_3,k}L(\fk a_3,\wtd I_3)\chi=R(\fk a_1,\wtd I_1)R(\fk a_2,\wtd I_2)R(\fk a_3,\wtd I_3)\chi.
\end{align*}
Hence \eqref{eq54} is proved.

Finally we extend $\scr E^\loc_2$ to a categorical extension $\scr E=(\mc A,\scr C,\boxdot,\fk N)$. This will finish our proof. For any $\mc H_i\in\scr C$, we can find $\mc H_{i_0}\in\mc P$ and an isometric $F\in\Hom_{\mc A}(\mc H_i,\mc H_{i_0})$. For any $\wtd I\in\Jtd$, we set $\fk N_i(\wtd I)=\fk M_{i_0}(\wtd I)$. If $\fk n=\fk m \in\fk M_{i_0}(\wtd I)=\fk N_i(\wtd I)$ and $\mc H_k\in\scr C$, we define $L(\fk m,\wtd I)\in\Hom_{\mc A(I^c)}(\mc H_k,\mc H_i\boxdot\mc H_k), R(\fk m,\wtd I)\in\Hom_{\mc A(I^c)}(\mc H_k,\mc H_k\boxdot\mc H_i)$ satisfying that for any $\chi\in\mc H_k$,
\begin{gather*}
L(\fk n,\wtd I)\chi=(F^*\otimes\id_k)L(\fk m,\wtd I)\chi,\\
R(\fk n,\wtd I)\chi=\ss_{i,k}L(\fk n,\wtd I)\chi=(\id_k\otimes F^*)R(\fk m,\wtd I)\chi.
\end{gather*}
This construction makes $\scr E$ a categorical extension of $\mc A$. Its closure $\ovl{\scr E}=(\mc A,\scr C,\boxdot,\mc H)$ is the desired vector-labeled closed categorical extension containing $\scr E^\loc$. 
\end{proof}
{~}\\

We now give an application of this theorem. 
\begin{df}\label{lb42}
A \textbf{left operator} of $\scr E^\loc=(\mc A,\mc F,\boxdot,\fk H)$ is a quadruple $(A,\fk a,\wtd I,\mc H_i)$, where $\fk a$ is an element, $\mc H_i\in\scr C,\wtd I\in\Jtd$, and for any $\mc H_k\in\scr C$, there is a bounded linear operator $A(\fk a,\wtd I)\in\Hom_{\mc A(I^c)}(\mc H_k,\mc H_i\boxdot\mc H_k)$, such that the following conditions are satisfied:\\
(a) If $\mc H_k,\mc H_{k'}\in\scr C$, $F\in\Hom_{\mc A}(\mc H_k,\mc H_{k'})$, then  the following diagram commutes.
\begin{gather}
\begin{CD}
\mc H_k @>F>> \mc H_{k'}\\
@V A(\fk a,\wtd I)  VV @V A(\fk a,\wtd I)  VV\\
\mc H_i\boxdot\mc H_k @> \id_i\otimes F>> \mc H_i\boxdot\mc H_{k'}
\end{CD}.
\end{gather}
(b) For any $\mc H_j\in\mc F,\mc H_k\in\scr C$,  $\wtd J\in\Jtd$  clockwise to $\wtd I$, and any $\fk b_0\in\fk H_j(\wtd J)$, the following diagram  commutes (not necessarily adjointly).
\begin{align}
\begin{CD}
\mc H_k @> \quad R(\fk b_0,\wtd J)\quad   >> \mc H_k\boxdot\mc H_j\\
@V A(\fk a,\wtd I)   V  V @V A(\fk a,\wtd I) VV\\
\mc H_i\boxdot\mc H_k @> \quad R(\fk b_0,\wtd J) \quad  >> \mc H_i\boxdot\mc H_k\boxdot\mc H_j
\end{CD}.
\end{align}

Similarly, a \textbf{right operator} of $\scr E^\loc$ is a quadruple $(B,\fk b,\wtd J,\mc H_j)$, where $\fk b$ is an element, $\mc H_j\in\scr C,\wtd J\in\Jtd$, and for any $\mc H_k\in\scr C$, there is a bounded linear operator $B(\fk b,\wtd J)\in\Hom_{\mc A(J^c)}(\mc H_k,\mc H_k\boxdot\mc H_j)$, such that the following conditions are satisfied:\\
(a) If $\mc H_k,\mc H_{k'}\in\scr C$, $F\in\Hom_{\mc A}(\mc H_k,\mc H_{k'})$, then  the following diagram commutes.
\begin{gather}
\begin{CD}
\mc H_k @> B(\fk b,\wtd J)  >> \mc H_k\boxdot\mc H_j\\
@V F VV @V F\otimes\id_j  VV\\
\mc H_{k'} @>B(\fk b,\wtd J) >> \mc H_{k'}\boxdot\mc H_j
\end{CD}.
\end{gather}
(b) For any $\mc H_i\in\mc F,\mc H_k\in\scr C$,  $\wtd I\in\Jtd$  anticlockwise to $\wtd J$, and any $\fk a_0\in\fk H_i(\wtd I)$, the following diagram  commutes.
\begin{align}
\begin{CD}
\mc H_k @> \quad B(\fk b,\wtd J)\quad   >> \mc H_k\boxdot\mc H_j\\
@V L(\fk a_0,\wtd I)   V  V @V L(\fk a_0,\wtd I) VV\\
\mc H_i\boxdot\mc H_k @> \quad B(\fk b,\wtd J) \quad  >> \mc H_i\boxdot\mc H_k\boxdot\mc H_j
\end{CD}.
\end{align}
\end{df}

\begin{thm}\label{lb43}
Let $(A,\fk a,\wtd I,\mc H_i)$ and $(B,\fk b,\wtd J,\mc H_j)$ be a left operator and a right operator of $\scr E^\loc=(\mc A,\mc F,\boxdot,\fk H)$, where $\wtd I$ is anticlockwise to $\wtd J$, and $\mc H_i,\mc H_j\in\scr C$. Then these two operators commute adjointly, in the sense that for any $\mc H_k\in\scr C$, the following diagram commutes adjointly.
\begin{align}
\begin{CD}
\mc H_k @> \quad B(\fk b,\wtd J)\quad   >> \mc H_k\boxdot\mc H_j\\
@V A(\fk a,\wtd I)   V  V @V A(\fk a,\wtd I) VV\\
\mc H_i\boxdot\mc H_k @> \quad B(\fk b,\wtd J) \quad  >> \mc H_i\boxdot\mc H_k\boxdot\mc H_j
\end{CD}.
\end{align}
\end{thm}

\begin{proof}
By step 2 of the proof of theorem \ref{lb18}, we can construct a categorical extension $\scr E=(\mc A,\scr C,\boxdot,\fk N)$ such that any left (resp. right) operator of $\scr E^\loc$ is also a left (resp. right) operator of $\scr E$. Let $\ovl{\scr E}$ be the closure of $\scr E$. Set $\xi=A(\fk a,\wtd I)\Omega\in\mc H_i(I)$. Then for any $\mc H_k\in\scr C,\fk c\in\fk N_k(\wtd J)$,
\begin{align*}
A(\fk a,\wtd I)R(\fk c,\wtd J)\Omega=R(\fk c,\wtd J)A(\fk a,\wtd I)\Omega=R(\fk c,\wtd J)\xi=R(\fk c,\wtd J)L(\xi,\wtd I)\Omega=L(\xi,\wtd I)R(\fk c,\wtd J)\Omega.
\end{align*}
Therefore $A(\fk a,\wtd I)$ equals $L(\xi,\wtd I)$ when acting on any $\mc H_k\in\scr C$. Similarly, if we let $\eta=B(\fk b,\wtd J)\Omega$, then $B(\fk b,\wtd J)$ equals $R(\eta,\wtd J)$. Therefore $A(\fk a,\wtd I)$ and $B(\fk b,\wtd J)$ commute adjointly.
\end{proof}

\section{VOAs and categorical extensions}

\subsection{Tensor categories of VOAs}\label{lb22}

We review the construction of tensor categories for ``rational" vertex operator algebras (VOAs) by Huang-Lepowsky. See \cite{Hua08b} for more details. The reader is also referred to \cite{Gui17a} section 2.4 for a  sketch of this construction. The notations we will use in this paper are close to those in \cite{Gui17a}.

Let $V$ be  a VOA of CFT type. This means that $V$ has grading  $V=\bigoplus_{n\in\mathbb Z_{\geq0}}V(n)$ with $V(0)=\mathbb C\Omega$ where $\Omega$ is the vacuum vector of $V$.  We let $\Rep^\ssp(V)$ be the category of semisimple $V$-modules. So if $W_i\in\Rep^\ssp(V)$, we assume that $W_i$ is a finite direct sum of irreducible $V$-modules. In this paper, unless otherwise stated, a $V$-module $W_k$ is always assumed to be semisimple. We also assume that the eigenvalues of $L_0$ on $W_i$ are real. For any $s\in\mathbb R$, we let $P_s$ be the projection of $W_i$ onto the $s$-eigenspace $W_i(s)$ of $L_0$.  If $w^{(i)}\in W_i$ is homogeneous (i.e., an eigenvector of $L_0$), we let $\Delta_{w^{(i)}}$ be the conformal weight (the corresponding eigenvalue) of $w^{(i)}$.  A vector $w^{(i)}\in W_i$ is called quasi-primary if it is homogeneous and $L_1w^{(i)}=0$. 

For any $W_i,W_j,W_k\in\Rep^\ssp(V)$, we let $\mc V{k\choose i~j}=\mc V{W_k\choose W_iW_j}$ be the vector space of intertwining operators of $V$. (See \cite{FHL93} for the general theory of intertwining operators of VOAs.) If $\mc Y_\alpha\in\mc V{k\choose i~j}$, we call $W_i,W_j,W_k$ the charge space, the source space, and the target space of $\mc Y_\alpha$ respectively. We assume the following:

\begin{cond}\label{CondA}
For any $W_i,W_j,W_k\in\Rep^\ssp(V)$, the vector space $\mc V{k\choose i~j}$ is finite dimensional.
\end{cond}
For each equivalence class of irreducible $V$-modules, we choose a representing module, and  $\mc E$ be the set of these $V$-modules. We let $\mc E$ contain the vacuum module $W_0=V$. We shall also write $i\in\mc E$ if $W_i\in\mc E$. The second condition we require on $V$ is:
\begin{cond}\label{lb19}
For any $W_i,W_j\in\Rep^\ssp(V)$, there are only finitely many $W_k\in\mc E$ satisfying $\dim\mc V{k\choose i~j}>0$.
\end{cond}

Now we can define a tensor bifunctor $\boxtimes$ on $\Rep^\ssp(V)$. For any $W_i,W_j\in\Rep^\ssp(V)$, we define
\begin{align}
W_{ij}\equiv W_i\boxtimes W_j=\bigoplus_{k\in\mc E}\mc V{k\choose i~j}^*\otimes W_k\label{eq60}
\end{align}
where $\mc V{k\choose i~j}^*$ is the dual vector space of $\mc V{k\choose i~j}$. Note that here $\mc V{k\choose i~j}$ is finite dimensional, and the sum of $k$ is finite by condition \ref{lb19}. The action of $V$ on $W_{ij}$ is
\begin{align}
Y_{ij}(v,x)=\bigoplus \id\otimes Y_k(v,x)\qquad(\forall v\in V)
\end{align}
where $Y_k$ is the vertex operator describing the action of $V$ on $W_k$, and $x$ is a formal variable. 

When $k\in\mc E$, any intertwining operator $\mc Y_\alpha\in\mc V{k\choose i~j}$ is naturally a linear map $\mc V{k\choose i~j}^*\rightarrow\mathbb C$, which can be extended naturally to a  homomorphism of $V$-modules $W_i\boxtimes W_j\rightarrow W_k$. For general $W_k\in\Rep^\ssp(V)$, $\mc V{k\choose i~j}$ can also be identified with $\Hom_V(W_i\boxtimes W_j, W_k)$ using the following identifications
\begin{gather*}
\mc V{k\choose i~j}\simeq\bigoplus_{t\in\mc E}\mc V{t\choose i~j}\otimes\Hom_V(W_t,W_k),\\
\Hom_V(W_i\boxtimes W_j, W_k)\simeq\bigoplus_{t\in\mc E}\Hom_V(W_i\boxtimes W_j,W_t)\otimes\Hom_V(W_t,W_k).
\end{gather*}

The tensor product of $F\in\Hom_V(W_{i_1},W_{i_2}),G\in\Hom_V(W_{j_1},W_{j_2})$ is defined in the following way. For each $k\in\mathcal E$ we have a linear map $(F\otimes G)^\tr:\mathcal V{k\choose i_2~j_2}\rightarrow \mathcal V{k\choose i_1~j_1}$, such that if $\mathcal Y\in\mathcal V{k\choose i_2~j_2}$, then  $(F\otimes G)^\tr\mathcal Y\in\mathcal V{k\choose i_1~j_1}$, and for any $w^{(i_1)}\in W_{i_1},w^{(j_1)}\in W_{j_1}$,
\begin{align}
\big((F\otimes G)^\tr\mathcal Y\big)(w^{(i_1)},x)w^{(j_1)}=\mathcal Y(Fw^{(i_1)},x)Gw^{(j_1)}.
\end{align}
Then $F\otimes G:\mathcal V{k\choose i_1~j_1}^*\rightarrow\mathcal V{k\choose i_2~j_2} ^*$ is defined  to be the transpose of $(F\otimes G)^\tr$, and can be extended to a homomorphism $$F\otimes G=\bigoplus_{k\in\mathcal E}(F\otimes G)\otimes \id_{k}:W_{i_1}\boxtimes W_{j_1}\rightarrow W_{i_2}\boxtimes W_{j_2}.$$  

To construct associativity and braid isomorphisms, we need to consider products and iterations of intertwining operators. For any $V$-module $W_i$, we let $W_{\overline i}$ denote its contragredient module. A sequence of intertwining operators $\mc Y_{\alpha_1},\dots,\mc Y_{\alpha_n}$ of $V$ is called a chain of intertwining operators, if the source space of $\mc Y_{\alpha_m}$ equals the target space of $\mc Y_{\alpha_{m-1}}$ for any $m=2,3,\dots,n$. We will also take the complex-analytic point of view instead of the formal one for intertwining operators. For any $\mc Y_\alpha\in\mc V{k\choose i~j}$,  and any $w^{(i)}\in W_i,w^{(j)}\in W_j,w^{(\ovl k)}\in W_{\ovl k},z\in\mathbb C^\times=\mathbb C\setminus\{0\}$,
\begin{align}
\bk {\mc Y_\alpha(w^{(i)},z)w^{(j)},w^{(\overline k)}}=\bk {\mc Y_\alpha(w^{(i)},x)w^{(j)},w^{(\overline k)}}\big|_{x=z}
\end{align}
depends not only on $z$ but also on the arg-value $\arg z$ of $z$. We regard $\mc Y_\alpha(w^{(i)},z)w^{(j)}$ as a vector in the algebraic completion $\wht{W_k}$ of $W_k$, which is also the dual vector space of $W_{\ovl k}$.

\begin{cond}\label{CondC}
Let $\mathcal Y_{\alpha_1},\dots,\mathcal Y_{\alpha_n}$ be an arbitrary chain of intertwining operators of $V$. For each $1\leq m\leq n$, we let $W_{i_m}$ be the charge space of $\mathcal Y_{\alpha_m}$. Let $W_{i_0}$ be the source space of $\mathcal Y_{\alpha_1}$, and let $W_k$ be the target space  of $\mathcal Y_{\alpha_n}$. Then for any $w^{(i_0)}\in W_{i_0},w^{(i_1)}\in W_{i_1},\dots, w^{(i_n)}\in W_{i_n},w^{(\overline k)}\in W_{\overline k}$, and arg-valued $z_1,z_2,\dots,z_n\in\mathbb C$ such that $0<|z_1|<|z_2|<\dots<|z_n|$, the expression
\begin{align}
\langle \mathcal Y_{\alpha_n}(w^{(i_n)},z_n)\mathcal Y_{\alpha_{n-1}}(w^{(i_{n-1})},z_{n-1})\cdots\mathcal Y_{\alpha_1}(w^{(i_1)},z_1)w^{(i_0)}, w^{(\overline k)}    \rangle\label{eq59}
\end{align}
converges absolutely and locally uniformly, which means that there exists a neighborhood $U$ of $(z_1,\dots,z_n)$ in the $n$-th order configuration space $\Conf_n(\mathbb C^\times)$ of $\mathbb C^\times=\mathbb C\setminus\{0\}$, such that  the series
\begin{align}
&\sum_{s_1,s_2,\dots,s_{n-1}\in\mathbb R}\big|\langle \mathcal Y_{\alpha_n}(w^{(i_n)},\zeta_n)P_{s_{n-1}}\mathcal Y_{\alpha_{n-1}}(w^{(i_{n-1})},\zeta_{n-1})P_{s_{n-2}}\nonumber\\
&\qquad\qquad\qquad\cdots P_{s_1}\mathcal Y_{\alpha_1}(w^{(i_1)},\zeta_1)w^{(i_0)}, w^{(\overline k)}    \rangle\big|
\end{align}
converges and is uniformly bounded for all $(\zeta_1,\dots,\zeta_n)\in U$. Here each $P_{s_m}$ ($1\leq m\leq n-1$) is the  projection of the  target space of $\mathcal Y_{\alpha_m}$ onto its weight-$s_m$ component. Moreover, the function locally defined by \eqref{eq59} for any $(z_1,\dots,z_n)\in U$ can be analytically continued to a multi-valued holomorphic function on $\Conf_n(\mathbb C^\times)$.
\end{cond}

Thus one can construct all genus-$0$ conformal blocks using products of intertwining operators.  Using the braid operators $B_\pm:\mc V{k\choose i~j}\rightarrow\mc V{k\choose j~i}$ defined by
\begin{align}
\mc Y_{B_\pm\alpha}(w^{(j)},z)w^{(i)}\equiv (B_\pm\mc Y_\alpha)(w^{(j)},z)w^{(i)}=e^{xL_{-1}}\mc Y_\alpha(w^{(i)},e^{\pm i\pi}z)w^{(j)}
\end{align}
(where $\arg (e^{\pm i\pi z})=\pm\pi+\arg z$) for any $\mc Y_\alpha\in\mc V{k\choose i~j},w^{(i)}\in W_i,w^{(j)}\in W_j$, one can show that for any $\mc Y_\alpha\in\mc V{k\choose i~j},\mc Y_\beta\in\mc V{t\choose k~s}, w^{(i)}\in W_i,w^{(j)}\in W_j$, the iteration of two intertwining operators
\begin{align*}
\mc Y_\beta(\mc Y_\alpha(w^{(i)},z_i-z_j)w^{(j)},z_j)
\end{align*}
converges absolutely and locally uniformly whenever $0<|z_i-z_j|<|z_j|$, in the sense that for any $w^{(s)}\in W_s,w^{(\ovl t)}\in W_{\ovl t}$, there exits a neighborhood $U\subset\Conf_2(\mathbb C^\times)$ of $(z_i,z_j)$, such that  the series
\begin{align*}
\sum_{a\in\mathbb R}\big|\bk {\mc Y_\beta(P_a\mc Y_\alpha(w^{(i)},z_i-z_j)w^{(j)},z_j)w^{(s)}|w^{(\ovl t)}}\big|
\end{align*}
converges and is uniformly bounded for all $(\zeta_i,\zeta_j)\in U$. We assume that iterations and products of intertwining operators can be related in the following way:

\begin{cond}\label{lb21}
Let $\mc Y_\alpha\in\mc V{t\choose i~r},\mc Y_\beta\in\mc V{r\choose j~k}$,  $z_i,z_j\in\mathbb C^\times$ satisfying $0<|z_i-z_j|<|z_j|<|z_i|$, and choose $\arg z_i,\arg z_j,\arg (z_i-z_j)$. Then there exist $W_s\in\Rep^\ssp(V),\mc Y_{\alpha'}\in\mc V{s\choose i~j},\mc Y_{\beta'}\in\mc V{t\choose s~k}$, and a neighborhood $U\subset\Conf_2(\mathbb C^\times)$ of $(z_i,z_j)$, such that for any $w^{(i)}\in W_i,w^{(j)}\in W_j$, and any $(\zeta_i,\zeta_j)\in U$ with $\arg \zeta_i,\arg \zeta_j,\arg(\zeta_i-\zeta_j)$ close to $\arg z_i,\arg z_j,\arg(z_i-z_j)$ respectively, the following equation holds when acting on $W_k$:
\begin{align}
\mc Y_\alpha(w^{(i)},\zeta_i)\mc Y_\beta(w^{(j)},\zeta_j)=\mc Y_{\beta'}(\mc Y_{\alpha'}(w^{(i)},\zeta_i-\zeta_j)w^{(j)},\zeta_j)
\end{align}
\end{cond}
Thus, products of intertwining operators can be written as iterations of intertwining operators. Using the braid operators $B_\pm$, one can easily prove the converse statement, i.e., iterations can be expressed as products of intertwining operators.\\

We are now ready to define the associativity isomorphism. Given three $V$-modules $W_i,W_j,W_k$,  we have natural identifications
\begin{gather}
(W_i\boxtimes W_j)\boxtimes W_k=\bigoplus_{s,t\in\mathcal E}\mathcal V{t\choose s~k}^*\otimes\mathcal V{s\choose i~j}^* \otimes W_t,\\
W_i\boxtimes (W_j\boxtimes W_k)=\bigoplus_{r,t\in\mathcal E} \mathcal V{t\choose i~r}^*\otimes\mathcal V{r\choose j~k}^*\otimes W_t.
\end{gather}
Choose basis $\Theta^t_{sk},\Theta^s_{ij},\Theta^t_{ir},\Theta^r_{jk}$ of these vector spaces of intertwining operators. Choose arg-valued $z_i,z_j\in\mathbb C^\times$ satisfying $0<|z_i-z_j|<|z_j|<|z_i|$ and $\arg z_i=\arg z_j=\arg(z_i-z_j)$.  For any $r,t\in\mathcal E,\alpha\in\Theta^t_{ir},\beta\in\Theta^r_{jk}$, there exist unique complex numbers $F^{\beta'\alpha'}_{\alpha\beta}$ ($\forall s\in\mc E,\alpha'\in\Theta^s_{ij},\beta'\in\Theta^t_{sk}$) independent of the choice of $z_i,z_j$, such that for any $w^{(i)}\in W_i,w^{(j)}\in W_j$, we have the \textbf{fusion relation}
\begin{align}
\mathcal Y_\alpha(w^{(i)},z_i)\mathcal Y_\beta(w^{(j)},z_j)=\sum_{s\in\mc E}\sum_{\alpha'\in\Theta^s_{ij},\beta'\in\Theta^t_{sk}}F^{\beta'\alpha'}_{\alpha\beta}\mathcal Y_{\beta'}(\mathcal Y_{\alpha'}(w^{(i)},z_i-z_j)w^{(j)},z_j).\label{eq68}
\end{align}
Thus  $\{F^{\beta'\alpha'}_{\alpha\beta}\}$ can be regarded as a matrix,  called \textbf{fusion matrix}. This matrix is invertible, as one can write iterations as products of intertwining operators, which gives the inverse matrix of $\{F^{\beta'\alpha'}_{\alpha\beta}\}$. For each $t\in\mc E$, define an isomorphism
\begin{gather}
A^\tr:\bigoplus_{r\in\mathcal E} \mathcal V{t\choose i~r}\otimes\mathcal V{r\choose j~k}\rightarrow\bigoplus_{s\in\mathcal E} \mathcal V{t\choose s~k}\otimes \mathcal V{s\choose i~j},\nonumber\\
\mathcal Y_\alpha\otimes \mathcal Y_\beta\mapsto\sum_{s\in\mc E} \sum_{\alpha'\in\Theta^s_{ij},\beta'\in\Theta^t_{sk}}F^{\beta'\alpha'}_{\alpha\beta}\mathcal Y_{\beta'}\otimes\mathcal Y_{\alpha'}.\label{eq69}
\end{gather}
Then $A^\tr$ is independent of the basis chosen. Define
\begin{equation}
A:\bigoplus_{s\in\mathcal E} \mathcal V{t\choose s~k}^*\otimes \mathcal V{s\choose i~j}^*\rightarrow\bigoplus_{r\in\mathcal E} \mathcal V{t\choose i~r}^*\otimes\mathcal V{r\choose j~k}^*
\end{equation}
to be the transpose of $A^\tr$, and extend it to
\begin{align}
A=\sum_{t\in\mathcal E}A\otimes\id_t: (W_i\boxtimes W_j)\boxtimes W_k\rightarrow W_i\boxtimes (W_j\boxtimes W_k).
\end{align}
Then $A$ is our associativity map. One can easily verify that $A$ is functorial. Moreover, $A$ satisfies the pentagon axiom.

Let $W_0=V$ be the identity object of $\Rep^\ssp(V)$. Then the isomorphism $W_0\boxtimes W_i=V\boxtimes W_i\rightarrow W_i$ is defined using the intertwining operator $Y_i\in\mc V{i\choose 0~i}$. Similarly, the isomorphism $W_i\boxtimes V\rightarrow W_i$ is defined using the \textbf{creation operator} $\mc Y_{\kappa(i)}$ of $W_i$, which is a type $i\choose i~0$ intertwining operator defined by $\mc Y_{\kappa(i)}=B_+ Y_i=B_-Y_i$. These isomorphisms satisfy the triangle axioms. Therefore $\Rep^\ssp(V)$ is a monoidal category.

\begin{cond}\label{lb24}
The monoidal category $\Rep^\ssp(V)$ is rigid, i.e., every object has left and right duals.
\end{cond}

Finally we define the braiding. Let $\ss_{i,j}:\mathcal V{k\choose i~j}^*\rightarrow \mathcal V{k\choose j~i}^*$ be the transpose of 
\begin{align*}
B_+:\mc V{k\choose j~i}\rightarrow\mc V{k\choose i~j},\qquad \mc Y\mapsto B_+\mc Y,
\end{align*}
and extend it to a morphism
\begin{align}
\ss_{i,j}=\sum_{t\in\mathcal E}\ss_{i,j}\otimes\id_t:W_i\boxtimes W_j\rightarrow W_j\boxtimes W_i.
\end{align}
This gives the \textbf{braid operator}. These braid operators satisfy the hexagon axioms. Therefore $\Rep^\ssp(V)$ becomes a rigid braided tensor category. We also define, for any $W_i\in\Rep^\ssp(V)$, the twist $\vartheta_i\in\End_V(W_i)$ to be the action of $e^{2i\pi L_0}$ on $W_i$. Then $\Rep^\ssp(V)$ becomes a rigid ribbon category (cf. \cite{Hua08b} theorem 4.1). In the case that there are only finitely many equivalence classes of irreducible $V$-modules (i.e., $\mc E$ is a finite set), $\Rep^\ssp(V)$ is a ribbon fusion category. \\

We now relate $\ss$ with the braid relations of intertwining operators. Let $\mc Y_\alpha\in\mc V{t\choose i~r},\mc Y_\beta\in\mc V{r\choose j~k}$, $z_i,z_j\in \mathbb C^\times$ with $|z_i|=|z_j|$. Choose $\arg z_i,\arg z_j$. Then for any $w^{(i)}\in W_i,w^{(j)}\in W_j$, the expression
$\mc Y_\alpha(w^{(i)},z_i)\mc Y_\beta(w^{(j)},z_j)$ is understood as a linear functional on $W_k\otimes W_{\ovl t}$ defined in the following way. If $w^{(k)}\in W_k,w^{(\ovl t)}\in W_{\ovl t}$, then
\begin{align*}
\bk{\mc Y_\alpha(w^{(i)},z_i)\mc Y_\beta(w^{(j)},z_j)w^{(k)},w^{(\ovl t)}}=\lim_{\lambda\searrow 1}\bk{\mc Y_\alpha(w^{(i)},\lambda z_i)\mc Y_\beta(w^{(j)},z_j)w^{(k)},w^{(\ovl t)}},
\end{align*}
(We set $\arg(\lambda z_i)=\arg z_i$ when $\lambda>0$.) the right hand side of which is definable due to condition \ref{CondC}.

For each $r,s,t\in\mc E$, choose basis $\Theta^t_{ir},\Theta^r_{jk},\Theta^s_{ik},\Theta^t_{js}$ of $\mc V{t\choose i~r},\mc V{r\choose j~k},\mc V{s\choose i~k},\mc V{t\choose j~s}$ respectively. Using condition \ref{lb21} and the braid operators $B_\pm$, one can show that  for any $r,t\in\mathcal E,\alpha\in\Theta^t_{ir},\beta\in\Theta^r_{jk}$, there exist unique complex numbers $(B_\pm)^{\beta'\alpha'}_{\alpha\beta}$ ($\forall s\in\mc E,\alpha'\in\Theta^s_{ik},\beta'\in\Theta^t_{js}$), such that for any $w^{(i)}\in W_i,w^{(j)}\in W_j$, the following \textbf{braid relation} holds for any $w^{(i)}\in W_i,w^{(j)}\in W_j$
\begin{align}
\mc Y_\alpha(w^{(i)},z_i)\mc Y_\beta(w^{(j)},z_j)=\sum_{s\in\mc E}\sum_{\alpha'\in\Theta^s_{ik},\beta'\in\Theta^t_{js}}(B_\pm)^{\beta'\alpha'}_{\alpha\beta}\mc Y_{\beta'}(w^{(j)},z_j)\mc Y_{\alpha'}(w^{(i)},z_i),\label{eq65}
\end{align}
where the sign $\pm$ is $+$ if $z_i$ is anticlockwise to $z_j$ (i.e., $\arg z_j<\arg z_i<\arg z_j+2\pi$), and $-$ if $z_i$ is clockwise to $z_j$. We can regard $\{(B_\pm)^{\beta'\alpha'}_{\alpha\beta} \}$ as  matrices, called  \textbf{braid matrices}, which again are invertible. $\{(B_\pm)^{\beta'\alpha'}_{\alpha\beta} \}$ depend only on the clockwise or anticlockwise order of $z_i$ and $z_j$, but not on the specific choice of $z_i,z_j$. Since we have
\begin{gather}
W_i\boxtimes (W_j\boxtimes W_k)=\bigoplus_{r,t\in\mathcal E} \mathcal V{t\choose i~r}^*\otimes\mathcal V{r\choose j~k}^*\otimes W_t,\label{eq62}\\
W_j\boxtimes(W_i\boxtimes W_k)=\bigoplus_{s,t\in\mc E}\mc V{t\choose j~s}^*\otimes\mc V{s\choose i~k}^*\otimes W_t,\label{eq63}
\end{gather}
the isomorphism $\ss_{j,i}\otimes\id_k:W_j\boxtimes(W_i\boxtimes W_k)\rightarrow W_i\boxtimes (W_j\boxtimes W_k)$ restricts, for each $t\in\mc E$, to an isomorphism
\begin{align*}
\bigoplus_{s\in\mc E}\mc V{t\choose j~s}^*\otimes\mc V{s\choose i~k}^*\rightarrow\bigoplus_{r\in\mc E}\mathcal V{t\choose i~r}^*\otimes\mathcal V{r\choose j~k}^*.
\end{align*}
By (for instance) \cite{Gui17a} proposition 2.12, the transpose of this map is given by
\begin{gather}
\bigoplus_{r\in\mc E}\mathcal V{t\choose i~r}\otimes\mathcal V{r\choose j~k}\rightarrow \bigoplus_{s\in\mc E}\mc V{t\choose j~s}\otimes\mc V{s\choose i~k},\nonumber\\
\mc Y_\alpha\otimes\mc Y_\beta\mapsto \sum_{s\in\mc E}\sum_{\alpha'\in\Theta^s_{ik},\beta'\in\Theta^t_{js}}(B_+)^{\beta'\alpha'}_{\alpha\beta}\mc Y_{\beta'}\otimes\mc Y_{\alpha'}.\label{eq64}
\end{gather}

\subsection{The intertwining operators $\mc L_i$ and $\mc R_i$}

In this section, we define, for any $W_i\in\Rep^\ssp(V)$, a pair of intertwining operators $\mc L_i$ and $\mc R_i$ that are closely related to the operators $L$ and $R$ in categorical extensions. Recall the definition of tensor product $W_i\boxtimes W_k=\bigoplus_{t\in\mc E}\mc V{t\choose i~k}^* \otimes W_t$. If $W_i\in\Rep^\ssp(V)$, we let $\mc L_i$ act on  any $W_k\in\Rep^\ssp(V)$ as an intertwining operator of type ${ik\choose i~k}={W_i\boxtimes W_k\choose W_i~W_k}$, such that for any $w^{(i)}\in W_i,w^{(k)}\in W_k$ and any $t\in\mc E,w^{(\ovl t)}\in W_{\ovl t},\mc Y_\alpha\in\mc V{t\choose i~k}$,
\begin{align}
\bk{\mc L_i(w^{(i)},x)w^{(k)},\mc Y_\alpha\otimes w^{(\ovl t)}}=\bk{\mc Y_\alpha(w^{(i)},x)w^{(k)},w^{(\ovl t)}  }.
\end{align}
This relation is simply written as
\begin{align}
\bk{\mc L_i(w^{(i)},x)w^{(k)},\mc Y_\alpha}=\mc Y_\alpha(w^{(i)},x)w^{(k)}.\label{eq88}
\end{align}
Let $\Theta^t_{ik}$ be a basis of $\mc V{t\choose i~k}$, and let $\{\wch{\mc Y}^\alpha:\alpha\in\Theta^t_{ik} \}$ be the dual basis, where each $\wch{\mc Y}^\alpha$ is the dual element of $\mc Y_\alpha$. Then we have another description of $\mc L_i$: for any $w^{(i)}\in W_i,w^{(k)}\in W_k$,
\begin{align}
\mc L_i(w^{(i)},x)w^{(k)}=\sum_{t\in\mc E}\sum_{\alpha\in\Theta^t_{ik}}\wch{\mc Y}^\alpha\otimes\mc Y_\alpha(w^{(i)},x)w^{(k)}.\label{eq66}
\end{align}
We also define, for any $W_k\in\Rep^\ssp(V)$, a type ${ki\choose i~k}={W_k\boxtimes W_i\choose W_i~W_k}$ intertwining operator $\mc R_i$, such that for any $w^{(i)}\in W_i,w^{(k)}\in W_k$,
\begin{align}
\mc R_i(w^{(i)},x)w^{(k)}=\ss_{i,k} \mc L_i(w^{(i)},x)w^{(k)}.
\end{align}
We write $\mc L_i$ (resp. $\mc R_i$) as $\mc L_i|_k$ (resp. $\mc R_i|_k$) if we want to emphasize that the source space of $\mc L_i$ (resp. $\mc R_i$) is $W_k$.


It is easy to check that $\mc L$ and $\mc R$ are functorial, in the sense that for any $F\ \in\Hom_V(W_i,W_{i'}),G\in\Hom_V(W_k,W_{k'}),w^{(i)}\in W_i,w^{(k)}\in W_k$,
\begin{gather}
(F\otimes G)\mc L_i(w^{(i)},x)w^{(k)}=\mc L_{i'}(Fw^{(i)},x)Gw^{(k)},\label{eq70}\\
(G\otimes F)\mc R_i(w^{(i)},x)w^{(k)}=\mc R_{i'}(Fw^{(i)},x)Gw^{(k)}.\label{eq74}
\end{gather}
{~}\\

Our next goal is to prove the commutativity  of $\mc L$ and $\mc R$. First we need a preparatory result.

\begin{pp}\label{lb23}
Choose $W_i,W_j,W_k\in\Rep^\ssp(V)$, and arg-valued distinct $z_i,z_j\in S^1$ such that $z_i$ is anti-clockwise to $z_j$. Then for any $w^{(i)}\in W_i,w^{(j)}\in W_j,w^{(k)}\in W_k$, the following braid relation holds.
\begin{align}
\mc L_i(w^{(i)},z_i)\mc L_j(w^{(j)},z_j)w^{(k)}=(\ss_{j,i}\otimes\id_k)\mc L_j(w^{(j)},z_j)\mc L_i(w^{(i)},z_i)w^{(k)}
\end{align}
\end{pp}

It will be interesting to compare the above formula with proposition \ref{lb17}. 

\begin{proof}
Recall \eqref{eq62} and \eqref{eq63}. For any $r,s,t\in\mc E$, we choose basis $\Theta^t_{ir},\Theta^r_{jk},\Theta^t_{js},\Theta^s_{ik}$ of $\mc V{t\choose i~r},\mc V{r\choose j~k},\mc V{t\choose j~s},\mc V{s\choose i~k}$ respectively. Then for any $\alpha\in\Theta^t_{ir},\beta\in\Theta^r_{jk}$,
\begin{align*}
\bk{\mc L_i(w^{(i)},z_i)\mc L_j(w^{(j)},z_j)w^{(k)},\mc Y_\alpha\otimes\mc Y_\beta}=\mc Y_\alpha(w^{(i)},z_i)\mc Y_\beta(w^{(j)},z_j)w^{(k)}.
\end{align*}
Now for these basis intertwining operators we have the braid relation \eqref{eq65}, where the sign of $\pm$ is taken to be $+$. Therefore, by the discussion at the end of section \ref{lb22}, the transpose of $\ss\otimes\id$ on the vector spaces of intertwining operators is described by \eqref{eq64}. Thus we can compute that
\begin{align*}
&\bk{ (\ss_{j,i}\otimes\id_k)\mc L_j(w^{(j)},z_j)\mc L_i(w^{(i)},z_i)w^{(k)},\mc Y_\alpha\otimes\mc Y_\beta }\\
=&\sum_{s\in\mc E}\sum_{\alpha'\in\Theta^s_{ik},\beta'\in\Theta^t_{js}}(B_+)^{\beta'\alpha'}_{\alpha\beta}\bk{\mc L_j(w^{(j)},z_j)\mc L_i(w^{(i)},z_i)w^{(k)},\mc Y_{\beta'}\otimes\mc Y_{\alpha'} }\\
=&\sum_{s\in\mc E}\sum_{\alpha'\in\Theta^s_{ik},\beta'\in\Theta^t_{js}}(B_+)^{\beta'\alpha'}_{\alpha\beta}\mc Y_{\beta'}(w^{(j)},z_j)\mc Y_{\alpha'}(w^{(i)},z_i)w^{(k)}\\
=&\mc Y_\alpha(w^{(i)},z_i)\mc Y_\beta(w^{(j)},z_j)w^{(k)}.
\end{align*}
Hence the proposition is proved.
\end{proof}

\begin{thm}\label{lb26}
Choose $W_i,W_j,W_k\in\Rep^\ssp(V),w^{(i)}\in W_i,w^{(j)}\in W_j$, and arg-valued distinct $z_i,z_j\in S^1$ such that $z_i$ is anti-clockwise to $z_j$. Then the diagram
\begin{align}
\begin{CD}
W_k @> \quad\mc R_j(w^{(j)},z_j)\quad >> W_k\boxtimes W_j\\
@V \mc L_i(w^{(i)},z_i) VV @V \mc L_i(w^{(i)},z_i) VV\\
W_i\boxtimes W_k @> \quad\mc R_j(w^{(j)},z_j)\quad>> W_i\boxtimes W_k\boxtimes W_j
\end{CD}
\end{align}
commutes, in the sense that the following braid relation holds for any $w^{(k)}\in W_k$:
\begin{align*}
\mc L_i(w^{(i)},z_i)\mc R_j(w^{(j)},z_j)w^{(k)}=\mc R_j(w^{(j)},z_j)\mc L_i(w^{(i)},z_i)w^{(k)}.
\end{align*}
\end{thm}

\begin{proof}
By hexagon axioms, $\ss_{j,ik}=(\id_i\otimes\ss_{j,k})(\ss_{j,i}\otimes\id_k)$. We thus compute, using proposition \ref{lb23} and the functoriality of $\mc L_i$, that
\begin{align*}
&\mc R_j(w^{(j)},z_j)\mc L_i(w^{(i)},z_i)w^{(k)}=\ss_{j,ik}\mc L_j(w^{(j)},z_j)\mc L_i(w^{(i)},z_i)w^{(k)}\\
=&(\id_i\otimes\ss_{j,k})(\ss_{j,i}\otimes\id_k)\mc L_j(w^{(j)},z_j)\mc L_i(w^{(i)},z_i)w^{(k)}\\
=&(\id_i\otimes\ss_{j,k})\mc L_i(w^{(i)},z_i)\mc L_j(w^{(j)},z_j)w^{(k)}=\mc L_i(w^{(i)},z_i)\ss_{j,k}\mc L_j(w^{(j)},z_j)w^{(k)}\\
=&\mc L_i(w^{(i)},z_i)\mc R_j(w^{(j)},z_j)w^{(k)}.
\end{align*}
\end{proof}

We close this section with the fusion relation of $\mc L$.
\begin{pp}\label{lb25}
Choose $W_i,W_j,W_k\in\Rep^\ssp(V)$, and arg-valued distinct $z_i,z_j\in S^1$ such that $0<|z_i-z_j|<|z_j|<|z_i|$ and $\arg z_i=\arg z_j=\arg(z_i-z_j)$. Then for any $w^{(i)}\in W_i,w^{(j)}\in W_j,w^{(k)}\in W_k$, the following fusion relation holds.
\begin{align}
\mc L_i(w^{(i)},z_i)\mc L_j(w^{(j)},z_j)w^{(k)}=\mc L_{ij}(\mc L_i(w^{(i)},z_i-z_j)w^{(j)},z_j)w^{(k)}.
\end{align}
\end{pp}
\begin{proof}
For each $r,s,t\in\mc E$ we choose  basis $\Theta^t_{sk},\Theta^s_{ij},\Theta^t_{ir},\Theta^r_{jk}$ of $\mc V{t\choose s~k},\mc V{s\choose i~j},\mc V{t\choose i~r},\mc V{r\choose j~k}$ respectively, and assume that the fusion relation \eqref{eq68} holds. Choose any $r,t\in\mc E,\alpha\in\Theta^t_{ir},\beta\in\Theta^r_{jk}$, then
\begin{align*}
\bk{\mc L_i(w^{(i)},z_i)\mc L_j(w^{(j)},z_j)w^{(k)},\mc Y_\alpha\otimes\mc Y_\beta}=\mc Y_\alpha(w^{(i)},z_i)\mc Y_\beta(w^{(j)},z_j)w^{(k)}.
\end{align*}
On the other hand,  the expression 
$\bk{\mc L_{ij}(\mc L_i(w^{(i)},z_i-z_j)w^{(j)},z_j)w^{(k)},\mc Y_\alpha\otimes\mc Y_\beta}$, when written more precisely,  should be $\bk{\mc L_{ij}(\mc L_i(w^{(i)},z_i-z_j)w^{(j)},z_j)w^{(k)},A^\tr(\mc Y_\alpha\otimes\mc Y_\beta)}$, where $A$ is the associativity isomorphism. By \eqref{eq69}, we have
\begin{align*}
&\bk{\mc L_{ij}(\mc L_i(w^{(i)},z_i-z_j)w^{(j)},z_j)w^{(k)},A^\tr(\mc Y_\alpha\otimes\mc Y_\beta)}\\
=&\sum_{s\in\mc E} \sum_{\alpha'\in\Theta^s_{ij},\beta'\in\Theta^t_{sk}}F^{\beta'\alpha'}_{\alpha\beta}\bk{\mc L_{ij}(\mc L_i(w^{(i)},z_i-z_j)w^{(j)},z_j)w^{(k)},\mc Y_{\beta'}\otimes\mc Y_{\alpha'}}\\
=&\sum_{s\in\mc E} \sum_{\alpha'\in\Theta^s_{ij},\beta'\in\Theta^t_{sk}}F^{\beta'\alpha'}_{\alpha\beta}\mc Y_{\beta'}(\mc Y_{\alpha'}(w^{(i)},z_i-z_j)w^{(j)},z_j)w^{(k)}\\
=&\mc Y_\alpha(w^{(i)},z_i)\mc Y_\beta(w^{(j)},z_j)w^{(k)}.
\end{align*}
Thus our proof is complete.
\end{proof}

\subsection{Unitarity}

Beginning with this section, we assume the following unitary condition (see \cite{DL14,CKLW18} for the definition of unitary VOAs):
\begin{cond}\label{CondF}
$V$ is a unitary VOA with inner product $\bk{\cdot|\cdot}$ and PCT operator $\varTheta$. 
\end{cond}
Recall that a $V$-module $W_i$ is called unitary, if $W_i$ is equipped with an inner product $\bk{\cdot|\cdot}$, such that for any $v\in V$, the vertex operator $Y_i(v,x)$ (where $x$ is a formal variable) on $W_i$ satisfies
\begin{align}
Y_i(v,x)^\dagger=Y_i(e^{xL_1}(-x^{-2})^{L_0}\varTheta v,x^{-1}).
\end{align}
Here $\dagger$ means the formal adjoint operation. So the above relation is equivalent to saying that
\begin{align*}
\bk{Y_i(v,x)w^{(i)}_1|w^{(i)}_2 }=\bk{w^{(i)}_1|Y_i(e^{xL_1}(-x^{-2})^{L_0}\varTheta v,x^{-1})w^{(i)}_2 }
\end{align*}
for any $w^{(i)}_1,w^{(i)}_2\in W_i$.

\begin{cond}\label{CondG}
If $W_i,W_j\in\Rep^\ssp(V)$ are unitarizable, then $W_i\boxtimes W_j$ is also unitarizable.
\end{cond}

By this condition, if $W_i,W_j$ are unitarizable, then for any $t\in\mc E$, $\mc V{t\choose i~j}$  is trivial unless $W_t$ is also unitary.

We let $\Rep^\uni(V)$ be the category of unitary semisimple $V$-modules. Whenever $W_t\in\mc E$ is unitarizable, we fix a unitary structure on $W_t$. If $t=0$, then the unitary structure on $W_0$ is chosen to be the one on $V$. Then one can define a ribbon categorical structure on $\Rep^\uni(V)$ in a similar way as for $\Rep^\ssp(V)$. $\Rep^\uni(V)$ is clearly equivalent to a ribbon tensor subcategory of $\Rep^\ssp(V)$. In the rest of this paper, we will always focus on modules in $\Rep^\uni(V)$ instead of in $\Rep^\ssp(V)$.

Note that for any $W_i\in\Rep^\uni(V)$, the inner product $\bk{\cdot|\cdot}$ induces a natural anti-unitary map $C_i:W_i\rightarrow W_{\ovl i}$, such that for any $v\in V$,
\begin{align}
Y_{\ovl i}(v,x)=C_i Y_i(\varTheta v,x)C_i^{-1}.
\end{align}
(Cf. \cite{Gui17a} equation (1.19).) We write $\ovl{w^{(i)}}=C_iw^{(i)}$ for any $w^{(i)}\in W_i$. One can also show that $C_0\varTheta:V=W_0\rightarrow W_{\ovl 0}$ is a unitary $V$-module isomorphism. Therefore $V$ is self-dual. We identify $V$ with $W_{\ovl 0}$ using $C_0\varTheta$. Under this identification, $C_0$ is the same as $\varTheta^{-1}=\varTheta$. So by our notation, $\ovl v=\varTheta v$ for any $v\in V$.\\

Let us now equip the ribbon category $\Rep^\uni(V)$ with a unitary structure. In the following, we assume that all modules are unitary. If $\mc Y_\alpha\in\mc V{k\choose i~j}$, we define the \textbf{adjoint intertwining operator} $\mc Y_\alpha^\dagger\equiv\mc Y_{\alpha^*}$ to be a type $j\choose\ovl i~k$ intertwining operator satisfying that for any $w^{(i)}\in W_i$,
\begin{align}
\mc Y_\alpha^\dagger(\ovl{w^{(i)}},x)\equiv\mc Y_{\alpha^*}(\ovl{w^{(i)}},x)=\mc Y_\alpha(e^{xL_1}(e^{-i\pi}x^{-2})^{L_0}w^{(i)},x^{-1})^\dagger.
\end{align}

Recall the creation operator $\mc Y_{\kappa(i)}$ defined in section \ref{lb22}. It is clear that $\mc Y_{\kappa(i)}=\mc L_i|_0=\mc R_i|_0$. The adjoint intertwining operator $\mc Y_{\kappa(i)^*}$ of $\mc Y_{\kappa(i)}$ is called the annihilation operator of $W_i$. Now for any $W_i,W_j\in\Rep^\uni(V),t\in\mc E$, we choose a basis $\Theta^t_{ij}$ of $\mc V{t\choose i~j}$. Choose any arg-valued $z_1,z_2\in\mathbb C^\times$ satisfying $0<|z_2-z_1|<|z_1|<|z_2|$ and $\arg z_1=\arg z_2=\arg(z_2-z_1)$. Then there exists a unique complex matrix $\{\Lambda^{\alpha\beta} \}$ independent of the choice of $z_i,z_2$, such that for any $w^{(i)}_1,w^{(i)}_2\in W_i$, the following fusion relation holds.
\begin{align}
Y_j(\mc Y_{\kappa(i)^*}(\ovl{w_2^{(i)}},z_2-z_1)w^{(i)}_1,z_1)=\sum_{t\in\mc E}\sum_{\alpha,\beta\in\Theta^t_{ij}}\Lambda^{\alpha\beta}\mc Y_{\beta^*}(\ovl{w^{(i)}_2},z_2)\mc Y_\alpha(w^{(i)}_1,z_1)\label{eq67}
\end{align}
Recall that by our notation in the last section, $\wch{\mc Y}^\alpha\in\mc V{t\choose i~j}^*$ is the dual element of $\mc Y_\alpha$. For each $t\in\mc E$, we define a sesquilinear form $\Lambda$ on $\mc V{t\choose i~j}^*$, antilinear on the second variable, such that for any $\alpha,\beta\in\Theta^t_{ij}$,
\begin{align}
\Lambda(\wch{\mc Y}^\alpha|\wch{\mc Y}^\beta)=\Lambda^{\alpha\beta}.
\end{align}
It is easy to check that this definition of $\Lambda$ is independent of the choice of basis.

\begin{cond}\label{CondH}
For each $W_i,W_j\in\Rep^\uni(V),W_t\in\mc E$, the sesquilinear form $\Lambda$ on $\mc V{t\choose i~j}^*$ is positive.
\end{cond}
By the rigidity of $\Rep^\uni(V)$ (condition \ref{lb24}), $\Lambda$ is also non-degenerate (cf. \cite{Gui17b} theorem 6.7 step 3). Therefore, $\Lambda$ is a (non-degenerate) inner product on each $\mc V{t\choose i~j}^*$, which can be extended naturally to an inner product on $W_i\boxtimes W_j=\bigoplus_{t\in\mc E}\mc V{t\choose i~j}^*\otimes W_t$, under which $W_i\boxtimes W_j$ becomes a unitary $V$-module. One can show  (cf. \cite{Gui17b} chapter 7) that under these inner products, all the structural maps (associativity isomorphisms, braid operators, etc.) are unitary. We thus identify the unitary $V$-modules $W_i\boxtimes(W_j\boxtimes W_k)$ and $(W_i\boxtimes W_j)\boxtimes W_k$ as $W_i\boxtimes W_j\boxtimes W_k$, and identify $V\boxtimes W_i,W_i\boxtimes V$ as $W_i$. Hence $\Rep^\uni(V)$ is a braided $C^*$-tensor category. Moreover, if $\mc E$ is finite, then $\Rep^\uni(V)$ is a unitary ribbon fusion category.\\

Recall that for each $W_i\in\Rep^\uni(V)$, $\mc L_i$ acts on each $W_j\in\Rep^\uni(V)$  as a type $ij\choose i~j$ intertwining operator. We let $\mc L_i^\dagger$ act on each $W_i\boxtimes W_j$ as the adjoint intertwining operator of $\mc L_i|_j$, which is of type $j\choose \ovl i~ij$. In other words, we let $\mc L_i^\dagger|_{ij}=(\mc L_i|_j)^\dagger$. In the remaining part of this section, we shall show that $\mc L^\dagger$ and $\mc R$ commute.

For any $W_i\in\Rep^\uni(V)$, we let $\ev_{\ovl i, i}\in\Hom_V(W_{\ovl i}\boxtimes W_i,V)$ be the morphism defined by the intertwining operator $\mc Y_{\kappa(i)^*}$. Then for any  $w^{(i)}_1,w^{(i)}_2\in W_i$,
\begin{align}
\ev_{\ovl i,i}\mc L_{\ovl i}(\ovl{w^{(i)}_1},x)w^{(i)}_2=\mc Y_{\kappa(i)^*}(\ovl{w^{(i)}_1},x)w^{(i)}_2=\mc L_i^\dagger(\ovl{w^{(i)}_1},x)w^{(i)}_2.\label{eq71}
\end{align}
More generally, we have:

\begin{pp}\label{lb28}
For any $W_i,W_j\in\Rep^\uni(V),w^{(i)}\in W_i,w^{(ij)}\in W_{ij}=W_i\boxtimes W_j$,
\begin{align}
\mc L_i^\dagger(\ovl{w^{(i)}},z)w^{(ij)}=(\ev_{\ovl i,i}\otimes\id_j)\mc L_{\ovl i}(\ovl{w^{(i)}},z)w^{(ij)}.\label{eq73}
\end{align}
\end{pp}

\begin{proof}
Choose any $w^{(i)}_1,w^{(i)}_2\in W_i,w^{(j)}\in W_j$, and choose arg-valued $z_1,z_2$ such that $0<|z_2-z_1|<|z_1|<|z_2|$ and $\arg z_1=\arg z_2=\arg (z_2-z_1)$. We first rewrite \eqref{eq67} using $\mc L$. Note that $\mc L_0|_j$ is just $Y_j$, and $\mc Y_{\kappa(i)^*}=\mc L_i^\dagger|_i$. Therefore the left hand side of \eqref{eq67} is
\begin{align*}
\mc L_0(\mc L_i^\dagger(\ovl{w_2^{(i)}},z_2-z_1)w^{(i)}_1,z_1)w^{(j)}
\end{align*}
when acting on $w^{(j)}\in W_j$. On the other hand, if we let $\Theta^t_{ij}$ be an orthonormal basis of $\mc V{t\choose i~j}$, and write $\mc L_i$ as \eqref{eq66}, then it is easy to see that the right hand side of \eqref{eq67} equals
\begin{align*}
\mc L_i^\dagger(\ovl{w^{(i)}_2},z_2)\mc L_i(w^{(i)}_1,z_1)w^{(j)}
\end{align*}
when acting on $w^{(j)}$. Therefore \eqref{eq67} is equivalent to
\begin{align}
\mc L_0(\mc L_i^\dagger(\ovl{w_2^{(i)}},z_2-z_1)w^{(i)}_1,z_1)w^{(j)}=\mc L_i^\dagger(\ovl{w^{(i)}_2},z_2)\mc L_i(w^{(i)}_1,z_1)w^{(j)}.\label{eq72}
\end{align}

Using proposition \ref{lb25}, the functoriality of $\mc L$ (equation \eqref{eq70}), and equations \eqref{eq71} and \eqref{eq72}, we compute
\begin{align*}
&(\ev_{\ovl i,i}\otimes\id_j)\mc L_{\ovl i}(\ovl{w^{(i)}_2},z_2)\mc L_i(w^{(i)}_1,z_1)w^{(j)}=(\ev_{\ovl i,i}\otimes\id_j)\mc L_{\ovl ii}(\mc L_{\ovl i}(\ovl{w^{(i)}_2},z_2-z_1)w^{(i)}_1,z_1)w^{(j)}\\
=&\mc L_0(\ev_{\ovl i,i}\mc L_{\ovl i}(\ovl{w^{(i)}_2},z_2-z_1)w^{(i)}_1,z_1)w^{(j)}=\mc L_0(\mc L_i^\dagger(\ovl{w^{(i)}_2},z_2-z_1)w^{(i)}_1,z_1 )w^{(j)}\\
=&\mc L_i^\dagger(\ovl{w^{(i)}_2},z_2)\mc L_i(w^{(i)}_1,z_1)w^{(j)}.
\end{align*}
Thus we've proved \eqref{eq73} when both sides ``act on" $\mc L_i(w^{(i)}_1,z_1)w^{(j)}$. Write $\mc L_i(w^{(i)}_1,z_1)=\sum_{n\in\mathbb R}\mc L_i(w^{(i)}_1)_nz^{-n-1}$. Then by \cite{Hua95} lemma 14.5 (see also \cite{Gui17a} proposition A.1), \eqref{eq73} holds when acting on $\mc L_i(w^{(i)}_1)_nw^{(j)}$ for any $w^{(i)}_1,\in W_i,w^{(j)}\in W_j,n\in\mathbb R$. By \cite{Hua95} lemma 14.9 (see also \cite{Gui17a} corollary A.4), vectors of the form $\mc L_i(w^{(i)}_1)_nw^{(j)}$ span the vector space $W_{ij}$. Therefore \eqref{eq73} is proved.
\end{proof}

We are now ready to prove the adjoint version of theorem \ref{lb26}.

\begin{thm}\label{lb27}
Choose $W_i,W_j,W_k\in\Rep^\ssp(V),w^{(i)}\in W_i,w^{(j)}\in W_j$, and arg-valued distinct $z_i,z_j\in S^1$ such that $z_i$ is anti-clockwise to $z_j$. Then the following diagram commutes in the sense of braiding of intertwining operators.
\begin{align}
\begin{CD}
W_k @> \quad\mc R_j(w^{(j)},z_j)\quad >> W_k\boxtimes W_j\\
@A \mc L_i^\dagger(\ovl{w^{(i)}},z_i) AA @A \mc L_i^\dagger(\ovl{w^{(i)}},z_i) AA\\
W_i\boxtimes W_k @> \quad\mc R_j(w^{(j)},z_j)\quad>> W_i\boxtimes W_k\boxtimes W_j
\end{CD}.\label{eq75}
\end{align}
\end{thm}

\begin{proof}
Consider the diagrams
\begin{align}
\begin{CD}
W_i\boxtimes W_k @> \quad\mc R_j(w^{(j)},z_j)\quad >> W_i\boxtimes W_k\boxtimes W_j\\
@V \mc L_{\ovl i}(\ovl{w^{(i)}},z_i) VV @V \mc L_{\ovl i}(\ovl{w^{(i)}},z_i) VV\\
W_{\ovl i}\boxtimes W_i\boxtimes W_k @> \quad\mc R_j(w^{(j)},z_j)\quad>> W_{\ovl i}\boxtimes W_i\boxtimes W_k\boxtimes W_j\\
@V \ev_{\ovl i,i}\otimes\id_k VV @V \ev_{\ovl i,i}\otimes\id_k\otimes\id_j VV\\
 W_k @> \quad\mc R_j(w^{(j)},z_j)\quad >> W_k\boxtimes W_j.
\end{CD}
\end{align}
The first small diagram commutes due to theorem \ref{lb26}, the second one commutes due to the functoriality of $\mc R$ (equation \eqref{eq74}). Therefore the large diagram commutes, which is equivalent to the commutativity of diagram \eqref{eq75} by proposition \ref{lb28}.
\end{proof}

\subsection{Smeared intertwining operators}

We recall the definition and some of the basic properties of energy bounded intertwining operators. See \cite{Gui17a} chapter 3 for more details. We first fix some notations. If $A$ is an unbounded operator on a Hilbert space $\mc H$, we let $\scr D(A)$ be the domain of $A$. If $A$ is densely-defined and preclosed, we let $\ovl A$ denote its closure, and $A^*=\ovl A^*$ its adjoint. If $A$ and $B$ are densely-defined with common domain $\scr D=\scr D (A)=\scr D(B)$, we say that $B$ is the (clearly unique) \textbf{formal adjoint} of $A$, and write $B=A^\dagger$, if for any $\xi,\eta\in\Dom$,
\begin{align}
\bk{A\xi|\eta}=\bk{\xi|B\eta}.
\end{align}

If $A,B$ are preclosed operators on $\mc H$, we say that $A$ \textbf{commutes strongly} with $B$, if the von Neumann algebra generated by $\ovl A,\ovl A^*$ commutes with the one generated by $\ovl B,\ovl B^*$. (See \cite{Gui17a} section B.1 for more details.) Therefore, by our definition, two bounded operators commute strongly if and only if they commute adjointly.

\begin{df}\label{lb37}
Let $\mc P,\mc Q, \mc R,\mc S$ be Hilbert spaces, and  $A:\mc P\rightarrow\mc R,B:\mc Q\rightarrow\mc S,C:\mc P\rightarrow\mc Q,D:\mc R\rightarrow\mc S$ be unbounded preclosed operators. By saying that the diagram of preclosed operators
\begin{align}
\begin{CD}
\mc P @>C>> \mc Q\\
@V A VV @V B VV\\
\mc R @>D>> \mc S
\end{CD}
\end{align}
\textbf{commutes strongly}, we mean the following: Let $\mc H=\mc P\oplus\mc Q\oplus\mc R\oplus\mc S$. Define unbounded preclosed operators $R,S$ on $\mc H$ with domains $\Dom(R)=\Dom(A)\oplus\Dom(B)\oplus\mc R\oplus \mc S$, $\Dom(S)=\Dom(C)\oplus\mc Q\oplus\Dom(D)\oplus \mc S$, such that
\begin{gather*}
R(\xi\oplus\eta\oplus\chi\oplus\varsigma)=0\oplus 0\oplus A\xi\oplus B\eta\qquad(\forall \xi\in\Dom(A),\eta\in\Dom(B),\chi\in\mc R,\varsigma\in \mc S),\\
S(\xi\oplus\eta\oplus\chi\oplus\varsigma)=0\oplus C\xi\oplus 0\oplus D\chi   \qquad(\forall \xi\in\Dom(C),\eta\in\mc Q,\chi\in \Dom(D),\varsigma\in\mc S).
\end{gather*}
(Such construction is called the \textbf{extension} from $A,B$ to $R$, and from $C,D$ to $S$.) Then $R$ and $S$ commute strongly.

\end{df}

Now we return to the unitary VOA $V$ and its unitary modules. For any $W_i\in\Rep^\uni(V)$, we let $\mc H_i$ be the Hilbert space completion of $W_i$. Then $L_0$ is a preclosed operator on $\mc H_i$ with dense domain $W_i$. Its closure $\ovl{L_0}$ is clearly self-adjoint. We set $\mc H_i^\infty=\bigcap_{n\in\mathbb Z_{\geq0}}\Dom((1+\ovl{L_0})^n)$. Then as $W_i\subset \mc H_i^\infty$, $\mc H_i^\infty$ is a dense subspace of $\mc H_i$. Vectors in $\mc H_i^\infty$ are called \textbf{smooth}.

Let $W_i,W_j,W_k\in\Rep^\uni(V)$. For any $\mc Y_\alpha\in\mc V{k\choose i~j}$ and any homogeneous vector $w^{(i)}\in W_i$, we write $\mc Y_\alpha(w^{(i)},x)=\sum_{n\in\mathbb R}\mc Y_\alpha(w^{(i)})_nx^{-n-1}$, where each $\mc Y_\alpha(w^{(i)})_n$ is a linear map from $W_j$ to $W_k$. For any $a\geq 0$, we say that $\mc Y_\alpha(w^{(i)},x)$ satisfies \textbf{$a$-th order energy bounds}, if there exist $M,b\geq0$, such that for any $n\in\mathbb R,w^{(j)}\in W_j$, 
\begin{align}
\lVert \mc Y_\alpha(w^{(i)})_nw^{(j)} \lVert\leq M(1+|n|)^b\lVert (1+L_0)^aw^{(j)}\lVert.
\end{align}
By \cite{Gui17a} proposition 3.4, if $w^{(i)}$ is quasi-primary and $\mc Y_\alpha(w^{(i)},x)$  satisfies $a$-th order energy bounds, then so does $\mc Y_{\alpha^*}(\ovl{w^{(i)}},x)$.

We say that $\mc Y_\alpha(w^{(i)},x)$ is \textbf{energy-bounded} if it satisfies $a$-th order energy bounds for some $a\geq0$.  We say that $V$ is energy-bounded if $Y(v,x)$ is energy-bounded for any homogeneous $v\in V$. We say that a unitary $V$-module $W_i$ is energy-bounded if $Y_i(v,x)$ is energy-bounded for any homogeneous $v\in V$.

We now define smeared intertwining operators for energy bounded intertwining operators (cf. \cite{Gui17a} section 3.2). Recall the discussion of arg-valued intervals in section \ref{lb29}. For any $\wtd I=(I,\arg_I)\in\Jtd$ and $f\in C^\infty_c(I)$ , we call $\wtd f=(f,\arg_I)$ a (smooth) \textbf{arg-valued function} on $S^1$ with support inside $\wtd I$, and let $C^\infty_c(\wtd I)$ be the set of all such $\wtd f$. We set the complex conjugate of $\wtd f$ to be $\ovl{\wtd f}=(\ovl f,\arg I)$. If $\wtd I\subset\wtd J\in\Jtd$, then $C^\infty_c(\wtd I)$ is naturally a subspace of $C^\infty_c(\wtd J)$ by identifying each $(f,\arg I)\in C^\infty_c(\wtd I)$ with $(f,\arg J)$.

Now if $\mc Y_\alpha\in\mc V{k\choose i~j}$, $w^{(i)}\in W_i$ is homogeneous, $\mc Y_\alpha(w^{(i)},x)$ is energy-bounded, $\wtd I=(I,\arg_I)\in\Jtd$, and $\wtd f=(f,\arg_I)\in C^\infty_c(\wtd I)$ , we  define the smeared intertwining operator $\mc Y_\alpha(w^{(i)},\wtd f)$ to be a bilinear form on $W_j\otimes W_{\ovl k}$ satisfying
\begin{align}
\mc Y_\alpha(w^{(i)},\wtd f)=\int_{\arg_I(I)}\mc Y_\alpha(w^{(i)},e^{i\theta})f(e^{i\theta})\cdot\frac{e^{i\theta}}{2\pi}d\theta.\label{eq90}
\end{align}
Then $\mc Y_\alpha(w^{(i)},\wtd f)$ maps $W_j$ into $\mc H_k^\infty$. Regarding $\mc Y_\alpha(w^{(i)},\wtd f)$ as an unbounded operator from $\mc H_j$ to $\mc H_k$ with domain $W_j$, $\mc Y_\alpha(w^{(i)},\wtd f)$ is preclosed, the closure of which contains $\mc H_j^\infty$. Moreover, we have
\begin{gather*}
\ovl{\mc Y_\alpha(w^{(i)},\wtd f)}\mc H_j^\infty\subset\mc H_k^\infty,\qquad \ovl{\mc Y_\alpha(w^{(i)},\wtd f)}^*\mc H_k^\infty\subset\mc H_j^\infty.
\end{gather*}
In the following, \emph{we will always denote by $\mc Y_\alpha(w^{(i)},\wtd f)$ the restriction of the closed operator $\ovl{\mc Y_\alpha(w^{(i)},\wtd f)}$ to the core $\mc H_j^\infty$.} Then the formal adjoint $\mc Y_\alpha(w^{(i)},\wtd f)^\dagger$ exists, which is the restriction of $\mc Y_\alpha(w^{(i)},\wtd f)^*$ to $\mc H_k^\infty$.

We now give  formulae for the rotation covariance of smeared intertwining operators.  Recall that we have an action of $\scr G$ on $\Jtd$ defined in section \ref{lb29}. For any $t\in\mathbb R$ and $\wtd I\in\Jtd$, write $\wtd\exp(it L_0)\wtd I=\wtd J=(J,\arg_J)$. We define a linear map $\fk r(t):C^\infty_c(\wtd I)\rightarrow C^\infty_c(\wtd J)$, such that for any $\wtd f=(f,\arg I)$, $\fk r(t)\wtd f=(\fk r(t) f,\arg_J)$ satisfies $\fk r(t)f(e^{i\theta})=f(e^{i(\theta-t)})$ ($\forall \theta\in\mathbb R$). Then using the proof of \cite{Gui17a} proposition 3.15, one can easily show that
\begin{align}
e^{it\ovl{L_0}}\mc Y_\alpha(w^{(i)},\wtd f)e^{-it\ovl{L_0}}=\mc Y_\alpha(w^{(i)},e^{i(\Delta_{w^{(i)}}-1)t}\fk r(t) \wtd f)\label{eq78}
\end{align}
for any homogeneous $w^{(i)}\in W_i$ with conformal weight $\Delta_{w^{(i)}}$. Set $f'(e^{i\theta})=\frac d{d\theta}f(e^{i\theta})$ and $\wtd f'=(f',\arg I)$. Then we have another version of rotation covariance
\begin{align}
[\ovl{L_0},\mc Y_\alpha(w^{(i)},\wtd f)]=\mc Y_\alpha(w^{(i)},(\Delta_{w^{(i)}}-1)\wtd f+i\wtd f' ),\label{eq79}
\end{align}
where both sides of the equation act on $\mc H_j^\infty$. (See also \cite{Gui17a} proposition 3.15.)

Next we relate $\mc Y_\alpha(w^{(i)},\wtd f)^\dagger$ with the smeared intertwining operator of $\mc Y_{\alpha^*}\equiv\mc Y_\alpha^\dagger$. It was proved in \cite{Gui17a} proposition 3.4 that if $\mc Y_\alpha$ satisfies $a$-th order energy bounds, then so does $\mc Y_{\alpha^*}$.  Now, for any $a\in\mathbb R,\wtd f\in C^\infty_c(\wtd I)$, we set $e_a\wtd f=(e_a f,\arg I)\in C^\infty_c(\wtd I)$, where $e_a f$ is the smooth function on $S^1$ defined by
\begin{align*}
e_a f(e^{i\theta}) = \left\{ \begin{array}{ll}
e^{ia\theta}f(e^{i\theta}) & \textrm{if $\theta\in\arg_I(I)$}\\
0 & \textrm{if $e^{i\theta}\notin I$}
\end{array} \right. .
\end{align*}
Then for any homogeneous $w^{(i)}\in W_i$,
\begin{align}
\mc Y_\alpha(w^{(i)},\wtd f)^\dagger=\sum_{m\in\mathbb Z_{\geq0}}\frac{e^{-i\pi\Delta_{w^{(i)}}}}{m!}\mc Y_{\alpha^*}(\ovl{L_1^mw^{(i)}},\ovl{e_{m+2-2\Delta_{w^{(i)}}}\wtd f})\label{eq76}
\end{align}
(cf. \cite{Gui17a} proposition 3.9), recalling that $\Delta_{w^{(i)}}$ is the conformal weight of $w^{(i)}$.

We also have braiding of smeared intertwining operators (cf. \cite{Gui17a} corollary 3.13):

\begin{pp}\label{lb30}
Choose disjoint $\wtd I,\wtd J\in\Jtd$, and  $z_i\in I,z_j\in J$ with arguments $\arg_I(z_i),\arg_J(z_j)$ respectively. Suppose $W_i,W_j,W_k,W_r,W_s,W_t$ are unitary $V$-modules, $\mc Y_\alpha\in\mc V{t\choose i~s},\mc Y_\beta\in\mc V{s\choose j~k},\mc Y_{\alpha'}\in\mc V{r\choose i~k},\mc Y_{\beta'}\in\mc V{t\choose j~r}$, and for any $w^{(i)}\in W_i,w^{(j)}\in W_j$, the following braid relation holds:
\begin{align*}
\mc Y_\alpha(w^{(i)},z_i)\mc Y_\beta(w^{(j)},z_j)=\mc Y_{\beta'}(w^{(j)},z_j)\mc Y_{\alpha'}(w^{(i)},z_i).
\end{align*}
Then if $w^{(i)},w^{(j)}$ are homogeneous, $\mc Y_\alpha(w^{(i)},x),\mc Y_\beta(w^{(j)},x),\mc Y_{\alpha'}(w^{(i)},x),\mc Y_{\beta'}(w^{(j)},x)$ are energy-bounded, and $\wtd f\in C^\infty_c(\wtd I),\wtd g\in C^\infty_c(\wtd J)$, the following equation holds when acting on $\mc H^\infty_k$:
\begin{align*}
\mc Y_\alpha(w^{(i)},\wtd f)\mc Y_\beta(w^{(j)},\wtd g)=\mc Y_{\beta'}(w^{(j)},\wtd g)\mc Y_{\alpha'}(w^{(i)},\wtd f).
\end{align*}
\end{pp}

This proposition, together with relation \eqref{eq76}, implies immediately the following main result of this section. Note that by our notation, if $W_i,W_j\in\Rep^\uni(V)$, then $\mc H_{ij}$ is the Hilbert space completion of $W_{ij}=W_i\boxtimes W_j$, and $\mc H^\infty_{ij}$ is the subspace of smooth vectors. Similarly, if we also have $W_k\in\Rep^\uni(V)$, then $\mc H_{ikj}$ is the Hilbert space completion of $W_{ikj}=W_i\boxtimes W_k\boxtimes W_j$, and $\mc H^\infty_{ikj}$ is its smooth subspace.

\begin{thm}\label{lb38}
Choose $W_i,W_j,W_k\in\Rep^\uni(V),w^{(i)}\in W_i,w^{(j)}\in W_j$, and disjoint $\wtd I,\wtd J\in \Jtd$ such that $\wtd I$ is anticlockwise to $\wtd J$. Assume that $w^{(i)},w^{(j)}$ are homogeneous,  and $\mc L_i|_k(w^{(i)},x),\mc L_i|_{kj}(w^{(i)},x),\mc R_j|_k(w^{(j)},x),\mc R_j|_{ik}(w^{(j)},x)$ are energy-bounded. Then the diagram
\begin{align}
\begin{CD}
\mc H^\infty_k @> \quad\mc R_j(w^{(j)},\wtd g)\quad >> \mc H^\infty_{kj}\\
@V \mc L_i(w^{(i)},\wtd f) VV @V \mc L_i(w^{(i)},\wtd f) VV\\
\mc H^\infty_{ik} @> \quad\mc R_j(w^{(j)},\wtd g)\quad>> \mc H^\infty_{ikj}
\end{CD}
\end{align}
commutes adjointly, in the sense that both this diagram and the following diagram commute:
\begin{align}
\begin{CD}
\mc H^\infty_k @> \quad\mc R_j(w^{(j)},\wtd g)\quad >> \mc H^\infty_{kj}\\
@A \mc L_i(w^{(i)},\wtd f)^\dagger AA @A \mc L_i(w^{(i)},\wtd f)^\dagger AA\\
\mc H^\infty_{ik} @> \quad\mc R_j(w^{(j)},\wtd g)\quad>> \mc H^\infty_{ikj}
\end{CD}.
\end{align}
\end{thm}

\begin{proof}
The first diagram commutes due to theorem \ref{lb26} and proposition \ref{lb30}. The second one commutes due to theorem \ref{lb27}, proposition \ref{lb30}, and relation \eqref{eq76}.
\end{proof}

\subsection{Conformal nets associated to VOAs}\label{lb55}

In  this section, we discuss some relations between unitary VOAs and conformal nets as well as their modules. Let $W_i$ be a unitary $V$-module. Then for any $\wtd f=(f,\arg I)$, the smeared vertex operator $Y_i(v,\wtd f)$ is independent of the choice of arguments as $Y_i(v,z)$ is a meromorphic field. We thus write $Y_i(v,\wtd f)$ as $Y_i(v,f)$. In particular, $Y_0=Y$, and $Y(v,\wtd f)$ is written as $Y(v,f)$.

\begin{cond}\label{CondI}
The unitary VOA $V$ is energy-bounded. Moreover, $V$ is \textbf{strongly local}, which means that for any disjoint $I,J\in\mc J$, homogeneous $u,v\in V$, and $f\in C^\infty_c(I),g\in C^\infty_c(J)$, the closed operators $\ovl{Y(u,f)}$ and $\ovl{Y(v,g)}$ commute strongly.
\end{cond}

Then by \cite{CKLW18}, there exists a (unique) conformal net $\mc A_V$ acting on $\mc H_0$ (the Hilbert space completion of $V=W_0$), such that for any $I\in \mc J$, $\mc A_V(I)$ is the von Neumann algebra generated by all $\ovl{Y(v,f)}$ and $\ovl{Y(v,f)}^*$ (where $v\in V$ is homogeneous, and $f\in C^\infty_c(I)$). Moreover, the projective representation  of $\Diffp(S^1)$ (and hence of $\scr G$) is integrated from the positive energy representation of the Virasoro algebra on $V$. We call $\mc A_V$ the \textbf{conformal net associated to $V$}.

A unitary $V$-module $W_i$ is called \textbf{strongly-integrable} (cf. \cite{CWX}), if $W_i$ is energy-bounded, and there is a (unique) $\mc A_V$-module $(\mc H_i,\pi_i)\in\Rep(\mc A_V)$, such that for any $I\in\mc J,f\in C^\infty_c(I)$, and any homogeneous $v\in V$, we have $\pi_i(\ovl{Y(v,f)})=\ovl{Y_i(v,f)}$. 

We now show that the action of $\scr G_{\mc A_V}$ on $\mc H_i$ is integrated from the action of the Virasoro algebra on $W_i$. For any $n\in\mathbb Z$ we set $e_n\in C^\infty(S^1)$ to be $e_n(e^{i\theta})=e^{in\theta}$. For any $f\in C^\infty(S^1)$, write $f=\sum_{n\in\mathbb Z}a_ne_n$ where $\{a_n \}$ are the Fourier series of $f$, and set $T(f)=\sum_n a_nL_{n-1}\in \Vect^\mathbb C(S^1)$. Then $T(f)$ is self-adjoint (namely, $iT(f)\in\Vect(S^1)$) when $e_{-1}f$ is real. Recall that $U$ and $U_i$ are respectively the representations of $\GA$ on $\mc H_0$ and $\mc H_i$.

\begin{pp}\label{lb34}
Let $W_i\in\Rep^\uni(V)$ be strongly-integrable, and let $\mc H_i$ be the corresponding $\mc A_V$-module. For any $g\in\scr G_{\mc A_V}$, if there exist $f\in e_1\cdot C^\infty_c(S^1,\mathbb R)$ and $\lambda\in\mbb C$ with $|\lambda|=1$ satisfying
\begin{align}
g=(\wtd\exp(iT(f)),\lambda e^{i\ovl {Y(\nu,f)}})\in \scr G_{\mc A_V}\subset\scr G\times\mc U(\mc H_0),\label{eq91}
\end{align}
then $U_i(g)=\lambda e^{i\ovl{Y_i(\nu,f)}}$.
\end{pp}

\begin{proof}

Our strategy is to define a unitary representation $U_i':\scr G_{\mc A_V}\curvearrowright\mc H_i$ satisfying the claim of this proposition, and to show that $U_i'$ equals the standard one $U_i$. 

Let $W_j=W_0\oplus W_i=V\oplus W_i$. Then $W_j$ is strongly-integrable. By \cite{TL99} theorem 5.2.1 (see also \cite{CKLW18} theorem 3.4), there exists a (continuous) projective representation $\fk U_j$ of $\scr G$ on $\mc H_j$ such that for any  $f\in e_1\cdot C^\infty_c(S^1,\mathbb R)$, the unitary operator $e^{i\ovl {Y_j(\nu,f)}}$ belongs to the equivalence class $\fk U_j(\wtd\exp(iT(f)))\in P\mc U(\mc H_j)$. Notice that $\mc H_j=\mc H_0\oplus\mc H_i$ and
\begin{align}
e^{i\ovl {Y_j(\nu,f)}}=\diag(e^{i\ovl {Y(\nu,f)}}, e^{i\ovl {Y_i(\nu,f)}}).\label{eq95}
\end{align}
Thus, for any $g_0\in\scr G$ of the form $\wtd\exp(iT(f))$, any element of $\mc U(\mc H_i)$ belonging to the equivalence class $\fk U_j(g_0)\in P\mc U(\mc H_j)$ takes the form $\diag(\fk V_0,\fk V_i)$ where $\fk V_0,\fk V_i$ are unitary operators on $\mc H_0,\mc H_j$ respectively, and $\fk V_0$ is a representing element of $U(g_0)$. (Recall that $U:\scr G\rightarrow P\mc U(\mc H_0)$ is integrated from the action of the Virasoro algebra on $V$.) By remark \ref{lb58}, $\scr G$ is generated by  elements of the form $\wtd\exp(iT(f))$ where $f\in e_1\cdot C^\infty_c(S^1,\mathbb R)$. (Here, we do not require $f$ to be supported in some open interval.) Thus the previous statement is true for any $g_0\in\scr G$. We now define a map $U_i':\scr G_{\mc A_V}\rightarrow\mc U(\mc H_i)$ as follows. Choose any $g=(g_0,\fk V_0)\in\scr G_{\mc A_V}\subset \scr G\times\mc U(\mc H_0)$, noting that $\fk V_0$ belongs to the equivalence class $U(g_0)$. Then one can find a unique $\fk V_i$ such that $\diag(\fk V_0,\fk V_i)$ belongs to the equivalence class $\fk U_j(g_0)$. We set $U_i(g)=\fk V_i$. It is easy to check that $U_i'$ is a homomorphism of groups. We thus obtain a unitary representation $U_i'$ of $\scr G_{\mc A_V}$ on $\mc H_i$. Moreover, if $g$ is of the form \eqref{eq91}, then, by \eqref{eq95}, we have $U_i'(g)=\lambda e^{i\ovl{Y_i(\nu,f)}}$. Therefore, to finish the proof, it remains to check that $U_i=U_i'$.

By remark \ref{lb58},  it suffices to show  $U_i'(g)=U_i(g)$ for any $g=(\wtd\exp(iT(f)),\lambda e^{i\ovl {Y(\nu,f)}})$ satisfying $I\in\mc J,f\in e_1\cdot C^\infty_c(I,\mathbb R),|\lambda|=1$. This follows from the strong-integrability of $W_i$:
\begin{align*}
U_i(g)=\pi_{i,I}(U(g))=\pi_{i,I}\big(\lambda e^{i\ovl {Y(\nu,f)}}\big)=\lambda e^{i\ovl {Y_i(\nu,f)}}=U_i'(g).
\end{align*}
\end{proof}

A more detailed study of the strong locality of VOA modules can be found in \cite{CWX}. (See also \cite{Ten18} for related topics.) Here we give a criterion for strong integrability which will be enough for applications to various examples.  To begin with, we let $\mc C$ be a full rigid monoidal subcategory of $\Rep^\uni(V)$. In other words, $\mc C$ is a class of objects of $\Rep^\uni(V)$ satisfying the following conditions:\\
(a) $\mc C$ contains the identity object $V$.\\
(b) If $W_i\in\mc C$, then any subobject of $W_i$ is equivalent to an object of $\mc C$.\\
(c) If $W_i\in\mc C$, then its dual $W_{\ovl i}$ is  equivalent to an object of $\mc C$.\\
(d) If $W_i,W_j\in\mc C$ then $W_i\boxtimes W_j\in\mc C$.

\begin{df}\label{lb51}
	Assume that any unitary $V$-module in $\mc C$ is energy-bounded. If $W_i\in\mc C$ and $w^{(i)}\in W_i$ is homogeneous, we say that the action $w^{(i)}\curvearrowright \mc C$ satisfies the \textbf{strong intertwining property}, if for any $W_j,W_k\in\mc C$, and $\mc Y_\alpha\in\mc V{k\choose i~j}$, the following conditions are satisfied:\\
	(a) $\mc Y_\alpha(w^{(i)},x)$ is energy-bounded.\\
	(b) For any homogeneous $v\in V$,  $\wtd I\in\Jtd$, $J\in\mc J$ which is disjoint from $I$, and $\wtd f\in C^\infty_c(\wtd I),g\in C^\infty_c(J)$, the following diagram of preclosed operators commutes strongly:
	\begin{align}\label{eq86}
	\begin{CD}
	\mc H_j @> \quad Y_j(v,g)\quad  >> \mc H_j\\
	@V \mc Y_\alpha(w^{(i)},\wtd f) VV @V \mc Y_\alpha(w^{(i)},\wtd f) VV\\
	\mc H_{k} @>\quad Y_k(v,g)\quad >> \mc H_{k}
	\end{CD}.
	\end{align}
\end{df}

Let $\mc F$ be a set of objects of $\mc C$. We say that $\mc F$ \textbf{generates} $\mc C$, if any irreducible object of $\mc C$ is equivalent to a subobject of a tensor product of elements in $\mc F$. The following theorem can be proved in a very similar way as \cite{Gui17b} theorem 4.8.

\begin{thm}\label{lb31}
Let $V$ be unitary and strongly local,  $\mc C$  a full rigid monoidal subcategory of $\Rep^\uni(V)$, and $\mc F$ a set of irreducible objects in $\mc C$. Assume that $\mc F$ generates $\mc C$, and for any $W_i\in\mc F$, there exists a non-zero homogeneous $w^{(i)}\in W_i$ such that $w^{(i)}\curvearrowright \mc C$ satisfies the strong intertwining property. Then any $W_k\in\mc C$ is strongly integrable. Moreover, for any $W_i\in\mc F,W_j,W_k\in\mc C,\mc Y_\alpha\in\mc V{k\choose i~j}$,  $\wtd I\in\Jtd$, $J\in\mc J$ disjoint from $I$, and $\wtd f\in C^\infty_c(\wtd I),y\in\mc A_V(J)$, the following diagram of preclosed operators commutes strongly.
	\begin{align}
\begin{CD}
\mc H_j @> \quad \pi_j(y)\quad  >> \mc H_j\\
@V \mc Y_\alpha(w^{(i)},\wtd f) VV @V \mc Y_\alpha(w^{(i)},\wtd f) VV\\
\mc H_{k} @>\quad \pi_k(y)\quad >> \mc H_{k}
\end{CD}.
\end{align}
\end{thm}
{~}\\

Note that definition \ref{lb51} does not rely on conditions \ref{CondF} and \ref{CondH}. Indeed, if $\mc C_0$ is a full rigid monoidal subcategory of $\Rep(V)$ whose objects are unitarizable,  and if $\mc C$ is the class of all unitary $V$-modules equivalent to some objects of $\mc C_0$, then we can still apply definition \ref{lb51}  to $\mc C$. Moreover,   condition \ref{CondH} (restricted to $\mc C$) will be a consequence of strong intertwining property; see remark \ref{lb45} and theorem \ref{lb54}. On the other hand, under the assumption of condition \ref{CondH},   the strong intertwining property for $w^{(i)}\curvearrowright\mc C$ can be equivalently stated as follows: For any homogeneous $v\in V$, and any $W_j\in\mc C,\wtd I\in\Jtd$ which is disjoint from $I$, and $\wtd f\in C^\infty_c(\wtd I),g\in C^\infty_c(J)$, the following diagram of preclosed operators commutes strongly:
\begin{align}\label{eq87}
\begin{CD}
\mc H_j @> \quad Y_j(v,g)\quad  >> \mc H_j\\
@V \mc L_i(w^{(i)},\wtd f) VV @V \mc L_i(w^{(i)},\wtd f) VV\\
\mc H_{ij} @>\quad Y_{ij}(v,g)\quad >> \mc H_{ij}
\end{CD}.
\end{align}

To see the equivalence of the two statements, note that  condition \eqref{eq87} is clearly a special case of the statement in definition \ref{lb51}. Now assume condition \eqref{eq87}. To prove \eqref{eq86}, we recall that $\mc Y_\alpha$ can be  identified with a morphism  $T_\alpha\in\Hom_V(W_i\boxtimes W_j,W_k)$ in a natural way. Then $\mc Y_\alpha=T_\alpha\mc L_i|_j$ by equation \eqref{eq88}. Since the two small diagrams of preclosed operators in
\begin{align}
\begin{CD}
\mc H_j @> \quad Y_j(v,g)\quad  >> \mc H_j\\
@V \mc L_i(w^{(i)},\wtd f) VV @V  V\mc L_i(w^{(i)},\wtd f)V\\
\mc H_{ij} @>\quad Y_{ij}(v,g)\quad >> \mc H_{ij}\\
	@V T_\alpha VV @VVT_\alpha V\\
\mc H_{k} @>\quad Y_k(v,g)\quad >> \mc H_{k}
\end{CD}
\end{align}
commute strongly, we have the strong commutativity of the large diagram by lemma \ref{lb44}, which is equivalent to the strong commutativity of \eqref{eq86}.

We close this section with a density property. First, for each $I\in\mc J$, we let $\mc A_V(I)_\infty$ be the set of all $x\in\mc A_V(I)$ such that $x\mc H_i^\infty\subset\mc H_i^\infty$ and $x^*\mc H_i^\infty\subset\mc H_i^\infty$ for any unitary $V$-module $W_i$. By \cite{Gui17b} proposition 4.2, $\mc A_V(I)_\infty$ is a strongly-dense $*$-subalgebra of $\mc A_V(I)$.

\begin{pp}\label{lb56}
Suppose that $W_i\in\mc F$, $w^{(i)}_0\in W_i$ is non-zero and homogeneous, and $w^{(i)}_0\curvearrowright\mc C$ satisfies the strong intertwining property. Then for each $W_j\in\mc C$ and $\wtd I\in\Jtd$, vectors of the form $\mc L_i(w^{(i)}_0,\wtd f)w^{(j)}$ (where $\wtd f\in C_c^\infty(\wtd I)$ and $w^{(j)}\in W_j$) spans a dense subspace of $\mc H_{ij}$.
\end{pp}

\begin{proof}
By \cite{Gui17a} proposition A.3, for any $w\in W_{ij}$ and $z\in\mathbb C^\times$, if $\bk {w|\mc L_i(w^{(i)},z)w^{(j)}}=0$ for any homogeneous $w^{(i)}\in W_i,w^{(j)}\in W_j$, then $w=0$. Since $W_i$ is irreducible, by the proof of \cite{Gui17a} corollary 2.15, if $\bk {w|\mc L_i(w^{(i)}_0,z)w^{(j)}}=0$ for any homogeneous $w^{(j)}\in W_j$, then $w=0$.

Let $\mc W$ be the closure of the subspace spanned by all $\mc L_i(w^{(i)}_0,\wtd f)w^{(j)}$ (where $\wtd f\in C_c^\infty(\wtd I)$ and $w^{(j)}\in W_j$) which contains all $\mc L_i(w^{(i)}_0,\wtd f)\eta$ (where $\wtd f\in C_c^\infty(\wtd I)$ and $\eta\in\mc H_j^\infty$). We shall show that its orthogonal complement $\mc W^\perp$ is trivial. Suppose that we can prove that $\mc W^\perp$ is an $\mc A_V$-submodule of $\mc H_{ij}$. If $\mc W^\perp$ is non-trivial, then by \cite{Gui17b} corollary 4.4, there is a non-zero vector $w\in \mc W^\perp$. So $\bk {w|\mc L_i(w^{(i)}_0,\wtd f)w^{(j)}}=0$ for any $w^{(j)}\in W_j,\wtd f\in C_c^\infty(\wtd I)$. Then, by our definition of smeared intertwining operator \eqref{eq90}, we have $\bk {w|\mc L_i(w^{(i)}_0,z)w^{(j)}}=0$ for any $w^{(j)}\in W_j$ and any $z\in I$ whose argument is taken to be $\arg_I(z)$. Thus by the first paragraph, we must have $w=0$, which is a contradiction.

We now prove that $\mc W^\perp$ is $\mc A_V$-invariant.  Let $\mbb D^\times=\{z\in\mbb C:| z|\leq1,z\neq 0\}$. Fix any $\xi\in\mc W^\perp$. For any $\wtd f\in C_c^\infty(\wtd I),y\in\mc A_V(I^c)_\infty,\eta\in\mc H_j^\infty$, by the strong intertwining property, 
\begin{align*}
\bk{y\mc L_i(w^{(i)}_0,\wtd f)\eta|\xi}=\bk{\mc L_i(w^{(i)}_0,\wtd f)y\eta|\xi}=0.
\end{align*}
By the positivity of $\ovl{L_0}$, the function
\begin{align*}
z\mapsto \bk{y\cdot z^{\ovl{L_0}}\mc L_i(w^{(i)}_0,\wtd f)\eta|\xi}
\end{align*}
is a multi-valued continuous function on $\mbb D^\times$, analytic on its interior, and (by \eqref{eq78}) equals zero on a small interval of $S^1$ containing $1$. Thus, by Schwarz reflection principle, the function is always zero. Thus $\bk{ye^{it\ovl{L_0}}\mc L_i(w^{(i)}_0,\wtd f)\eta|\xi}=0$ for any $\wtd f\in C_c^\infty(\wtd I),\eta\in\mc H_j^\infty,t\in\mbb R,y\in\mc A_V(I^c)_\infty$. By \eqref{eq78}, we conclude that $\bk{y\mc L_i(w^{(i)}_0,\wtd f)\eta|\xi}=0$ for any $\wtd J\in\Jtd,\wtd f\in C_c^\infty(\wtd J),\eta\in\mc H_j^\infty,y\in\mc A_V(I^c)_\infty$. Another application of Schwarz reflection principle shows that $\bk{e^{it\ovl{L_0}}ye^{-it\ovl{L_0}}\mc L_i(w^{(i)}_0,\wtd f)\eta|\xi}=0$ for any $\wtd J\in\Jtd,\wtd f\in C_c^\infty(\wtd J),\eta\in\mc H_j^\infty,I_1\sjs I^c,y\in\mc A_V(I_1)_\infty,t\in\mbb R$. Thus, for any $K\in\mc J$ whose size is smaller than that of $I^c$, we have $\bk{y\mc L_i(w^{(i)}_0,\wtd f)\eta|\xi}=0$ for any $\wtd J\in\Jtd,\wtd f\in C_c^\infty(\wtd J),\eta\in\mc H_j^\infty,y\in\mc A_V(K)_\infty$. Thus $\mc W^\perp$ is invariant under the action of  $\mc A_V(K)_\infty$ whenever $K$ has smaller size than $I^c$. Thus $\mc W^\perp$ is $\mc A_V$-invariant by the additivity of conformal nets and by the fact that $\mc A_V(K)_\infty$ is dense in $\mc A_V(K)$.
\end{proof}

\subsection{Vertex categorical extensions}\label{lb48}

Let $V$ be unitary and strongly local. Let $\mc C$ be a full rigid monoidal subcategory of $\Rep^\uni(V)$ as in the last section. We assume that $\mc F$ is a generating set of irreducible objects in $\mc C$ satisfying the conditions described in theorem \ref{lb31}. Then by that theorem, any unitary $V$-module $W_k$ in $\mc C$ can be integrated to an $\mc A_V$-module $\mc H_k$. We define a $*$-functor $\fk F:\mc C\rightarrow\Rep(\mc A_V)$ mapping each $W_k\in\mc C$ to $\fk F(W_k)=\mc H_k$. If $W_k,W_{k'}\in\mc C$, and $G\in\Hom_V(W_k,W_{k'})$, then $\fk F(G)\in\Hom_{\mc A_V}(\mc H_k,\mc H_{k'})$ is the closure of $G$ if we regard $G$ as a densely-defined linear operator from $\mc H_k$ to $\mc H_{k'}$ with domain $W_k$. Then by \cite{CWX} or by \cite{Gui17b} theorem 4.3, $\fk F:\mc C\rightarrow\Rep(\mc A_V)$ is a fully-faithful $*$-functor. We now equip $\fk F(\mc C)$ with the braided $C^*$-tensor categorical structure  $(\fk F(\mc C),\boxdot,\ss)$  naturally equivalent to $(\mc C,\boxtimes,\ss)$ under the $*$-functor $\fk F$. So, for instance, if $W_i,W_j\in\mc C$, we set $\mc H_i\boxdot\mc H_j=\mc H_{ij}$ (not to be confused with $\mc H_i\boxtimes\mc H_j$ defined by Connes fusion) to be the $\mc A_V$-module integrated from $W_{ij}=W_i\boxtimes W_j$. The braid operator $\ss_{i,j}:\mc H_{ij}\rightarrow\mc H_{ji}$ is defined to be the closure of $\ss_{i,j}:W_{ij}\rightarrow W_{ji}$

\begin{df}
Let $\mc C$ be a full rigid monoidal subcategory of $\Rep^\uni(V)$, $W_i,W_j\in\mc C$. Choose homogeneous vectors $w^{(i)}\in W_i,w^{(j)}\in W_j$. We say that the actions $w^{(i)},w^{(j)}\curvearrowright \mc C$ satisfy the \textbf{strong braiding property}, if for any $W_k\in\mc C$, $\wtd I,\wtd J\in\Jtd$ such that $\wtd I$ is anticlockwise to $\wtd J$, $\wtd f\in C^\infty_c(\wtd I),\wtd g\in C^\infty_c(\wtd J)$, the following diagram of preclosed operators commutes strongly:
\begin{align}
\begin{CD}
\mc H_k @> \quad\mc R_j(w^{(j)},\wtd g)\quad >> \mc H_{kj}\\
@V \mc L_i(w^{(i)},\wtd f) VV @V \mc L_i(w^{(i)},\wtd f) VV\\
\mc H_{ik} @> \quad\mc R_j(w^{(j)},\wtd g)\quad>> \mc H_{ikj}
\end{CD}.
\end{align}
\end{df}

\begin{thm}[Construction of vertex categorical extensions]\label{lb36}
Assume that $V$ satisfies conditions \ref{CondA} - \ref{CondI}. Let $\mc C$ be a full rigid monoidal subcategory of $\Rep^\uni(V)$ whose objects are energy-bounded, and let $\mc F$ be a set of irreducible $V$-modules in $\mc C$, such that $\mc F$ generates $\mc C$. Suppose that for each $W_i\in\mc F$ we can find a non-zero homogeneous vector $w^{(i)}\in W_i$, such that the following conditions hold:\\
(a) For any $W_i\in\mc F$, the action $w^{(i)}\curvearrowright\mc C$ satisfies the strong intertwining property.\\
(b) For any $W_i,W_j\in\mc F$, the actions $w^{(i)}, w^{(j)}\curvearrowright\mc C$ satisfy the strong braiding property.\\
Then objects in $\mc C$ are strongly integrable, and there exists a categorical local extension $\scr E^\loc=(\mc A_V,\fk F(\mc F),\boxdot,\fk H)$ of $\mc A_V$, which can be extended to a unique vector-labeled  closed categorical extension $\ovl{\scr E}=(\mc A_V,\fk F(\mc C),\boxdot,\mc H)$. Moreover, $\ovl{\scr E}$ is conformal.
\end{thm}
\begin{proof}
For each $\wtd I\in\Jtd,W_i\in\mc F$, we let $\fk H_i(\wtd I)=C^\infty_c(\wtd I)\times \mc A_V(I)$. Choose $\fk a=(\wtd f,x)\in C^\infty_c(\wtd I)\times \mc A_V(I)=\fk H_i(I)$. For each $W_k\in\mc C$, consider the left polar decomposition  $\ovl{\mc L_i|_k(w^{(i)},\wtd f)}=UH$ of the closed operator $\ovl{\mc L_i|_k(w^{(i)},\wtd f)}$ from $\mc H_k$ to $\mc H_{ik}$, where $U$ is the partial isometry (the phase) from $\mc H_k$ to $\mc H_{ik}$, and $H$ is the self-adjoint operator on $\mc H_k$. We write the phase $U$ as $U_i|_k(w^{(i)},\wtd f)$. Similarly, we let $V_i|_k(w^{(i)},\wtd f)$ be the phase of $\ovl{\mc R_i|_k(w^{(i)},\wtd f)}$.

Now for any $\xi^{(k)}\in\mc H_k$, we define
\begin{align}
L(\fk a,\wtd I)\xi^{(k)}=U_i|_k(w^{(i)},\wtd f)\pi_k(x)\xi^{(k)},\qquad R(\fk a,\wtd I)\xi^{(k)}=V_i|_k(w^{(i)},\wtd f)\pi_k(x)\xi^{(k)}.
\end{align}
We now verify that such construction makes $\scr E^\loc=(\mc A_V,\fk F(\mc F),\boxdot,\fk H)$ a categorical local extension of $\mc A_V$. By the strong intertwining property and theorem \ref{lb31}, the actions of $L(\fk a,\wtd I)$ and $R(\fk a,\wtd I)$ on $\mc H_k$ satisfy $L(\fk a,\wtd I)\in\Hom_{\mc A(I^c)}(\mc H_k,\mc H_{ik})$, $R(\fk a,\wtd I)\in\Hom_{\mc A(I^c)}(\mc H_k,\mc H_{ki})$. Isotony is easy to check. Since $\mc L_i$ and $\mc R_i$ are related by the braid operator $\ss$, so do their phases. So braiding is checked. Neutrality follows immediately from the braiding and the coherence theorem of $\ss$. Functoriality follows from \eqref{eq70}, \eqref{eq74}. The Reeh-Schlieder property and the density of fusion products follows from proposition \ref{lb56}. Finally, locality follows immediately from the strong braiding property.

Thus we've proved that $\scr E^\loc=(\mc A_V,\fk F(\mc F),\boxdot,\fk H)$ is a categorical local extension. By theorem \ref{lb18}, there exists a unique vector-labeled closed categorical extension  $\ovl{\scr E}=(\mc A_V,\fk F(\mc C),\boxdot,\mc H)$ containing $\scr E^\loc$. By theorem \ref{lb35}, $\ovl{\scr E}$ is conformal.
\end{proof}

\begin{co}
Assume that $V$ and $\mc C$ satisfy the conditions in theorem \ref{lb36}. Then $(\mc C,\boxtimes,\ss)$ is equivalent to a braided $C^*$-tensor subcategory of $(\Rep(\mc A_V),\boxtimes,\mathbb B)$ under the $*$-functor $\fk F$.
\end{co}
\begin{proof}
This follows immediately from  theorems \ref{lb36} and \ref{lb12}.
\end{proof}

The construction (but not just the existence) of the conformal categorical extensions in theorem \ref{lb36} is very important to us. We call them \textbf{vertex categorical extensions}. In the following we give some criteria for the strong intertwining and  braiding properties, which are the crucial conditions required in theorem \ref{lb36}. The following lemma can be proved using results in \cite{TL99} (see \cite{Gui17a} theorem B.9)

\begin{lm}\label{lb39}
Let $D$ be a self-adjoint positive operator on a Hilbert space $\mc H$, and let $\mc H^\infty=\bigcap_{n\in\mathbb Z_{\geq0}}\Dom(D^n)$ be the dense subspace of smooth vectors in $\mc H$. Suppose that $K,T$ are preclosed operators on $\mc H$ satisfying the following conditions:\\
(a) $\Dom(K)=\Dom(T)=\mc H^\infty$, $K\mc H^\infty\subset \mc H^\infty,T\mc H^\infty\subset\mc H^\infty$.\\
(b) $K$ is symmetric, which is equivalent to saying that $K=K^\dagger$ or $K\subset K^*$.\\
(c) $KT\xi=TK\xi$ for any $\xi\in\mc H^\infty$.\\
(d) There exists $m\in\mathbb Z_{\geq0}$, such that for any $n\in\mathbb Z_{\geq0}$, we can find  positive numbers $|K|_{n+1},|K|_{D,n+1},|T|_{n+m}$, such that for all $\xi\in\mc H^\infty$,
\begin{gather}
\lVert (1+D)^nK\xi\lVert\leq |K|_{n+1}\lVert (1+D)^{n+1}\xi \lVert,\\
\lVert (1+D)^n(DK-KD)\xi\lVert\leq |K|_{D,n+1}\lVert (1+D)^{n+1}\xi \lVert,\\
\lVert (1+D)^nT\xi\lVert\leq |T|_{n+m}\lVert (1+D)^{n+m}\xi \lVert.
\end{gather}
Then $\ovl K$ is self-adjoint, and $T$ and $K$ commute strongly.
\end{lm}

The above lemma can be applied to the situations where one of the two preclosed operators $A$ and $B$ is symmetric. In the case that neither of them is symmetric, we need to decompose $A$ into its real and imaginary parts $A=H+iK$ (where $H$ and $K$ are symmetric), prove the strong commutativity of $H,K$ with $B$, and finally show the strong commutativity of $A$ and $B$ by linearity. However, to be rigorous, we have to first verify (at least in our situations) that linear sums preserve the strong commutativity of preclosed operators. This is achieved by the following very useful lemma.

\begin{lm}\label{lb44}
	Let $P(z_1,\cdots,z_m)$ and $Q(\zeta_1,\cdots,\zeta_n)$ be polynomials of $z_1,\dots,z_m$ and $\zeta_1,\dots,\zeta_n$ respectively. Let $D$ be a self-adjoint positive operator on $\mc H$, and set $\mc H^\infty=\bigcap_{n\in\mathbb Z_{\geq0}}\Dom(D^n)$. Choose  preclosed operators $A_1,\dots,A_m$ and $B_1,\dots,B_n$   on 
	$\mc H$ with common invariant (dense) domain $\mc H^\infty$. Assume that there exists $\varepsilon>0$ such that   $e^{itD}A_re^{-itD}$ commutes strongly with $B_s$ for any $r=1,\dots,m,s=1,\dots,n$, and  $t\in(-\varepsilon,\varepsilon)$.  Assume also that the unbounded operators $A=P(A_1,\cdots.A_m),B=Q(B_1,\cdots,B_n)$ (with common domain $\mc H^\infty$) are preclosed. Then $A$ commutes strongly with $B$.
\end{lm}
Note that here $P(A_1,\cdots.A_m)$ and $Q(B_1,\cdots,B_n)$ are understood in the obvious way, i.e., by substituting the operators into the polynomials. So, for instance, if $P(z_1,z_2,z_3)=2z_1z_3+z_2^2$ then $A=2A_1A_3+A_2^2$. Note also that an invariant domain of an unbounded operator means that this domain is invariant under the action of this operator.
\begin{proof}
We first prove the special case when  $m=1$ and $P(z_1)=z_1$. Then $A=A_1$. Let $\mc M$ be the von Neumann algebra generated by $\ovl A,\ovl A^*$, and choose an arbitrary $x\in\mc M$. Then it is not hard to show that for any $h\in C^\infty_c(-\varepsilon,\varepsilon)$ satisfying $\int^\varepsilon_{-\varepsilon} h(t)dt=1$, the operator $x_h=\int^\varepsilon_{-\varepsilon} e^{itD}x e^{-itD}h(t)dt$ satisfies $x_h\mc H^\infty\subset\mc H^\infty,x_h^*\mc H^\infty\subset\mc H^\infty$ (see, for example, the proof of \cite{Gui17b} theorem 4.2), and that as $h$ converges to the delta-function at $0$, $x_h$ converges strongly to $x$. Now, by assumption, each $x_h$ commutes strongly with any $B_s$. Therefore, $x_hB_s\xi=B_sx_h\xi,x_h^*B_s=B_sx_h^*\xi$  for any $\xi\in\mc H^\infty$ and any $s=1,\dots,n$. By the invariance of $\mc H^\infty$ under $B_1,\dots,B_n$,  we have $x_hB\xi=Bx_h\xi, x_h^*B\xi=Bx_h^*\xi$ for any $\xi\in\mc H^\infty$, which implies the strong commutativity of $x_h$ and $\ovl B$ (see \cite{Gui17a} section B.1).  Since $x_h$ converges strongly to $x$, $x$ and $\ovl B$ also commute strongly. Thus $\ovl A$ and $\ovl B$ commute strongly. Note that for any $t\in (-\varepsilon,\varepsilon)$, $e^{itD}A e^{-itD}$ and $B_1,\dots,B_n$ satisfy a similar condition. Therefore $e^{itD}A e^{-itD}$ also commutes strongly with $B$.

Now for general $m$ and $P$, we know from the last paragraph that $A_r$ commutes strongly with $e^{-itD}Be^{itD}$ for any $r=1,\dots,m$ and $t\in(-\varepsilon,\varepsilon)$. Therefore, by the last paragraph again, $B$ commutes strongly with $A$.
\end{proof}

Using the above two lemmas and the rotation covariance of smeared intertwining operators \eqref{eq78}, one can prove the following theorems.

\begin{thm}
Let $V$ be  unitary and strongly local,  $\mc C$  a full rigid monoidal subcategory of $\Rep^\uni(V)$ whose objects are energy-bounded. Choose  $W_i\in\mc C$, and a quasi-primary vector $w^{(i)}\in W_i$. If $\mc L_i|_k(w^{(i)},x)$ satisfies $1$-st order energy bounds for any $W_k\in\mc C$, then the action $w^{(i)}\curvearrowright\mc C$ satisfies the strong intertwining property.
\end{thm}
\begin{proof}
See either step 2 of the proof of \cite{Gui17b} theorem 4.8, or the proof of the next theorem.
\end{proof}

\begin{thm}\label{lb41}
Let $V$ be  unitary and strongly local,  $\mc C$  a full rigid monoidal subcategory of $\Rep^\uni(V)$.  Choose $W_i,W_j\in\mc C$,  quasi-primary $w^{(i)}\in W_i$, and homogeneous $w^{(j)}\in W_j$. If for any $W_k\in\mc C$, $\mc L_i|_k(w^{(i)},x)$ satisfies $1$-st order energy bounds, and $\mc L_j|_k(w^{(j)},x)$ is energy bounded, then the actions $w^{(i)},w^{(j)}\curvearrowright\mc C$ satisfy the strong braiding property. 
\end{thm}
\begin{proof}
For each $\wtd I,\wtd J\in\Jtd$ with $\wtd I$ anticlockwise to $\wtd J$, and any $\wtd f\in C^\infty_c(\wtd I),\wtd g\in C^\infty_c(\wtd J)$, consider the diagram
\begin{align}
\begin{CD}
\mc H_k @> \quad\mc R_j|_k(w^{(j)},\wtd g)\quad >> \mc H_{kj}\\
@V \mc L_i|_k(w^{(i)},\wtd f) VV @V \mc L_i|_{kj}(w^{(i)},\wtd f) VV\\
\mc H_{ik} @> \quad\mc R_j|_{ik}(w^{(j)},\wtd g)\quad>> \mc H_{ikj}
\end{CD}.
\end{align}
Set $\mc H=\mc H_k\oplus\mc H_{kj}\oplus\mc H_{ik}\oplus\mc H_{ikj}$ and $\mc H^\infty=\mc H^\infty_k\oplus\mc H^\infty_{kj}\oplus\mc H^\infty_{ik}\oplus\mc H^\infty_{ikj}$, and extend  $\mc L_i|_k(w^{(i)},\wtd f),\mc L_i|_{kj}(w^{(i)},\wtd f)$ (resp. $\mc R_j|_k(w^{(j)},\wtd g),\mc R_j|_{ik}(w^{(j)},\wtd g)$) to a preclosed operator $A$ (resp. $B$) on $\mc H$ (see definition \ref{lb37}).  Let $H=(A+A^\dagger)/2$ and $K=(A-A^\dagger)/(2i)$ be symmetric operators on $\mc H$ with domains $\mc H^\infty$. By lemma \ref{lb39}, theorem \ref{lb38}, the energy bounds conditions of intertwining operators, the adjoint formula \eqref{eq76}, and  equation \eqref{eq79} which shows the  energy bounds of $[\ovl {L_0},H]$ and $[\ovl{L_0},K]$, the preclosed operators $H$ and $K$ commute strongly with $B$. Therefore, by \eqref{eq78} and lemma \ref{lb44}, $A=H+iK$ commutes strongly with $B$.
\end{proof}

We summarize the results of this section in the following theorem which will be convenient for applications.

\begin{thm}\label{lb40}
Assume that $V$ satisfies conditions \ref{CondA} - \ref{CondI}. Let $\mc C$ be a full rigid monoidal subcategory of $\Rep^\uni(V)$ whose objects are energy-bounded, and let $\mc F$ be a set of irreducible $V$-modules in $\mc C$, such that $\mc F$ generates $\mc C$. Suppose that for each $W_i\in\mc F$ we can find a non-zero quasi-primary vector $w^{(i)}$, such that whenever $W_i,W_j\in\mc F,W_k\in\mc C$, the intertwining operators $\mc L_i|_k(w^{(i)},x)$ and $\mc L_j|_k(w^{(j)},x)$ satisfy $1$-st order energy bounds. Then objects in $\mc C$ are strongly integrable, and there exists a  vertex categorical extension $\ovl{\scr E}=(\mc A_V,\fk F(\mc C),\boxdot,\ss)$ of $\mc A_V$. Consequently, $(\mc C,\boxtimes,\ss)$ is equivalent to a braided $C^*$-tensor subcategory of $(\Rep(\mc A_V),\boxtimes,\mathbb B)$ under the $*$-functor $\fk F$.

Moreover, if $W_h,W_l\in\mc C$, $w^{(h)}\in W_h$ and $w^{(l)}\in W_l$ are homogeneous, and $w^{(h)}\curvearrowright\mc C,w^{(l)}\curvearrowright\mc C$ satisfy the strong intertwining property, then the actions $w^{(h)},w^{(l)}\curvearrowright\mc C$ satisfy the strong braiding property.
\end{thm}

\begin{proof}
The claim of the first paragraph follows directly from what we've proved in this section. We now prove the second half. For each $\wtd I\in\Jtd,\wtd f\in C^\infty_c(\wtd I)$, we consider, for any $W_k\in\mc C$, the preclosed operator $\mc L_h|_k(w^{(h)},\wtd f)$ from $\mc H_k$ to $\mc H_{hk}$. Let $\ovl {\mc L_h|_k(w^{(h)},\wtd f)}=HU$ be the right polar decomposition, and let $H=\int_0^{+\infty} \lambda dE_k(\lambda)$ be the spectral decomposition of the positive operator $H$. Now  choose an arbitrary $\lambda\geq0$, and let $\fk a=(\lambda,\wtd f)$. We define a bounded linear operator $A(\fk a,\wtd I)$ acting on any $\mc H_k$ (where $W_k\in\mc C$), such that for any $\xi^{(k)}\in\mc H_k$,
\begin{align*}
A(\fk a,\wtd I)\xi^{(k)}=\ovl{E_k(\lambda)\ovl {\mc L_h|_k(w^{(h)},\wtd f)}}\xi^{(k)}.
\end{align*}
Then by the strong intertwining property of $w^{(h)}\curvearrowright\mc C$, $A(\fk a,\wtd I)\in\Hom_{\mc A_V(I^c)}(\mc H_k,\mc H_{hk})$. By the functoriality of $\mc L_h$ and theorem \ref{lb41}, $(A,\fk a,\wtd I,\mc H_h)$ is a left operator (see definition \ref{lb42}) of the categorical local extension $\mc E^\loc=(\mc A_V,\fk F(\mc F),\boxdot,\fk H)$ constructed in theorem \ref{lb36}.

Similarly, for any $\wtd J\in\Jtd$ clockwise to $\wtd I$, any $\wtd g\in C^\infty_c(\wtd J)$, and any $W_k\in\mc C$, we take the right polar decomposition $\ovl {\mc R_l|_k(w^{(l)},\wtd g)}=H'U'$, take the spectral decomposition $H'=\int_0^{+\infty}\mu dE'_k(\mu)$, choose an arbitrary $\mu\geq0$,   and let $\fk b=(\mu,\wtd g)$. Then there is a right operator $(B,\fk b,\wtd J,\mc H_l)$ of $\mc E^\loc$, such that for any $\xi^{(k)}\in\mc H_k$,
\begin{align*}
B(\fk b,\wtd J)\xi^{(k)}=\ovl{E'_k(\mu)\ovl {\mc R_l|_k(w^{(l)},\wtd g)}}\xi^{(k)}.
\end{align*}
Therefore, by theorem \ref{lb43}, the diagram
\begin{align*}
\begin{CD}
\mc H_k @> \quad\ovl{E'_k(\mu)\ovl {\mc R_l|_k(w^{(l)},\wtd g)}}\quad  >> \mc H_{kl}\\
@V \ovl{E_k(\lambda)\ovl {\mc L_h|_k(w^{(h)},\wtd f)}} VV @VV \ovl{E_{kl}(\lambda)\ovl {\mc L_h|_{kl}(w^{(h)},\wtd f)}} V\\
\mc H_{hk} @> \quad\ovl{E'_{hk}(\mu)\ovl {\mc R_l|_{hk}(w^{(l)},\wtd g)}}\quad >> \mc H_{hkl}
\end{CD}
\end{align*}
commutes adjointly for any $\lambda,\mu\geq0$. This proves the strong commutativity of the diagram
\begin{align*}
\begin{CD}
\mc H_k @> \quad\ovl {\mc R_l|_k(w^{(l)},\wtd g)}\quad  >> \mc H_{kl}\\
@V \ovl {\mc L_h|_k(w^{(h)},\wtd f)} VV @VV \ovl {\mc L_h|_{kl}(w^{(h)},\wtd f)}  V\\
\mc H_{hk} @> \quad\ovl {\mc R_l|_{hk}(w^{(l)},\wtd g)}\quad >> \mc H_{hkl}
\end{CD}
\end{align*}
for any $W_k\in\mc C$. Therefore $w^{(h)},w^{(l)}\curvearrowright\mc C$ satisfy the strong braiding property.
\end{proof}

\begin{rem}\label{lb45}
We remark that in theorems \ref{lb36} and \ref{lb40}, condition \ref{CondH} is not necessarily needed. By all the other conditions and the strong intertwining property, we can prove the positivity of the sesquilinear form $\Lambda$ on $\mc V{t\choose i~j}$ for any $W_i,W_j\in\mc C,W_t\in\mc E$ using the arguments in \cite{Gui17a,Gui17b}. Thus $\mc C$ is a braided $C^*$-tensor category, which is enough to prove these theorems. Indeed, in \cite{Gui17b} we gave two criteria (conditions A and B in section 5.3) for the positivity of $\Lambda$, both concerning the ($1$-st order) energy bounds conditions for intertwining operators. The reason these energy bounds conditions are required is to guarantee the strong intertwining property, which is the essential property for proving the main results of that paper. Those results clearly hold if we replace the $1$-st order energy bounds condition with the more general strong intertwining properties.\footnote{\cite{Gui17b} conditions A and B  require that the homogeneous vectors are quasi-primary. This is not necessary once we know the strong intertwining property. Indeed, the quasi-primary condition is used only in the following two places: (1) If $w^{(i)}$ is quasi-primary and $\mc Y_\alpha(w^{(i)},x)$ satisfies $1$-st order energy bounds, then the adjoint intertwining operator $\mc Y_{\alpha^*}(w^{(i)},x)$ also satisfies $1$-st order energy bounds, cf. \cite{Gui17a} proposition 3.4. (2) The formula for the adjoint of generalized intertwining operators, cf. \cite{Gui17b} corollary 5.7. Now (1) is used only to prove the strong intertwining property. As for (2), it is not hard to write down the adjoint formula for general homogeneous vectors using the non-smeared version of \cite{Gui17a} equation (3.25) and \cite{Gui17b} remark 5.6.  Therefore one can safely drop the quasi-primary condition once the strong intertwining property is known.} Thus, let us summarize the positivity result in \cite{Gui17a,Gui17b} in the following theorem.
\end{rem}

\begin{thm}\label{lb54}
Assume that $V$ satisfies conditions \ref{CondA}-\ref{CondF} and \ref{CondI}. Let $\mc C_0$ be a full rigid monoidal subcategory of $\Rep(V)$ whose objects are unitarizable and energy-bounded. Let $\mc F$ be a set of irreducible $V$-modules generating $\mc C_0$, and fix for each $W_i\in\mc F$ a unitary structure. Let $\mc C$ be the $C^*$-category of all unitary $V$-modules equivalent to some objects of $\mc C_0$. Suppose  for each $W_i\in\mc F$ we can find a non-zero homogeneous vector $w^{(i)}\in W_i$ such that  $w^{(i)}\curvearrowright\mc C$ satisfies the strong intertwining property. Then for any $W_j,W_k\in\mc C,W_t\in\mc E$, the sesquilinear form $\Lambda$ on $\mc V{t\choose j~k}^*$ is positive. As a consequence, $\mc C$ becomes a unitary ribbon fusion category.
\end{thm}

Here the notions of full rigid monoidal subcategories and generating sets of objects are understood in the same way as in section \ref{lb55}.

\section{Examples and applications}

In this chapter, we apply the main results in the previous chapter (mainly theorems \ref{lb36} and \ref{lb40}) to various examples. Let us assume that $V$ is a unitary regular VOA of CFT type. Here ``regular" means that any weak $V$-module is completely reducible \cite{DLM95}. Then $V$ satisfies conditions \ref{CondA} - \ref{CondF} by a series works of Huang and Lepowsky \cite{HL95a,HL95b,HL95c,Hua95,Hua05a,Hua05b,Hua08a,Hua08b}.  Examples of such $V$ include unitary Virasoro VOAs (minimal models), unitary affine VOAs (WZW models), and  lattice VOAs (cf.\cite{DLM95}). For these examples, all semisimple representations are unitarizable. (See \cite{FQS84,Wang93} for unitary Virasoro VOAs, \cite{Kac94,FZ92} for unitary affine VOAs, and \cite{FLM89} or \cite{LL12} for lattice VOAs.) Therefore condition \ref{CondG} holds for these examples. Unitary Heisenberg VOAs also satisfy conditions \ref{CondA} - \ref{CondG}. Indeed, the convergence of products of intertwining operators (condition \ref{CondC}) can be checked directly using the explicit construction of intertwining operators (the well known ``vertex operator" construction). One can also compute by hand the tensor categorical structures of their representation categories using the braid and fusion relations obtained in \cite{DL93}. A detailed discussion will be given later in this chapter. Condition \ref{CondI} also holds for all these examples: see \cite{BS90} section 2 for Virasoro, affine, and Heisenberg VOAs; see \cite{TL04} chapter VI for lattice VOAs (see also \cite{Gui18} theorem A.6). 

In the following, we will verify the strong intertwining property and the strong braiding property for many of these examples. As discussed in remark \ref{lb45}, the positivity of the sesquilinear form  $\Lambda$, and hence the unitarity of the relevant braided ribbon fusion categories are  consequences but not  assumptions of these analytic properties.

\subsection{Unitary Virasoro and affine VOAs}\label{lb50}

Suppose that $V$ is a unitary Virasoro VOA $L(c,0)$ ($c<1$), or a unitary affine VOA $L_\gk(l,0)$ at level $l\in\mathbb Z_{\geq0}$, where $\gk$ is a complex simple Lie algebra of type $A$, $C$, or $G_2$. Then by the works of \cite{Loke94} (for Virasoro VOAs), \cite{Was98} (for type $A$ affine VOAs), \cite{Gui18} (for type $C$ and $G_2$ affine VOAs), there exists a set $\mc F$ of irreducible unitary $V$-modules, such that for any $W_i\in\mc F,W_k\in\Rep^\uni(V)$, and any lowest weight vector $w^{(i)}\in W_i$ (which is automatically quasi-primary), the intertwining operator $\mc L_i|_k(w^{(i)},x)$ satisfies $1$-st order energy bounds. (Indeed, except for type $G_2$ affine VOAs, the $0$-th order energy bounds hold.) In the case $V=L(c,0)$ where $c=1-\frac 6{m(m+1)}$ ($m=2,3,4,\dots$), we can choose $\mc F=\{L(c,h_{1,2}), L(c,h_{2,2})\}$, where for each $r=1,2,\dots,m-1,s=1,2,\dots,m$, $h_{r,s}=\frac{((m+1)r-ms)^2-1}{4m(m+1)}$ is the highest weight of $L(c,h_{r,s})$. If $V=L_{\gk}(l,0)$ and $\gk$ is of type $A,C,G_2$, one can choose $L(l,\Box)$ to be the smallest (in the sense of the dimension of the lowest weight subspace) non-vacuum irreducible $V$-module, and let $\mc F=\{L(l,\Box)\}$. Thus, by remark \ref{lb45} or by \cite{Gui17b} theorems 6.7 and 7.8, the sesquilinear form $\Lambda$ is always positive, and $\Rep^\uni(V)$ is a unitary fusion category. By \cite{Gal12} theorem 3.5, $\Rep^\uni(V)$ admits a unique unitary ribbon structure (which, by \cite{Gui17b} section 7.3, is defined by the twist $e^{2i\pi L_0}$). By \cite{Hua08b}, the $S$-matrix is non-degenerate. Therefore $\Rep^\uni(V)$ is a unitary modular tensor category. 

Now apply theorem \ref{lb40}, we see that $\fk F(\Rep^\uni(V))$ is closed under Connes fusion $\boxtimes$, and the braided $C^*$-tensor category $(\fk F(\Rep^\uni(V)),\boxtimes,\mathbb B)$ is equivalent to $(\Rep^\uni(V),\boxtimes,\ss)$. Therefore $(\fk F(\Rep^\uni(V)),\boxtimes,\mathbb B)$ is a unitary braided fusion category, which admits a unique unitary ribbon structure. We thus conclude that $\fk F(\Rep^\uni(V))$ is equivalent to $\Rep^\uni(V)$ as a unitary modular tensor category.

Finally we determine the category $\fk F(\Rep^\uni(V))$. For a general unitary regular $V$ of CFT type, it is widely believed that  $\fk F(\Rep^\uni(V))$ is the category $\Rep^\ssp(\mc A_V)$ of semisimple $\mc A_V$-modules. In the case that $V$ is one of the examples mentioned above, this conjecture can actually be proved. Indeed, if $V$ is a unitary minimal model, then by \cite{Wang93}, one has a complete classification of irreducible $V$-modules. By \cite{Xu00a} theorem 4.6 and \cite{KL04} corollary 3.3, irreducible $\mc A_V$-modules were also classified, and one easily sees that $V$ and $\mc A_V$ have the same number of equivalence classes of irreducible representations. (See also \cite{KL04} the discussions before corollary 3.6.) One thus concludes that $\fk F(\Rep^\uni(V))=\Rep^\ssp(\mc A_V)$.

Now assume that $V$ is a unitary affine VOA $L_\gk(l,0)$, the strong integrability of all representations of which has already been shown. Then by \cite{Hen19} theorems 26 and 27 (with $\gk=\fk{sl}_2$ excluded), or the first theorem any \cite{Zel15} theorem 2.16, any semisimple $\mc A_V$-module $\mc H_i$ can be integrated from an irreducible positive energy representation of $\wht{\fk g}$. Such  $\wht{\fk g}$-module can be extended uniquely to a unitary $V$-module $W_i$ such that $Y_i(X(-1)\Omega)_n=X(n)$ ($\forall n\in\mathbb Z,X\in\fk g$). By \cite{CKLW18} theorem 8.1, $\mc A_V$ is generated by operators of the form $\ovl{Y(X(-1)\Omega,f)}$. Therefore, by strong integrability, the $V$-module $W_i$ integrates to $\mc H_i$. Thus the conjecture is proved in this case. We conclude the following:

\begin{thm}
Let $V$ be a unitary Virasoro VOA $L(c,0)$ ($c<1$), or a unitary affine VOA $L_\gk(\l,0)$ where $\gk$ is a complex simple Lie algebra of type $A$, $C$, or $G_2$. Then any unitary $V$-module is strongly integrable, $\Rep^\ssp(\mc A_V)$ is closed under Connes fusions, and   $(\Rep^\ssp(\mc A_V),\boxtimes,\mathbb B)$ and $(\Rep^\uni(V),\boxtimes,\ss)$ are equivalent as unitary modular tensor categories  under the $*$-functor $\fk F$.
\end{thm}

An important question in algebraic quantum field theory is to prove the complete rationality \cite{KLM01} of conformal nets corresponding to rational chiral conformal field theories. The first non-trivial examples of completely rational conformal nets are those associated to unitary affine  VOAs (WZW-models) of type $A$ by \cite{Xu00b}. The complete rationality of $c<1$ Virasoro nets was proved in \cite{KL04}. Now, with the help of the above theorem, we have the following expanded list of completely rational WZW-nets.

\begin{co}
If $V$ is a unitary affine VOA of type $A$, $C$, or $G_2$, then the conformal net $\mc A_V$ is completely rational.
\end{co}

\begin{proof}
By the previous theorem, $\Rep^\ssp(\mc A_V)$ is a fusion category since this is true for $\Rep^\uni(V)$. Thus the strong rationality of $\mc A_V$ follows from \cite{LX04} theorem 4.9.
\end{proof}

If $V=L_\gk(l,0)$ is a unitary affine VOA of type $B$ or $D$, a partial result can be obtained. Write $\gk=\fk{so}_n$ where $n\geq5$. The smallest non-vacuum irreducible $V$-module $L_\gk(l,\Box)$ (the one corresponding to the vector representation of $\gk$) unfortunately does not generate the whole tensor category $\Rep^\uni(V)$. One also needs to include the spin representations, in which case the linear energy bounds conditions of intertwining operators are not guaranteed. Set $\mc F=\{L_\gk(l,\Box) \}$. Then the tensor category $\mc C$ generated by $\mc F$ is the tensor subcategory of single-valued $V$-modules. Here an  $L_\gk(l,0)$-module $W_i$ is called single-valued if   the action of $\gk$ on the lowest weight subspace of any irreducible component of $W_i$ can be integrated to a representation of $SO(n)$ (but not just its covering space $\mathrm{Spin}(n)$). By \cite{TL04} theorem VI.3.1 and \cite{Gui18} theorem 3.3, for any lowest weight vector $w^{(\Box)}\in L_\gk(l,\Box)=W_\Box$ and any $W_k\in\Rep^\uni(V)$, the intertwining operator $\mc L_\Box|_k(w^{(\Box)},x)$ satisfies $0$-th order energy bounds. Therefore, by theorem \ref{lb40}, we have the following equivalence theorem.

\begin{thm}
Let $V$ be a unitary affine VOA of type $B$ or $D$, and let $\mc C$ be the tensor category of unitary single-valued $V$-modules. Then any object in $\mc C$ is strongly integrable, the category $\fk F(\mc C)$ of all $\mc A_V$-modules which can be integrated from objects in $\mc C$ is closed under Connes fusions, and  $(\fk F(\mc C),\boxtimes,\mathbb B)$ and $(\mc C,\boxtimes,\ss)$ are equivalent as unitary ribbon fusion categories under the $*$-functor $\fk F$. 
\end{thm}

We remark that we can prove the strong integrability of  any (not necessarily single-valued) representation  $W_i$ of $V$ in a slightly weaker sense: there exists an $\mc A_V$-module $(\mc H_i,\pi_i)$ such that $\pi_i(\ovl{Y(X(-1)\Omega,f)})=\ovl{Y_i(X(-1)\Omega),f)}$ whenever $X\in\fk g,I\in\mc J,f\in C^\infty_c(I)$, and $Y(X(-1)\Omega,f)$ is symmetric (see \cite{Gui17b} remark 5.8). This is due to the fact that any $Y_i(X(-1),x)$ satisfies $1$-st order energy bounds, so that the smeared vertex operators of which are intertwined strongly by any energy bounded intertwining operators (see \cite{Gui17a} proposition 3.16). Since, by \cite{CKLW18} theorem 8.1, operators of the form $\ovl{Y(X(-1)\Omega,f)}$  (where $f\in C^\infty_c(I)$) generate $\mc A_V(I)$ for each $I$, $\pi_i$ is uniquely determined by $Y_i$. So we can define a fully faithful $*$-functor $\fk F:\Rep^\uni(V)\rightarrow \Rep(\mc A_V)$.

However, it will be much harder to show that the whole modular tensor category $\Rep^\uni(V)$ is equivalent to its image in $\Rep(\mc A_V)$ (i.e., $\Rep^\ssp(\mc A_V)$) under the $*$-functor $\fk F$. The difficulty lies in that, due to lack of $1$-st order energy bounds, we don't know how to prove the strong braiding property for the intertwining operators whose charge spaces are double-valued representations (say, the spin representations). A possible way to tackle this problem is through conformal inclusions: one tries to realize $V$ as a unitary VOA extension of another unitary rational VOA $U$, such that there exists a generating set of irreducible $U$-modules whose intertwining operators satisfy $1$-st order energy-bounds. Then one can construct the vertex category extension of $\mc A_U$, which can be shown to be restricted to the one of $\mc A_V$ (notice that the tensor category of $V$ is smaller than that of $U$). The equivalence of $\Rep^\uni(V)$ and $\Rep^\ssp(\mc A_V)$ as ribbon categories can therefore be proved. (Indeed, we will use this method to treat lattice VOAs in subsequent sections.) A general theory of categorical extensions for VOA extensions will be developed in future works.

\subsection{Unitary Heisenberg VOAs}

Heiserberg VOAs are not rational as they have infinitely many equivalence classes of irreducible representations. But it is still interesting to study their tensor categories and categorical extensions, mainly because of their relations to Lattice VOAs (which are rational). The main purpose of this section is to prove that all intertwining operators of a unitary Heisenberg VOA satisfy the strong intertwining and braiding properties. This result will be used in the next section to construct vertex categorical extensions for even lattice VOAs.

Heisenberg VOAs share many similarities with affine VOAs, as the former are affinizations of (obviously non-semisimple) abelian Lie algebras. The main differences are that the levels add no constraints on Heisenberg VOAs, and that all (positive) levels are equivalent. So we may well assume that the level $l=1$. In the following we summarize some key features of  the tensor categories of Heisenberg VOAs. A detailed account of the representation theory of Heisenberg VOAs (as well as Lattice VOAs) can be found in \cite{LL12} chapter 6.  \cite{DL93} lays down the foundation of the tensor product theory for representations of Heisenberg VOAs. A brief exposition of this theory can be found in \cite{TZ11}. These works were written in the language of formal variables. In particular, the braid and fusion relations of intertwining operators were written in the form of the Jacobi identity for ``generalized vertex algebras". A translation of these results in the language of complex variables was provided in \cite{Gui18} chapter A, where a discussion of the energy-bounds condition is also included.\\

Let $V=L_\hk(1,0)$ be the unitary Heisenberg VOA for a unitary finite-dimensional complex abelian Lie algebra $\hk$. Here ``unitary" means that $\hk$ is equipped with an inner product $(\cdot|\cdot)$ and an anti-unitary involution $*$. Let $\hk_{\mathbb R}=\{X\in\hk:X^*=-X \}$ be the real Lie algebra for $\hk$. The real inner product $(\cdot|\cdot)$ on $i\hr$ induces a natural isomorphism  between $i\hr$ and its dual vector space $(i\hr)^*$.  The equivalence classes of irreducible unitary $V$-modules can then be identified with $(i\hr)^*\simeq i\hr$ in the following way: For any $\lambda\in i\hr$, we have an  irreducible  positive energy representation $W_\lambda=L_\hk(1,\lambda)$ of the affinization $\wht\hk$ of $\hk$, such that  $\alpha(0)w_\lambda=(\alpha|\lambda)w_\lambda$ for any lowest conformal-weight vector $w_\lambda$ and any $\alpha\in \hk$.\footnote{From the perspective of infinite dimensional Lie algebras, $w_\lambda$ is called a \emph{highest} weight vector.} $L_\hk(1,\lambda)$ can be extended uniquely to a unitary $V$-module.

For any $\lambda,\mu,\nu\in i\hr$, The fusion rule $N^\nu_{\lambda\mu}=\dim\mc V{\nu\choose\lambda~\mu}=\dim\mc V{W_\nu\choose W_\lambda W_\mu}$ equals $1$ when $\nu=\lambda+\mu$, and $0$ otherwise. We therefore have $W_\lambda\boxtimes W_\mu\simeq W_{\lambda+\mu}$. If $\nu=\lambda+\mu$, there is a distinguished non-zero type $\nu\choose\lambda~\mu$ intertwining operator $\mc Y^\nu_{\lambda,\mu}$ defined using the well-known ``vertex operator" construction $\exp\big(\sum_{n>0}\frac {\lambda(-n)}nx^n\big)Y(v,x)\exp\big(-\sum_{n>0}\frac {\lambda(n)}nx^{-n}\big)$, with which the braid and fusion relations are easy to express. To simplify our notations, we let $\mc Y_\lambda$ act on any possible $W_\mu$ as the intertwining operator $\mc Y^\nu_{\lambda,\mu}$. By \cite{DL93} theorem 5.1, for any $\lambda,\mu,\nu\in i\hr,w^{(\lambda)}\in W_\lambda,w^{(\mu)}\in W_\mu,w^{(\nu)}\in W_\nu$, we have the fusion relation
\begin{align}
\mc Y_\lambda(w^{(\lambda)},z_1)\mc Y_\mu(w^{(\mu)},z_2)w^{(\nu)}=\mc Y_{\lambda+\mu}(\mc Y_\lambda(w^{(\lambda)},z_1-z_2)w^{(\mu)},z_2)w^{(\nu)}
\end{align}
for any $z_1,z_2\in\mathbb C^\times$ satisfying $0<|z_1-z_2|<|z_2|<|z_1|$ and $\arg z_1=\arg z_2=\arg(z_1-z_2)$, and the braid relation
\begin{align}
\mc Y_\lambda(w^{(\lambda)},z_1)\mc Y_\mu(w^{(\mu)},z_2)w^{(\nu)}=e^{i\pi(\lambda|\mu)}\mc Y_\mu(w^{(\mu)},z_2)\mc Y_\lambda(w^{(\lambda)},z_1)w^{(\nu)}
\end{align}
whenever $z_1,z_2\in S^1$ and $z_1$ is anticlockwise to $z_2$.

The energy bounds condition for the intertwining operators of $V$ was essentially proved in \cite{TL04} chapter VI. A brief explanation of the proof can be found in \cite{Gui18} chapter A. Here we only summarize  the results that will be used later: For any $\lambda,\mu\in i\hr$ and homogeneous $w^{(\lambda)}\in W_\lambda$, $\mc Y^{\lambda+\mu}_{\lambda,\mu}(w^{(\lambda)},x)$ is energy bounded.  If, moreover, $w^{(\lambda)}$ has lowest conformal dimension, and $(\lambda|\lambda)\leq 1$, then $\mc Y^{\lambda+\mu}_{\lambda,\mu}(w^{(\lambda)},x)$ satisfies $0$-th order energy bounds.

Now we can easily construct the vertex categorical extension for $V$. Set $\mc F=\{W_\lambda=L_\hk(1,\lambda):(\lambda|\lambda)\leq 1 \}$, which clearly generates $\Rep^\uni(V)$. For each $W_\lambda\in\mc F$, we choose a non-zero lowest weight vector $w^{(\lambda)}\in W_\lambda$. Then $\mc Y_\lambda(w^{(\lambda)},x)$ satisfies $0$-th order (and hence $1$-st order) energy bounds. Now  theorem \ref{lb40} implies the following theorem.

\begin{thm}
Let $V=L_\hk(1,0)$ be a unitary Heisenberg VOA. Then any irreducible unitary $V$-module is strongly integrable, the category $\fk F(\Rep^\uni(V))$ of all $\mc A_V$-modules which can be integrated from semisimple unitary $V$-modules is closed under Connes fusions, and the braided $C^*$-tensor categories $(\fk F(\Rep^\uni(V)),\boxtimes,\mathbb B)$ and $(\Rep^\uni(V),\boxtimes,\ss)$ are equivalent under the $*$-functor $\fk F$.
\end{thm}

We now prove the strong braiding property for intertwining operators of $V$. By the second half of theorem \ref{lb40}, it suffices to check the strong intertwining property.  Choose  $\lambda,\mu\in i\hr$ and disjoint $\wtd I,\wtd J\in\Jtd$. Then for any homogeneous  $w^{(\lambda)}\in W_\lambda$, $\wtd f\in C^\infty_c(\wtd I)$, and $g\in C^\infty_c(J),\alpha\in\hk$ satisfying that $Y(\alpha(-1)\Omega,g)$ is symmetric, lemma \ref{lb39} tells us that $\ovl{Y(\alpha(-1)\Omega,g)}$ is self-adjoint, and the smeared intertwining operator $\mc Y^{\lambda+\mu}_{\lambda,\mu}(w^{(\lambda)},\wtd f)$, when regarded as an unbounded operator on $\mc H_\lambda\oplus\mc H_\mu$ with domain $\mc H_\lambda^\infty\oplus\mc H_\mu^\infty$, commutes strongly with the preclosed operator $\diag(Y_\lambda(\alpha(-1)\Omega,g),Y_\mu(\alpha(-1)\Omega,g))=Y_{\lambda\oplus\mu}(\alpha(-1)\Omega,g)$ (see also \cite{Gui17a} proposition 3.16). By strong integrability, $\ovl{Y_{\lambda\oplus\mu}(\alpha(-1)\Omega,g)}=\pi_{\lambda\oplus\mu}(\ovl{Y(\alpha(-1)\Omega,g)})$. By \cite{CKLW18} theorem 8.1, $\mc A_V(J)$ is generated by all such $\ovl{Y(\alpha(-1)\Omega,g)}$. Therefore $\mc Y^{\lambda+\mu}_{\lambda,\mu}(w^{(\lambda)},\wtd f)$ commutes strongly with $\pi_{\lambda\oplus\mu}(y)$ for any $y\in\mc A_V(J)$. The strong intertwining property for $w^{(\lambda)}\curvearrowright\Rep^\uni(V)$ hence follows. By  theorem \ref{lb40}, we have the strong braiding property for any $w^{(\lambda)},w^{(\mu)}\in\Rep^\uni(V)$. Note that we can identify $\mc L_\lambda(w^{(\lambda)},x)$ with $\mc Y_\lambda(w^{(\lambda)},x)$, and identify $\mc R_\lambda(w^{(\mu)},x)$ with $e^{i\pi(\mu|\nu)}\mc Y_\mu(w^{(\mu)},x)$ when acting on any $W_\nu$. The strong braiding property can therefore be written in the following equivalent form:
\begin{thm}\label{lb47}
Let $V=L_\hk(1,0)$ be a unitary Heisenberg VOA. Then for any $\lambda,\mu,\nu\in i\hr$, any homogeneous vectors $w^{(\lambda)}\in W_\lambda,w^{(\mu)}\in W_\mu$, any intervals $\wtd I,\wtd J\in\Jtd$ with $\wtd I$ anticlockwise to $\wtd J$, and any $\wtd f\in C^\infty_c(\wtd I),\wtd g\in C^\infty_c(\wtd J)$, the following diagram of preclosed operators commutes strongly.
\begin{align}
\begin{CD}
\mc H_\nu @> \qquad \mc Y_\mu(w^{(\mu)},\wtd g) \qquad >> \mc H_{\nu+\mu}\\
@V \mc Y_\lambda(w^{(\lambda)},\wtd f) VV @VV  \mc Y_\lambda(w^{(\lambda)},\wtd f) V\\
\mc H_{\lambda+\nu} @>\quad e^{i\pi(\lambda|\mu)}\mc Y_\mu(w^{(\mu)},\wtd g) \quad>> \mc H_{\lambda+\nu+\mu}
\end{CD}\label{eq80}
\end{align}
\end{thm}

\subsection{Lattice VOAs}

In this section, a unitary Heisenberg VOA $L_\hk(1,0)$ is denoted by $U$, and the symbol $V$ will be reserved for a lattice VOA. Let $\Upsilon$ be an even lattice in $i\hr$ satisfying $\textrm{rank}(\Upsilon)=\dim(i\hr)$, and let $\Upsilon^\circ$ be the dual lattice of $\Upsilon$. Then the unitary $U$-module $V=\bigoplus_{\alpha\in\Upsilon}L_\hk(1,\alpha)$ can be extended to a unitary VOA structure by choosing a map $\epsilon:\Upsilon^\circ\times\Upsilon^\circ\rightarrow S^1$ satisfying
\begin{gather*}
\epsilon(\alpha,0)=1,\qquad  \epsilon(\alpha,\beta+\gamma)\epsilon(\beta,\gamma)=\epsilon(\alpha,\beta)\epsilon(\alpha+\beta,\gamma)\qquad (\forall \alpha,\beta,\gamma \in\Upsilon^\circ),\\
\epsilon(\alpha,\beta)=(-1)^{(\alpha|\beta)}\epsilon(\beta,\alpha)\qquad (\forall \alpha,\beta \in\Upsilon)
\end{gather*}
(see \cite{LL12} remark 6.4.12 for the existence of such $\epsilon$), and setting, for each $\alpha,\mu\in\Upsilon,w^{(\alpha)}\in L_\hk(1,\alpha),w^{(\mu)}\in L_\hk(1,\mu)$,
\begin{align}
Y(w^{(\alpha)},x)w^{(\mu)}=\epsilon(\alpha,\mu)\mc Y_\alpha(w^{(\alpha)},x)w^{(\mu)}.\label{eq77}
\end{align}
where $\mc Y_\alpha$ is  as in the last section. Then $(V,Y)$ becomes a VOA, called the lattice VOA for $\Upsilon$. By \cite{Miy04} proposition 2.7 or \cite{DL14} theorem 4.12, $V$ is unitary. As $\Upsilon\subset\Upsilon^\circ$, we have a quotient map $[\cdot]:\Upsilon^\circ\rightarrow\Upsilon^\circ/\Upsilon,\lambda\mapsto [\lambda]$. Then for each $\lambda\in\Upsilon^\circ$, the unitary $U$-module $W_{[\lambda]}=\bigoplus_{\mu\in \lambda+\Upsilon}L_\hk(1,\mu)$ can be extended to an irreducible unitary $V$-module $(W_{[\lambda]},Y_{[\lambda]})$ by letting $Y_{[\lambda]}(w^{(\alpha)},x)w^{(\mu)}$ equal the right hand side of \eqref{eq77} for any $\alpha\in\Upsilon,\mu\in\lambda+\Upsilon,w^{(\alpha)}\in L_\hk(1,\alpha),w^{(\mu)}\in L_\hk(1,\mu)$. Moreover, any irreducible $V$-module arises in this way (\cite{LL12} theorem 6.5.24). We thus have a bijection between $\Upsilon^\circ/\Upsilon$ and the equivalence classes of irreducible (unitary) $V$-modules.

Intertwining operators of $V$ can be described as follows (cf. \cite{DL93} proposition 12.2). For any $\lambda_0,\mu_0,\nu_0\in\Upsilon^\circ$, we let $\mc V{[\nu_0]\choose [\lambda_0]~[\mu_0]}$ be the vector space of type ${[\nu_0]\choose[\lambda_0]~[\mu_0]}={W_{[\nu_0]}\choose W_{[\lambda_0]} W_{[\mu_0]}}$ intertwining operators of $V$, and let $N^{[\nu_0]}_{[\lambda_0][\mu_0]}$ be the fusion rule $\dim\mc V{[\nu_0]\choose [\lambda_0]~[\mu_0]}$. Then $N^{[\nu_0]}_{[\lambda_0][\mu_0]}$ equals $1$ when $\nu_0-\lambda_0-\mu_0\in\Upsilon$, and equals $0$ otherwise. Therefore $W_{[\lambda_0]}\boxtimes W_{[\mu_0]}\simeq W_{[\lambda_0+\mu_0]}$. A distinguished type $[\lambda_0+\mu_0]\choose[\lambda_0]~[\mu_0]$ intertwining operator $\mc Y_{[\lambda_0][\mu_0]}^{[\lambda_0+\mu_0]}$, written simply as $\mc Y_{[\lambda_0]}$, can be chosen to satisfy that for any $\lambda\in\lambda_0+\Upsilon,\mu\in\mu_0+\Upsilon,w^{(\lambda)}\in L_\hk(1,\lambda),w^{(\mu)}\in L_\hk(1,\mu)$,
\begin{align}
\mc Y_{[\lambda_0]}(w^{(\lambda)},x)w^{(\mu)}=\frac {\epsilon(\lambda,\mu)\epsilon(\mu-\mu_0,\lambda) e^{i\pi(\mu-\mu_0|\lambda)}}{\epsilon(\lambda,\mu-\mu_0)}\cdot \mc Y_\lambda(w^{(\lambda)},x)w^{(\mu)}
\end{align}
Thus the energy-boundedness of $V$-intertwining operators follows from that of $U$-intertwining operators.

We now prove the strong braiding property of intertwining operators of $V$. First we need a lemma.

\begin{lm}\label{lb46}
Let $A,B$ be preclosed operators on a Hilbert space $\mc H$ with common invariant domain $\Dom$.  Let $\{p_\alpha \}$ be a collection of  projections on $\mc H$ satisfying $\bigvee_\alpha p_\alpha=\id_{\mc H}$. Assume that for any $\alpha$,  $p_\alpha\Dom\subset\Dom$, $p_\alpha$ commutes strongly with $A,B$, and the restrictions of $A,B$ to $p_\alpha\mc H$ (with common domain $p_\alpha\Dom$) commute strongly. Then $A$ and $B$ commutes strongly.
\end{lm}
Note that since $p_\alpha$ commutes strongly with $A$ and $A\Dom,p_\alpha\Dom\subset\Dom$, we have $Ap_\alpha\Dom= p_\alpha A\Dom\subset p_\alpha\Dom$, and similarly $Bp_\alpha\Dom\subset p_\alpha\Dom$. Therefore the restrictions  in this lemma make sense.

\begin{proof}
For each $n$ we set $\mc H_\alpha=p_\alpha\mc H,\Dom_\alpha=p_\alpha\Dom$, and let $A|_{\mc H_\alpha}$  be  the preclosed operator on $\mc H_\alpha$ with dense domain $\Dom_\alpha$ satisfying $A|_{\mc H_\alpha}\xi=A\xi$ for any $\xi\in\Dom_\alpha$. Then, using the strong commutativity of $p_\alpha$ and $A$, one easily checks that $\ovl{A|_{\mc H_\alpha}}$ is the restriction of $\ovl A p_\alpha$ to $\mc H_\alpha$. To put it simply, we have $\ovl{A|_{\mc H_\alpha}}=\ovl A p_\alpha$. Let $\ovl A=UH$ be the left polar decomposition of $\ovl A$ with $U$ the phase of $\ovl A$. Then by the uniqueness of polar decompositions, $\ovl{A|_{\mc H_\alpha}}$ also has polar decomposition $\ovl{A|_{\mc H_\alpha}}=Up_\alpha\cdot Hp_\alpha$. Define $B|_{\mc H_\alpha}$ in a similar way, and let $\ovl B=U'H'$ be the left polar decomposition of $B$. Then we also have the polar decomposition $\ovl{B|_{\mc H_\alpha}}=U'p_\alpha\cdot H'p_\alpha$ of $\ovl{B|_{\mc H_\alpha}}$. Now we choose $x$ (resp. $y$) to be an arbitrary element in  the von Neumann algebra generated by $\ovl A,\ovl A^*$ (resp. $\ovl B,\ovl B^*$). Then, since $\ovl{A|_{\mc H_\alpha}}$ commutes strongly with $\ovl{B|_{\mc H_\alpha}}$ by assumption, we see that $xp_\alpha$ commutes with $yp_\alpha$. As $[x,p_\alpha]=[y,p_\alpha]=0$, we have $xy\xi=yx\xi$ for any $\xi\in\mc H_\alpha$. Since $\alpha$ is arbitrary, we actually have the commutativity of $x$ and $y$, which therefore proves the strong commutativity of $\ovl A$ and $\ovl B$.
\end{proof}

\begin{thm}
	Let $V$ be the VOA for a non-degenerate even lattice $\Upsilon\subset i\hr$. Then for any $\lambda_0,\mu_0,\nu_0\in \Upsilon^\circ$, any homogeneous vectors $w^{[\lambda_0]}\in W_{[\lambda_0]},w^{[\mu_0]}\in W_{[\mu_0]}$, any intervals $\wtd I,\wtd J\in\Jtd$ with $\wtd I$ anticlockwise to $\wtd J$, and any $\wtd f\in C^\infty_c(\wtd I),\wtd g\in C^\infty_c(\wtd J)$, the following diagram of preclosed operators commutes strongly.
\begin{align}
	\begin{CD}
	\mc H_{[\nu_0]} @> \quad \mc R_{[\mu_0]}(w^{[\mu_0]},\wtd g) \quad >> \mc H_{[\nu_0][\mu_0]}\\
	@V \mc L_{[\lambda_0]}(w^{[\lambda_0]},\wtd f) VV @VV  \mc L_{[\lambda_0]}(w^{[\lambda_0]},\wtd f) V\\
	\mc H_{[\lambda_0][\nu_0]} @>\quad\mc R_{[\mu_0]}(w^{[\mu_0]},\wtd g) \quad>> \mc H_{[\lambda_0][\nu_0][\mu_0]}
	\end{CD}
\end{align}
\end{thm}

\begin{proof}
We first prove the special case where there exist $\lambda\in\lambda_0+\Upsilon$ and $\mu\in\mu_0+\Upsilon$ such that $w^{[\lambda_0]}\in W_\lambda=L_\hk(1,\lambda),w^{[\mu_0]}\in W_\mu= L_\hk(1,\mu)$. Write $w^{(\lambda)}=w^{[\lambda_0]},w^{(\mu)}=w^{[\mu_0]}$. Identify $\mc H_{[\lambda_0][\nu_0]},\mc H_{[\nu_0][\mu_0]},\mc H_{[\lambda_0][\nu_0][\mu_0]}$ with $\mc H_{[\lambda_0+\nu_0]},\mc H_{[\nu_0+\mu_0]},\mc H_{[\lambda_0+\nu_0+\mu_0]}$ respectively. (There is no need to choose canonical identifications.) Set $\mc H=\mc H_{[\nu_0]}\oplus\mc H_{[\lambda_0+\nu_0]}\oplus\mc H_{[\nu_0+\mu_0]}\oplus\mc H_{[\lambda_0+\nu_0+\mu_0]}$ and extend $\mc L_{[\lambda_0]}(w^{[\lambda_0]},\wtd f)$ (resp. $\mc R_{[\mu_0]}(w^{[\mu_0]},\wtd g)$) to a preclosed operator $A$ (resp. $B$) on $\mc H$ with domain $\mc H^\infty=\mc H^\infty_{[\nu_0]}\oplus\mc H^\infty_{[\lambda_0+\nu_0]}\oplus\mc H^\infty_{[\nu_0+\mu_0]}\oplus\mc H^\infty_{[\lambda_0+\nu_0+\mu_0]}$ as in definition \ref{lb37}. 

Notice that, for example, $\mc H_{[\lambda_0+\nu_0]}=\bigoplus_{\nu\in\nu_0+\Upsilon}\mc H_{\lambda+\nu}$, where we recall that $\mc H_{\lambda+\nu}$ is the $\mc A_U$-module integrated from the $U$-module $W_{\lambda+\nu}=L_\hk(1,\lambda+\nu)$. Therefore, for each $\nu\in\nu_0+\Upsilon$, we have a projection $p_\nu$ of $\mc H$ onto the subspace $\mc K_\nu=\mc H_{\nu}\oplus\mc H_{\lambda+\nu}\oplus\mc H_{\nu+\mu}\oplus\mc H_{\lambda+\nu+\mu}$ of $\mc H$. Then its smooth subspace $\mc K_\nu^\infty$ satisfies $\mc K_\nu^\infty=p_\nu\mc H^\infty$. Moreover, it is easy to see that $p_\nu$ commutes strongly with $A$ and $B$. Thus, by lemma \ref{lb46}, it suffices to verify the strong commutativity of $A$ and $B$ when restricted to each $\mc K_\nu$. But by our knowledge of the fusion rules of $U$, it is clear that the strong commutativity of the preclosed operators $A|_{\mc K_\nu}$ and $B|_{\mc K_\nu}$ (with common invariant domain  $\mc K^\infty_\nu$) is equivalent to that of diagram \eqref{eq80}, which is already proved by theorem \ref{lb47}. Thus this special case is proved.

Now, in the  general case, a homogeneous vector $w^{[\lambda_0]}\in W_{[\lambda_0]}$ (resp. $w^{[\mu_0]}\in W_{[\mu_0]}$) can be written as a finite sum of homogeneous vectors of the form $w^{(\lambda)}\in W_\lambda$ (where $\lambda\in\lambda_0+\Upsilon$) (resp. $w^{(\mu)}\in W_\mu$ (where $\mu\in\mu_0+\Upsilon$) ). Thus the strong braiding property follows from rotation covariance \eqref{eq78} and lemma \ref{lb44}.
\end{proof}

We note that when one of $\lambda_0,\mu_0$ is $0$, the above theorem says nothing but the strong intertwining property for the intertwining operators of $V$. When both $\lambda_0,\mu_0$ are $0$, this theorem says that $V$ is strongly local. If we combine  this theorem with the results in section \ref{lb48}, we immediately have the following theorem:

\begin{thm}
Let $V$ be a (unitary) even lattice VOA. Then $V$ is strongly local, and any unitary $V$-module is strongly integrable. The sesquilinear form $\Lambda$ defined on each vector space of intertwining operators of $V$ is positive(-definite). Hence $\Rep^\uni(V)$ is a unitary modular tensor category. Let $\fk F(\Rep^\uni(V))$ be the category of all $\mc A_V$-modules integrated from objects in $\Rep^\uni(V)$. Then $\fk F(\Rep^\uni(V))$ is closed under Connes fusions, and  $(\fk F(\Rep^\uni(V)),\boxtimes,\mathbb B)$ and $(\Rep^\uni(V),\boxtimes,\ss)$ are equivalent as unitary modular tensor categories under the $*$-functor $\fk F$.
\end{thm}

Hence, once we know that all semisimple $\mc A_V$-modules arise from integrating unitary $V$-modules, we have the equivalence of unitary modular tensor categories $(\Rep^\ssp(\mc A_V),\boxtimes,\mathbb B)\simeq(\Rep^\uni(V),\boxtimes,\ss)$.

\section{Relation to DHR superselection theory}

In this chapter, we show that the representation category $\Rep(\mc A)$ of a conformal net $\mc A$ is equivalent to the braided $C^*$-tensor category $\DHR(\mc A)$ of DHR (Doplicher-Haag-Roberts) endomorphisms of $\mc A$ localized in an arbitrary open interval $I_0\in\mc J$. We first review the DHR theory for conformal nets developed in \cite{DHR71,DHR74,FRS89,FRS92}.

First, we define a universal $C^*$-algebra $C^*(\mc A)$  following \cite{Fre90}. Let $C_0(\mc A)$ be the free $*$-algebra generated by all $\mc A(I)$ ($I\in\mc J$). Then any $\mc A$-module $(H_i,\pi_i)\in\Rep(\mc A)$ can be naturally extended to a $C_0(\mc A)$-module, also denoted by $\pi_i$. Define a $C^*$-seminorm $\lVert\cdot\lVert$ on $C_0(\mc A)$ satisfying $\lVert A\lVert=\sup_{\mc H_i\in\Rep(\mc A)}\lVert\pi_i(A)\lVert$ for any $A\in C_0(\mc A)$, and let $C^*(\mc A)$ be the completion of $C_0(\mc A)$ under this norm. Then any representation $\mc H_i$ of $\mc A$ can be extended uniquely to a representation of $C^*(\mc A)$ on $\mc H_i$.

\subsubsection*{DHR endomorphisms}

By an  endomorphism $\rho$ of $C^*(\mc A)$, we always mean that $\rho$ is a continuous unital $*$-endomorphism.  In the following, we fix an open interval $I_0\in\mc J$. We say that an endomorphism $\rho$ is \textbf{localized} in $I_0$, if the restriction of $\rho$ to $\mc A(I^c_0)$ is the identity embedding $\id:\mc A(I^c_0)\hookrightarrow C^*(\mc A)$. If, moreover, for any $I_1,I\in\mc J$ satisfying $I_0\cup I_1\subset I$, there exists a unitary $U\in \mc A(I)$ such that $\mathrm{Ad}(U)\circ\rho$ is localized in $I_1$, we say that $\rho$ is \textbf{transportable}. The category of transportable endomorphisms localized in $I_0$ is denoted by $\DHR(\mc A)$.  Each $\rho\in\DHR(\mc A)$ is associated with a canonical (locally normal) representation $(\mc H_\rho,\pi_\rho)$ of $\mc A$, which satisfies $\mc H_\rho=\mc H_0$ (as Hilbert spaces) and  $\pi_\rho(x)=\pi_0(\rho(x))$ for any $I\in\mc J,x\in\mc A(I)$. 

For any $\rho_1,\rho_2\in\DHR(\mc A)$, we define the Hom space
\begin{align*}
\Hom(\rho_1,\rho_2)=\{T\in \mc A(I_0):T\rho_1(A)=\rho_2(A)T \quad(\forall A\in C^*(\mc A)) \}.
\end{align*} 
Then $\pi_0(T)\in\Hom_{\mc A}(\mc H_{\rho_1},\mc H_{\rho_2})$. Conversely, by Haag duality and the fact that $\rho_1,\rho_2$ are localized in $I_0$, any element in $\Hom_{\mc A}(\mc H_{\rho_1},\mc H_{\rho_2})$ arises in this way. We therefore have a natural identification $\Hom(\rho_1,\rho_2)\simeq \Hom_{\mc A}(\mc H_{\rho_1},\mc H_{\rho_2})$. 

The tensor (fusion) product $\boxtimes$ of any $\rho_1,\rho_2\in\DHR(\mc A)$ is defined to be the composition of the two endomorphisms $\rho_1\boxtimes\rho_2=\rho_2\circ\rho_1=\rho_2\rho_1$. If $R\in\Hom(\rho_1,\rho_3),S\in\Hom(\rho_2,\rho_4)$, then one can easily verify that $S\rho_2(R)\in\Hom(\rho_2\rho_1,\rho_4\rho_3)$. We therefore set the tensor product of $R$ and $S$ to be $R\otimes S=\rho_4(R)S=S\rho_2(R)$. We set the identity object of $\End(C^*(\mc A))$ to be the identity endomorphism of $C^*(\mc A)$. Associativity isomorphisms are defined in the natural way. Then $\DHR(\mc A)$ becomes a $C^*$-tensor category. The braid operator $\varepsilon(\rho_1,\rho_2)\in\Hom(\rho_2\rho_1,\rho_1\rho_2)$ is defined by choosing disjoint open intervals $I_1,I_2\subset I_0$ such that $I_2$ is anticlockwise to $I_1$ in $I_0$, choosing $U_1,U_2\in\mc A(I_0)$ such that $\Ad(U_1)\circ \rho_1$ and $\Ad(U_2)\circ\rho_2$ are localized in $I_1$ and $I_2$ respectively, and defining the \textbf{statistic operator}
\begin{align}
\varepsilon(\rho_1,\rho_2)=\rho_1(U_2^*)U_1^*U_2\rho_2(U_1).
\end{align}
This operator is independent of the particular choice of $U_1,U_2,I_1,I_2$. Using $\varepsilon$ to define braiding, one has a $C^*$-braided tensor category $(\DHR(\mc A),\boxtimes,\varepsilon)$.

\subsubsection*{The $*$-functor $\fk G:\Repi(\mc A)\rightarrow\DHR(\mc A)$}

To show the equivalence of $\Rep(\mc A)$ and $\DHR(\mc A)$, it will be more convenient to consider a slightly different tensor category $\Repi(\mc A)$ equivalent to $\Rep(\mc A)$. Let $L$ and $R$ denote the left and the actions in the Connes categorical extension of $\mc A$. For any $\wtd I\in\Jtd$, we say that a vector $\xi\in\mc H_i(I)$ is \textbf{unitary}, if the map $L(\xi,\wtd I)=Z(\xi,I):\mc H_0\rightarrow\mc H_i$ is unitary.   Existence of a unitary vector in $\mc H_i(I)$ follows from the fact that $\mc A(I^c)$ is a type III factor. We let $\mc U_i(I)$ denote the set of all unitary vectors in $\mc H_i(I)$.

\begin{lm}\label{lb49}
For any $\xi\in\mc U_i(I),\mc H_j\in \Rep(\mc A)$, the map $L(\xi,\wtd I):\mc H_j\rightarrow\mc H_i\boxtimes\mc H_j$ is unitary.
\end{lm}

\begin{proof}
It is easy to see that the action of $L(\xi,\wtd I)^*L(\xi,\wtd I)$ on $\mc H_j$ equals $\pi_j(L(\xi,\wtd I)^*L(\xi,\wtd I)|_{\mc H_0})=\id_j$. Therefore $L(\xi,\wtd I)^*L(\xi,\wtd I)|_{\mc H_j}$ is an isometry. Now choose any $\wtd J\in\Jtd$ clockwise to $\wtd I$. Then vectors of the form  $L(\xi,\wtd I)R(\eta,\wtd J)\chi^{(0)}=R(\eta,\wtd J)L(\xi,\wtd J)\chi^{(0)}$ (where $\eta\in\mc H_j(J),\chi^{(0)}\in\mc H_0$) span a dense subspace of $\mc H_i\boxtimes \mc H_j$. Thus $L(\xi,\wtd I)$ is unitary when acting on $\mc H_j$.
\end{proof}

Now we fix an arg function $\arg_{I_0}$ of $I_0$, and let $\wtd I_0=(I_0,\arg_{I_0})$. Define a new category $\Repi(\mc A)$ whose objects are $(\mc H_i,\xi)$ where $\mc H_i\in\Rep(\mc A),\xi\in\mc U_i(I_0)$. If $(\mc H_i,\xi),(\mc H_j,\eta)\in\Repi(\mc A)$, we let the Hom space be $\Hom((\mc H_i,\xi),(\mc H_j,\eta))=\Hom_{\mc A}(\mc H_i,\mc H_j)$. We define a tensor (fusion) bifunctor $\boxtimes$, such that
\begin{align*}
(\mc H_i,\xi)\boxtimes(\mc H_j,\eta)=(\mc H_i\boxtimes\mc H_j,L(\xi,\wtd I_0)\eta),
\end{align*}
where we notice that $L(\xi,\wtd I_0)\eta\in\mc U_{i\boxtimes j}(I_0)$ by lemma \ref{lb49}. Tensor products of morphisms, and all the structural isomorphisms (associativity, braiding, etc.) are defined using those of $\Rep(\mc A)$, disregarding all the unitary vectors. The identity object is chosen to be $(\mc H_0,\Omega)$. Then $\Repi(\mc A)$ is clearly a braided $C^*$-tensor category equivalent to $\Rep(\mc A)$.

We now define a $*$-functor $\fk G:\Repi(\mc A)\rightarrow \DHR(\mc A)$. Choose any $(\mc H_i,\xi)\in\Repi(\mc A)$. An endomorphism $\rho_i=\fk G(\mc H_i,\xi)$ can be defined as follows (cf.\cite{Fre90}). Choose any $I\in\mc J$, and choose $I_1\subset I^c$ such that $I_1\cup I_0$ can be covered by an open interval $J$. We choose arg functions of $I_1$ and $J$ such that $\wtd I_1,\wtd I_0\subset\wtd J$. (In fact the arg functions are irrelavent here since we will only deal with left actions on the vacuum module.) Choose an arbitrary $\xi_1\in\mc U_i(I_1)$. Then the action of $L(\xi_1,\wtd I_1)^*L(\xi,\wtd I_0)$ on $\mc H_0$ lies inside $\End_{\mc A(J^c)}(\mc H_0)=\mc A(J^c)'=\mc A(J)$. Regard $L(\xi_1,\wtd I_1)^*L(\xi,\wtd I_0)$ as an element in $\mc A(J)$ and write it as $U(\xi_1,\xi)$, we thus define
\begin{gather*}
\rho_i:\mc A(I)\rightarrow C^*(\mc A),\\
x\mapsto U(\xi_1,\xi)^*\cdot x\cdot U(\xi_1,\xi).
\end{gather*}
Such $\rho_i$ is independent of the particular choice of $I_1$ and $\xi_1$, and can be extended to a transportable endomorphism of $C^*(\mc A)$ localized in $I_0$. 

In the case that $I\cup I_0$ is not dense in $S^1$, we can choose an open interval $K\in\mc J$ covering $I\cup I_0$, and it is not hard to show that for any $x\in\mc A(I)$,
\begin{align}
\rho_i(x)=L(\xi,\wtd I_0)^*\pi_i(x)L(\xi,\wtd I_0),\label{eq81}
\end{align}
where $L(\xi,\wtd I_0)$ is acting on $\mc H_0$. This formula and the Haag duality $\End_{\mc A(K^c)}(\mc H_0)=\mc A(K^c)'=\mc A(K)$ implies $\rho_i(\mc A(I))\subset\mc A(K)$. In particular, $\rho_i(\mc A(I_0))\subset\mc A(I_0)$. We also notice that $\rho_i$ is determined by its values on  $\mc A(I)$ for all small $I$, since this is true for $(\mc H_{\rho_i},\pi_{\rho_i})$. Thus we can always use relation \eqref{eq81} to characterize $\rho_i$.

Now if $(\mc H_i,\xi),(\mc H_{i'},\xi')\in\Repi(\mc A)$ and $F\in\Hom_{\mc A}(\mc H_i,\mc H_{i'})=\Hom((\mc H_i,\xi),(\mc H_{i'},\xi'))$, we define
\begin{align*}
\fk G(F)=L(\xi',\wtd I_0)^*\cdot F\cdot L(\xi,\wtd I_0)
\end{align*}
with $L(\xi',\wtd I_0)$ and $L(\xi,\wtd I_0)$ acting on $\mc H_0$. That $\fk G(F)\in\mc A(I_0)$ follows from Haag duality. Write $\rho_i=\fk G(\mc H_i,\xi),\rho_{i'}=\fk G(\mc H_{i'},\xi')$. Then using \eqref{eq81}, one can easily verify $\fk G(F)\rho_i(x)=\rho_{i'}(x)\fk G(F)$ for any $x\in\mc A(I)$ where $I\in \mc J$ is small enough such that $I\cup I_0$ is not dense. Therefore $\fk G(F)\in\Hom(\rho_i,\rho_{i'})$. Thus we've defined the functor $\fk G$. It is obvious that $\fk G$ is fully faithful and $*$-preserving.

\subsubsection*{Equivalence of the braided $C^*$-tensor categories}

We now show that $\fk G:\Repi(\mc A)\rightarrow\DHR(\mc A)$ is an equivalence  of  braided $C^*$-tensor categories. That $\fk G$ preserves the monoidal structures is verified by the following propositions.

\begin{pp}
Choose any $(\mc H_i,\xi),(\mc H_j,\eta)\in\Repi(\mc A)$. Then $\fk G((\mc H_i,\xi)\boxtimes(\mc H_j,\eta))=\fk G(\mc H_i,\xi)\boxtimes\fk G(\mc H_j,\eta)$.
\end{pp}
\begin{proof}
Choose any $I\in\mc J$ such that $I\cup I_0$ is non-dense. Then we choose $\wtd J\in\Jtd$ clockwise to $\wtd I_0$ and disjoint from $I$. Write $\rho_i=\fk G(\mc H_i,\xi),\rho_j=\fk G(\mc H_j,\eta)$. Then $\rho_i\boxtimes\rho_j=\rho_j\rho_i$. On the other hand, $(\mc H_i,\xi)\boxtimes(\mc H_j,\eta)=(\mc H_i\boxtimes\mc H_j, L(\xi,\wtd I_0)\eta)$. We let $\rho=\fk G(\mc H_i\boxtimes\mc H_j, L(\xi,\wtd I_0)\eta)$. We want to show $\rho=\rho_j\rho_i$.

Choose any $x\in\mc A(I)$. Then by proposition \ref{lb13},
\begin{align}
&\rho(x)=L(L(\xi,\wtd I_0)\eta,\wtd I_0)^*\pi_{i\boxtimes j}(x)L(L(\xi,\wtd I_0)\eta,\wtd I_0)\nonumber\\
=&L(\eta,\wtd I_0)^*L(\xi,\wtd I_0)^*\pi_{i\boxtimes j}(x)L(\xi,\wtd I_0)L(\eta,\wtd I_0).\label{eq82}
\end{align}
Now choose $\xi_1\in\mc U_i(J),\eta_1\in\mc U_j(J)$. Then we have
\begin{gather*}
\pi_i(x)=R(\xi_1,\wtd J)xR(\xi_1,\wtd J)^*,\\
\pi_j(x)=R(\eta_1,\wtd J)x R(\eta_1,\wtd J)^*,\\
\pi_{i\boxtimes j}(x)=R(\eta_1,\wtd J)R(\xi_1,\wtd J)x R(\xi_1,\wtd J)^*R(\eta_1,\wtd J)^*.
\end{gather*}
Using these relations and \eqref{eq82}, and apply locality (condition (f) of definition \ref{lb9}), we have
\begin{align*}
&\rho(x)=L(\eta,\wtd I_0)^*L(\xi,\wtd I_0)^*R(\eta_1,\wtd J)R(\xi_1,\wtd J)\cdot x\cdot R(\xi_1,\wtd J)^*R(\eta_1,\wtd J)^*L(\xi,\wtd I_0)L(\eta,\wtd I_0)\\
=&L(\eta,\wtd I_0)^*R(\eta_1,\wtd J)L(\xi,\wtd I_0)^*R(\xi_1,\wtd J)\cdot x\cdot R(\xi_1,\wtd J)^*L(\xi,\wtd I_0)R(\eta_1,\wtd J)^*L(\eta,\wtd I_0)\\
=&L(\eta,\wtd I_0)^*R(\eta_1,\wtd J)L(\xi,\wtd I_0)^*\pi_i(x)L(\xi,\wtd I_0)R(\eta_1,\wtd J)^*L(\eta,\wtd I_0)\\
=&L(\eta,\wtd I_0)^*R(\eta_1,\wtd J)\rho_i(x)R(\eta_1,\wtd J)^*L(\eta,\wtd I_0)=L(\eta,\wtd I_0)^*\pi_j(\rho_i(x))L(\eta,\wtd I_0)=\rho_j(\rho_i(x)).
\end{align*}
\end{proof}

\begin{pp}\label{lb61}
If $F\in\Hom((\mc H_i,\xi),(\mc H_{i'},\xi')),G\in\Hom((\mc H_j,\eta),(\mc H_{j'},\eta'))$, then $\fk G(F\otimes G)=\fk G(F)\otimes\fk G(G)$.
\end{pp}

\begin{proof}
Write
\begin{gather*}
R=\fk G(F)=L(\xi',\wtd I_0)^*\cdot F\cdot L(\xi,\wtd I_0),\\
S=\fk G(G)=L(\eta',\wtd I_0)^*\cdot G\cdot L(\eta,\wtd I_0).
\end{gather*}
Then
\begin{align}
&R\otimes S=S\rho_j(R)=L(\eta',\wtd I_0)^*\cdot G\cdot L(\eta,\wtd I_0)\cdot\rho_j(L(\xi',\wtd I_0)^*\cdot F\cdot L(\xi,\wtd I_0))\nonumber\\
=&L(\eta',\wtd I_0)^*\cdot G\cdot L(\eta,\wtd I_0)\cdot L(\eta,\wtd I_0)^*\pi_j(L(\xi',\wtd I_0)^*\cdot F\cdot L(\xi,\wtd I_0))L(\eta,\wtd I_0)\nonumber\\
=&L(\eta',\wtd I_0)^*\cdot G\cdot \pi_j(L(\xi',\wtd I_0)^*\cdot F\cdot L(\xi,\wtd I_0))L(\eta,\wtd I_0)\label{eq83}
\end{align}
Now choose any $\wtd J\in\Jtd$ clockwise to $\wtd I_0$, and any $\eta_1\in\mc U_j(J)$. Then using locality and the functoriality of the right  actions (condition (b) of definition \ref{lb9}),
\begin{align}
&\pi_j(L(\xi',\wtd I_0)^*\cdot F\cdot L(\xi,\wtd I_0))=R(\eta_1,\wtd J)L(\xi',\wtd I_0)^*\cdot F\cdot L(\xi,\wtd I_0)R(\eta_1,\wtd J)^*\nonumber\\
=&L(\xi',\wtd I_0)^*R(\eta_1,\wtd J)\cdot F\cdot R(\eta_1,\wtd J)^*L(\xi,\wtd I_0)\nonumber\\
=&L(\xi',\wtd I_0)^*(F\otimes\id_j)R(\eta_1,\wtd J)R(\eta_1,\wtd J)^*L(\xi,\wtd I_0)\nonumber\\
=&L(\xi',\wtd I_0)^*(F\otimes\id_j)L(\xi,\wtd I_0)\label{eq84}
\end{align}
when acting on $\mc H_j$. Substitute this result into the right hand side of \eqref{eq83}, and apply proposition \ref{lb13} and the functoriality of the left actions, we get
\begin{align*}
&R\otimes S=L(\eta',\wtd I_0)^*\cdot G\cdot L(\xi',\wtd I_0)^*(F\otimes\id_j)L(\xi,\wtd I_0)L(\eta,\wtd I_0)\nonumber\\
=&L(\eta',\wtd I_0)^* L(\xi',\wtd I_0)^*(\id_i\otimes G)(F\otimes\id_j)L(\xi,\wtd I_0)L(\eta,\wtd I_0)\\
=&L(L(\xi',\wtd I_0)\eta',\wtd I_0)^*(F\otimes G)L(L(\xi,\wtd I_0)\eta,\wtd I_0),
\end{align*}
which clearly equals $\fk G(F\otimes G)$.
\end{proof}

\begin{pp}
For any $(\mc H_i,\xi)\in\Repi(\mc A)$, the isomorphisms $\sharp_i:(\mc H_i,\xi)\boxtimes(\mc H_0,\Omega)\rightarrow(\mc H_i,\xi)$ and $\flat_i:(\mc H_0,\Omega)\boxtimes(\mc H_i,\xi)\rightarrow(\mc H_i,\xi)$ satisfy $\fk G(\sharp_i)=1=\fk G(\flat_i)$.
\end{pp}

\begin{proof}
Under the identifications $(\mc H_i,\xi)\boxtimes(\mc H_0,\Omega)=(\mc H_i\boxtimes\mc H_0,L(\xi,\wtd I_0)\Omega)=(\mc H_i,\xi)$ and $(\mc H_0,\Omega)\boxtimes(\mc H_i,\xi)=(\mc H_0\boxtimes\mc H_i,L(\Omega,\wtd I_0)\xi)=(\mc H_i,\xi)$, both $\sharp_i$ and $\flat_i$ are  $\id_i$. Thus their images under $\fk G$ are $1$.
\end{proof}

Finally, we check that $\fk G$ preserves the braid structures.

\begin{pp}\label{lb62}
We have $\fk G(\mathbb B)=\varepsilon$. More precisely, for any $(\mc H_i,\xi),(\mc H_j,\eta)\in\Repi(\mc A)$, if we let $\rho_i=\fk G(\mc H_i,\xi),\rho_j=\fk G(\mc H_j,\eta)$, then $\fk G(\mathbb B_{i,j})=\varepsilon(\rho_i,\rho_j)$.
\end{pp}
\begin{proof}

By the fact that $\mbb B_{i,j}$ and $\varepsilon(\rho_i,\rho_j)$ intertwine the tensor products of morphisms, and by proposition \ref{lb61}, to verify	$\fk G(\mathbb B_{i,j})=\varepsilon(\rho_i,\rho_j)$, it suffices to replace $(\mc H_i,\xi),(\mc H_j,\eta)$ with some unitarily equivalent objects. Therefore, we may assume that $\xi\in\mc U_i(I_1),\eta\in\mc U_j(I_2)$ where $\wtd I_1,\wtd I_2\subset\wtd I_0$ and $\wtd I_2$ is anticlockwise to $\wtd I_1$. Then $\rho_i$ and $\rho_j$ are localized in $I_1,I_2$ respectively. It follows that $\rho_i\circ\rho_j=\rho_j\circ\rho_i$ and $\varepsilon(\rho_i,\rho_j)=\id$.

On the other hand, recall that $(\mc H_i,\xi)\boxtimes(\mc H_j,\eta)=(\mc H_i\boxtimes\mc H_j,L(\xi,\wtd I_1)\eta)$ and $(\mc H_j,\eta)\boxtimes(\mc H_i,\xi)=(\mc H_j\boxtimes\mc H_i,L(\eta,\wtd I_2)\xi)$. Then, by proposition \ref{lb13},
\begin{align*}
\fk G(\mbb B_{i,j})=L(L(\eta,\wtd I_2)\xi,\wtd I_0)^*\mbb B_{i,j}L(L(\xi,\wtd I_1)\eta,\wtd I_0)=L(\xi,\wtd I_1)^*L(\eta,\wtd I_2)^*\mbb B_{i,j}L(\xi,\wtd I_1)L(\eta,\wtd I_2),
\end{align*}
which, by proposition \ref{lb17}, equals
\begin{align*}
L(\xi,\wtd I_1)^*L(\eta,\wtd I_2)^*L(\eta,\wtd I_2)L(\xi,\wtd I_1)=\id.
\end{align*}	
This proves $\fk G(\mathbb B_{i,j})=\varepsilon(\rho_i,\rho_j)$.
\end{proof}

Combine all these propositions together, we arrive at the following conclusion.
\begin{thm}
The braided $C^*$-tensor categories $(\Repi(\mc A),\boxtimes,\mathbb B)$ and $(\DHR(\mc A),\boxtimes,\varepsilon)$ are equivalent under the $*$-functor $\fk G$. Moreover, the functorial (i.e. natural) unitary isomorphism $\fk G(\mc H_i,\xi)\boxtimes\fk G(\mc H_j,\eta)\rightarrow\fk G((\mc H_i,\xi)\boxtimes(\mc H_j,\eta))$ realizing this equivalence is the identity operator.
\end{thm}

\begin{rem}
We sketch another way of proving  the equivalence of $\DHR(\mc A)$ and $\Rep(\mc A)$ as follows. One can define  a fully-faithful essentially-surjective $*$-functor $\fk E:\DHR(\mc A)\rightarrow \Rep(\mc A)$ such that for each object $\rho$ of $\DHR(\mc A)$, $\fk E(\rho)$ is the representation $(\mc H_\rho,\pi_\rho)$ mentioned at the beginning of this chapter: $\mc H_\rho=\mc H_0$ and $\pi_\rho=\pi_0\circ\rho$. A morphism $F\in\Hom(\rho_1,\rho_2)$ can be regarded as a morphism between representations. We let $\fk E(F)=F$. Note that $\fk E(\rho_1)\boxtimes\fk E(\rho_2)=\mc H_{\rho_1}\boxtimes\mc H_{\rho_2}$ is not identical to $\fk E(\rho_1\boxtimes\rho_2)=\mc H_{\rho_2\circ\rho_1}$. However, there is a well-known unitary isomorphism between these two $\mc A$-modules (cf. \cite{Con94} section 5.B), which, in our context, is defined by
\begin{gather*}
\Psi_{\rho_1,\rho_2}:\mc H_{\rho_1}\boxtimes\mc H_{\rho_2}\rightarrow\mc H_{\rho_2\circ\rho_1},\\
L\big(\pi_0(x)\Omega,\wtd I_0\big)\pi_0(y)\Omega\mapsto \pi_0\big(\rho_2(x)y\big)\Omega\qquad(\forall x,y\in\mc A(I_0)).
\end{gather*}
It is not hard to check that the $\Psi$ defined for each $\rho_1,\rho_2$ is a functorial map preserving the monoidal and braid structures of the two categories as in theorem \ref{lb12}. (See the end of \cite{HPT16} section 2.1 for the precise definition of the equivalence of two braided ($C^*$-)tensor categories.) In particular, as in the proof of proposition \ref{lb62}, to check that $\Psi$ preserves the braidings, it suffices to consider the case that $\rho_1$ and $\rho_2$ are localized in $I_1,I_2\subset I_0$ respectively where $I_2$ is anticlockwise to $I_1$.
\end{rem}

\noindent {\small \sc Department of Mathematics, Rutgers University, USA.}

\noindent {\em E-mail}: bin.gui@rutgers.edu\qquad binguimath@gmail.com

\end{document}